\documentclass[10pt,draft, reqno]{amsart}

\usepackage[margin=18mm]{geometry}
\usepackage{amssymb, amsmath, amsthm, amscd}

\usepackage[T1]{fontenc}
\usepackage{eucal,mathrsfs,dsfont}%gives nicer set-names. Use \mathds instead of \mathds
\usepackage{color}%cyan,red;magenta,green;yellow,blue

\usepackage{soul}

\renewcommand{\leq}{\leqslant}
\renewcommand{\geq}{\geqslant}
\renewcommand{\le}{\leqslant}
\renewcommand{\ge}{\geqslant}

\def\eps{\varepsilon}

\definecolor{mno}{rgb}{0.5,0.1,0.5}

\newcommand{\R}{\mathds R}

\newcommand{\M}{\mathscr{M}}

\newcommand{\e}{\varepsilon}
\newcommand{\T}{\mathds T}

\newcommand{\Pp}{\mathds P}
\newcommand{\Ee}{\mathds E}

\newcommand{\I}{\mathds 1}

\newcommand{\Z}{\mathds Z}
\newcommand{\D}{\mathscr{D}}

\newcommand{\sL}{\mathcal{L}}
\newcommand{\sE}{\mathcal{E}}
\newcommand{\sF}{\mathcal{F}}
\newcommand{\sN}{\mathcal{N}}

\newtheorem{theorem}{Theorem}[section]
\newtheorem{lemma}[theorem]{Lemma}
\newtheorem{proposition}[theorem]{Proposition}

\theoremstyle{definition}

\newtheorem{example}[theorem]{Example}
\newtheorem{remark}[theorem]{Remark}
\newtheorem{assumption}[theorem]{Assumption}

\numberwithin{equation}{section}

\begin{document}
\allowdisplaybreaks

\title[High-order convergence rates for symmetric L\'evy type operators]
{\bfseries High-order convergence rates of periodic homogenization for symmetric L\'evy type operators}

\author{Xin Chen, \quad Zhen-Qing Chen,\quad Takashi Kumagai\quad \hbox{and}\quad Jian Wang}
\thanks{\emph{X.\ Chen:}
   School of Mathematical Sciences, Shanghai Jiao Tong University, 200240 Shanghai, P.R. China. \texttt{chenxin217@sjtu.edu.cn}}
   \thanks{\emph{Z.-Q.\ Chen:}
   Department of Mathematics, University of Washington, Seattle, WA 98195, USA.
 \texttt{zqchen@uw.edu}}
   \thanks{\emph{T.\ Kumagai:}
Department of Mathematics,
Waseda University, Tokyo 169-8555, Japan.
\texttt{t-kumagai@waseda.jp}}
  \thanks{\emph{J.\ Wang:}
    School of Mathematics and Statistics \& Key Laboratory of Analytical Mathematics and Applications (Ministry of Education) \& Fujian Provincial Key Laboratory
of Statistics and Artificial Intelligence, Fujian Normal University, 350007 Fuzhou, P.R. China. \texttt{jianwang@fjnu.edu.cn}}

\begin{abstract}  In this paper, we establish higher-order convergence rates
of the periodic homogenization
for symmetric L\'evy-type operators, encompassing the subcritical
$\alpha$-stable regime, critical regime, and supercritical diffusive regime.
To this end, we develop a systematic framework to decompose the contributions of the underlying jumping kernel across small, intermediate, and large spatial scales -- a strategy tailored to all the aforementioned regimes. To the best of our knowledge, this work represents the first comprehensive study of higher-order convergence rates in the homogenization of non-local operators.

\medskip


\medskip

\end{abstract}

\maketitle

\section{Introduction}\label{section1}

\subsection{Setting and background}
In this paper, we consider the following symmetric L\'evy-type operator with periodic coefficients:
\begin{equation}\label{e1-1}
\begin{split}
\sL f(x)&={\rm p.v.}\int_{\R^d}\left(f(y)-f(x)\right)K(x,y)
j(x-y)
\,dy\\
&=\lim_{\eta \to 0}\int_{\{y\in \R^d: |y-x|>\eta\}}\left(f(y)-f(x)\right)K(x,y)
j(x-y)\,dy,\quad f\in C_b^2(\R^d),
\end{split}
\end{equation}
where $K\in C_b^1(\R^d\times \R^d)$
is a multivariate 1-periodic
 function (that is, it can be viewed as a function defined on $\T^d\times \T^d$
 with
 $\T^d:=(\R/\Z)^d$) such that
$
K(x,y)=K(y,x)$ and $0<K_1\le K(x,y)\le K_2<\infty$ for all $x,y\in \R^d$
with some constants $K_1$ and $K_2$,
and the function
$j:\R^d\to [0,\infty)$ satisfies
$$
j(z)=j(-z) \ \hbox{ for all } z\in \R^d \quad
\hbox{ and } \quad 0<  \int_{\R^d}\left(1\wedge |z|^2\right)j(z)\,dz<\infty.
$$
Here
$C^2_b(\R^d)$ denotes the space of twice differentiable functions on $\R^d$ with bounded first and second derivatives.

\medskip

Throughout the paper, we assume that there exists a Feller process $X:=\{(X_t)_{t\ge0}; (\Pp_x)_{x\in \R^d}\}$ associated with $\sL$ in the sense that for every $f\in C_b^2(\R^d)$,
\begin{equation}\label{e:mar}\left\{f(X_t)-f(X_0)-\int_0^t \sL f(X_s)\,ds,t\ge0\right\}\,\, \hbox{ is a martingale }\end{equation} under $\Pp_x$ for all $x\in \R^d$ with respect to the natural filtration generated by $X$.
For later use,
denote by $\D(\sL)$ the set of functions
$f$ on $\R^d$ so   that \eqref{e:mar} holds.
The notation $\D(\sL)$ is closely
connected to the domain of the full generator of a Markov process; see
\cite[p.\ 24-25]{BSW}. All the conventions above apply to other operators $\sL_\e$ and $\bar \sL$ below.

\medskip

We consider a suitable scaling of the Feller process $X$.
Let $\varphi $ be a strictly positive  function defined on $\R_+:=(0, \infty)$ with $\lim_{\e \to 0+} \varphi (\e)=\infty$.
For $\e>0$, define $ X_t^\e:= \e X_{\varphi(\e)t}$ for all $t\ge0$. Then $X^\e:=(X^\e_t)_{t\ge0}$ is a Feller process on $\R^d$
whose associated generator $(\sL_\e, \D (\sL_\e))$ acting on $f\in C_b^2(\R^d)$ is given by
 $$
\sL_\e f(x):=\varphi(\e)\e^{-d}{\rm p.v.}\int_{\R^d}\left(f(y)-f(x)\right)K (\e^{-1} x,\e^{-1} y ) j \left(\e^{-1} (x-y) \right)\,dy.
$$

As shown in \cite{CCKW}, the scaling limit of
$\sL_\e$ as $\e \to 0$
depends on the behaviors of the
kernel
 $j(z)$ as $|z|\to \infty$.

\begin{assumption}\label{a1-1}\it Suppose that one of the following three
mutually exclusive conditions holds.
\begin{itemize}
\item [{\rm(i)}]
{\rm (}Subcritical $\alpha$-stable regime{\rm)}
There exist a constant $\alpha\in (0,2)$ and a function
$\varphi:\R_+\to \R_+$
with $\lim_{\e \to 0+} \varphi (\e)= \infty$
such that
\begin{equation}\label{a1-1-1}
\lim_{\e \to 0} \varphi(\e)\e^{-d}j(\e^{-1} z  )=\frac{1}{|z|^{d+\alpha}},\quad z\in \R^d\backslash \{0\}.
\end{equation}

\item [{\rm(ii)}] $($Critical regime$)$ There exists a
function $\varphi:\R_+ \to \R_+$
with $\lim_{\e \to 0+} \varphi (\e)= \infty$
such that
\begin{equation}\label{a1-1-2}
\lim_{\e \to 0}\e^2\varphi(\e)=0,
\quad
\lim_{\e \to 0}\left(\e^3\varphi(\e)\int_{\{|z|\le \e^{-1}\}}|z|^3j(z)\,dz+\e\varphi(\e)\int_{\{|z|>\e^{-1}\}}|z|j(z)\,dz\right)=0.
\end{equation}
and the limit
\begin{equation}\label{a1-1-3}
A:=\lim_{\e \to 0}\e^2\varphi(\e)\int_{\{|z|\le \e^{-1}\}}(z\otimes z)j(z)\,dz
\end{equation}
exists as a finite and non-trivial symmetric matrix.

\item [{\rm(iii)}]
$($Supercritical diffusive regime$)$
    Let  $\{e_i\}_{1\le i\le d}$ be a   basis of $\R^d$ consisting of unit vectors, and denote by
 ${\rm supp}[j]$ the support of the jumping density function $j$.
 Suppose that for each $1\leq i\leq d$, there is a sequence of distinct points
 $\{z_k^i ; \, k\ge 1\}\subset {\rm supp}[j] \setminus  \{0\}$
 so that
 \begin{equation}\label{a1-1-2a}
\lim_{k\to \infty} z_k^i =z_\infty^i \in \R^d  \quad \hbox{and} \quad
  \lim_{k \to \infty}\frac{z_k^i - z_\infty^i}{|z_k^i-z^i_\infty |}= \pm e_i .
 \end{equation}
 Moreover, assume that
\begin{equation}\label{a1-1-4}
\int_{\R^d}|z|^2 j(z)\,dz<\infty.
\end{equation}
\end{itemize}
\end{assumption}
\noindent
In
the paper, \emph{we always take $\varphi(\e)=\e^{-2}$ for the supercritical diffusive regime}, i.e., when Assumption \ref{a1-1}(iii) holds.

\medskip

Denote by $C_\infty(\R^d)$
the set of continuous functions $f$ on $\R^d$ such that $\lim_{|x|\to\infty} f(x)=0. $
Fix $\lambda >0$.
For
$h\in C_\infty(\R^d)$,
 let
$$ u_\e  (x) :=G^{(\e)}_\lambda h (x)  := \Ee \left[  \int_0^\infty e^{-\lambda t} h(X^\e_t) \,dt  \Big| X^\e_0  =x\right]
$$
 be the $\lambda$-resolvent for the process $X^\e$. So $u_\e$ is the unique function in $ \D(\sL_\e)$ that solves
\begin{equation}\label{e1-3-}
(\lambda -\sL_\e )u_\e  = h    \quad \hbox{on }  \R^d
\end{equation}
in the pointwise sense; see e.g. \cite[Lemma 1.27 and p.\ 25]{BSW}.
To consider the convergence rate of $u_\e$ as $\e \to 0$,
in most of the cases we make the following assumption on the regularity for the Poisson equation associated with the operator $\sL$
(see,  e.g. \cite[(A3) on p.\,2876]{CCKW}). Here and in what follows, let $C(\T^d)$ be the set of continuous functions on $\T^d$. Without any confusion we can see $f\in C(\T^d)$ as that defined on $\R^d$ by the standard periodic extension.

\begin{assumption}\label{a1-2}\it
For every $f\in C(\T^d)$ with $\int_{\T^d}f(y)\,dy=0$, there exists a unique multivariate $1$-periodic
solution $\phi_f\in C^1(\T^d)\cap\D(\sL)$ such that
$\sL \phi_f=f$
on $\T^d$ in the pointwise sense having
$\int_{\T^d}\phi_f(y)\,dy=0$.
 Moreover,
$\|\phi_f\|_\infty+\|\nabla \phi_f\|_\infty\le C_0\|f\|_\infty,$
where $C_0>0$ is independent of $f$.
\end{assumption}

As
mentioned
 in \cite[Proposition 7.6]{CCKW}, if there are constants $c_0>0$ and $\alpha_0\in (1,2)$ such that $j(z)=c_0 |z|^{-(d+\alpha_0)}$
for all $z\in \R^d$ with $|z|\le 1$, then Assumption \ref{a1-2} holds.
See the appendix of this paper for details.

Next, we describe the limit of $u_\e$.  For any $y\in \T^d$, define
\begin{equation}\label{e2-1a}
\Phi_\e(y):=\begin{cases}
{\rm p.v.}\displaystyle\int_{\{|z|\le \e^{-1}\}}z K(y,y+z)j(z)\,dz\quad &\hbox{ when Assumption \ref{a1-1}(i) or (ii) holds},\\
{\rm p.v.}\displaystyle\int z K(y,y+z)j(z)\,dz\quad &\hbox{ when Assumption \ref{a1-1}(iii) holds}.\end{cases}
\end{equation}
Note that $j(z)=j(-z)$ for all $z\in \R^d$, and
$K(x, y)$ is multivariate 1-periodic and symmetric in $(x, y)$. When Assumption \ref{a1-1}(i) or (ii) holds,
we have
\begin{align*}
\int_{\T^d}\Phi_\e (y)\,dy
&= \lim_{\delta \to 0} \int_{\T^d} \int_{\{\delta<|z|\leq \e^{-1}\} } z K(y, y+z) j(z) \,dz\, dy = - \lim_{\delta \to 0} \int_{\T^d} \int_{\{\delta<|z|\leq \e^{-1}\} } z K(y, y- z)  j(z) \,dz \,dy \\
&= - \lim_{\delta \to 0} \int_{\T^d} \int_{\{\delta<|z|\leq \e^{-1}\} } z K(y- z, y)  j(z) \,dz \,dy = - \lim_{\delta \to 0} \int_{\T^d} \int_{\{\delta<|z|\leq \e^{-1}\} } z K(\tilde y, \tilde y +z)  j(z)\, dz \,d\tilde  y \\
&=-  \int_{\T^d}\Phi_\e (y)\,dy.
\end{align*}
The above holds true by the same reasoning, when Assumption \ref{a1-1}(iii)
is satisfied.
Thus  under Assumption \ref{a1-1}, $\int_{\T^d}\Phi_\e (y)\,dy=0$.
Consequently,  under Assumption \ref{a1-2},
there exists a unique periodic   solution $\phi_0^\e : \T^d  \to \R^d$
in $C^1(\T^d)\cap  \D(\sL)$
to the following corrector equation in the  pointwise sense
\begin{equation}\label{e2-2}
\sL \phi_0^\e =- \Phi_\e  \,\,\hbox{on  }  \T^d
\ \hbox{ with } \ \int_{\T^d}\phi_0^\e (y)\,dy=0,
\end{equation}
and   there is a constant $c_0>0$ (which is independent of $\e$) so that
\begin{equation}\label{e2-2a}
\|\phi_0^\e\|_\infty+\|\nabla \phi_0^\e\|\le c_0\|\Phi_\e\|_\infty.
\end{equation}
Below we write the components of $\phi_0^\e(y)$ by $\phi_0^\e(y)=(\phi_{0,1}^\e(y),\cdots,\phi_{0,d}^\e(y))$.
We define the following operator $\bar \sL$ with constant coefficients,
which will turn out to be the limit operator of $\sL_\e$ as $\e \to 0$.
For $f\in C_b^2(\R^d)$, define
\begin{equation}\label{e2-1a-}
\bar \sL f(x) =\begin{cases} \bar \sL_\alpha f(x):=
{\rm p.v.}\displaystyle\int_{\R^d}\left(f(y)-f(x)\right)\frac{\bar K}{|y-x|^{d+\alpha}}\,dy\quad &\hbox{ when Assumption \ref{a1-1}(i) holds}, \medskip \\
\bar \sL_2 f(x):=\frac{1}{2}\left\langle \bar A, \nabla^2 f(x)\right\rangle\quad &\hbox{ when Assumption \ref{a1-1}(ii) holds}, \medskip  \\
\bar \sL_{>2} f(x):=\frac{1}{2}\left\langle \bar A_0, \nabla^2 f(x)\right\rangle\quad &\hbox{ when Assumption \ref{a1-1}(iii) holds},
\end{cases}
\end{equation}where
$\bar K:=\int_{\T^d}\int_{\T^d}K(x,y)\,dx\,dy,$ $\bar A=\bar K A$
with
$A$ being the matrix  defined by \eqref{a1-1-3}, and
\begin{equation}\label{e2-3-2}
\bar A_0:=\int_{\T^d} \left(\int_{\R^d}\left(z+\phi_0(y+z)-\phi_0(y)\right)\otimes \left(z+\phi_0(y+z)-\phi_0(y)\right)K(y,y+z)
j(z)\,dz \,\right)dy.
\end{equation}
Here and in what follows, for the supercritical diffusive regime, i.e., when \eqref{a1-1-4} holds, we write
$\phi_0$ and $\Phi$ for $\phi_0^\e$ and $\Phi_\e$,
 respectively,  since they are independent of $\e$ by \eqref{e2-1a} and \eqref{e2-2}.
Note that $\bar \sL$ of \eqref{e2-1a-} is a L\'evy operator. It uniquely determines a L\'evy process $\bar X:=(\bar X_t)_{t\ge0}$ on $\R^d$
having $\bar \sL$ as its infinitesimal generator.

For any $\lambda>0$ and $h\in  C_\infty(\R^d)$, let
$$ \bar u(x) :=\bar G_\lambda h (x)= \Ee \left[  \int_0^\infty e^{-\lambda t} h (\bar X_t) \,dt  \Big| \bar X_0 =x\right]
$$
 be the $\lambda$-resolvent of $h$ for the L\'evy process $\bar X$. Note that $\bar u$ solves
\begin{equation}\label{e1-3}
( \lambda -\bar \sL)   \bar u  = h \quad \hbox{on } \R^d
\end{equation}
in the  pointwise sense.
Under Assumption \ref{a1-1} and Assumption \ref{a1-2}, Chen, Chen, Kumagai and Wang \cite{CCKW}
proved that,
for any $h\in C_\infty(\R^d)$ and $x\in \R^d$,
$$
 \lim_{\e\to 0}u_\e(x)=\bar u(x).
$$
We refer the reader to \cite{BCKW,CCKW0,CCKW1,CKK,CKW,CKW1,FHS,HDS,KPZ,KKM,PZ,PZ1,PZ2,KP,S,S1} and references therein for the recent progress
on qualitative homogenization for different type of non-local operators.

The results on quantitative homogenization for non-local operators are
relatively limited. Piatnitski, Sloushch, Suslina and Zhizhina \cite{PSSZ} established operator estimates in homogenization for a class of convolution-type non-local operators in the supercritical diffusive regime.
Later they studied in \cite{PSSZ24,PSSZ25,PSSZ26} operator estimates (with explicit rates) as well as
an optimal convergence rate of the resolvents
in homogenization problem of symmetric $\alpha$-stable-like operators
with periodic coefficients in $L^2(\R^d;dx)$ by means of the operator-theoretic approach
and the Gelfand transform.
 Chen, Chen, Kumagai and Wang \cite{CCKW2025} investigated quantitative homogenization for stable-like random walks in the i.i.d. random conductance model by introducing a regional corrector. Subsequently, accounting for the blow-up behavior of solutions near the boundary, Chen, Chen, Kumagai and Wang \cite{CCKW2} derived
convergence rate estimates for the periodic homogenization of Dirichlet problems associated with symmetric stable-like operators. To the best of our knowledge, quantitative estimates, particularly high-order estimates, for the periodic homogenization of general L\'evy-type operators (covering the
subcritical $\alpha$-stable regime, critical regime, and supercritical
diffusive regime specified in Assumption \ref{a1-1})
had not been  explored.
 In this paper, we address this problem by studying the higher-order convergence rates of $u_\e$ as $\e\to0$.

Higher order homogenization extends the classical homogenization by including higher-order terms in the expansion
to analyze scale effects and improve accuracy. It
can be used to determine the effective heterogeneous media by going beyond the traditional
averaging, capturing the microscopic oscillations and wave dispersion.
Higher order homogenization typically utilizes a two-scale asymptotic expansion, introducing a fast variable
for the micro-structure and a slow variable for the macro-scale, creating a hierarchy of equations and correctors.
This is the approach we will use and develop in this paper for non-local  L\'evy-type operators of the form \eqref{e1-1}.

\subsection{Three typical examples} To illustrate  the contributions of this paper, we first present three
 examples conforming to Assumption \ref{a1-1} in this section, demonstrating that our results apply to a broad class of non-local operators characterized by distinct jumping kernels.

We note that the operator \((\sL, C_c^2(\mathbb R^d))\) and its limit counterpart \((\bar \sL,C_c^2(\mathbb R^d))\) are symmetric on $L^2(\R^d;dx)$.
So it is routine to see that there exist (symmetric) closed extension \((\sL, \D_2(\sL))\) and \((\bar \sL, \D_2(\bar \sL))\) in $L^2(\R^d;dx)$ for
\((\sL, C_c^2(\mathbb R^d))\) and \((\bar \sL,C_c^2(\mathbb R^d))\), respectively.
Throughout this paper, we always use
$\sL$ to denote
the action to functions both in $\D(\sL)$ and $\D_2(\sL)$ without cause of confusion.
In this sense, the solutions $u_\e$ and $\bar u$ to the equations \eqref{e1-3-} and \eqref{e1-3} respectively lie in $L^2(\R^d;dx)$ provided $h\in C_\infty\cap L^2(\R^d;dx).$
In particular, by the symmetric property of $\sL_\eps$ in $L^2(\R^d;dx)$ it is easy to verify that
\begin{equation}\label{e1-2a}
\|u_\e\|_{L^2(\R^d;dx)}\le \lambda^{-1}\|h\|_{L^2(\R^d;dx)}.
\end{equation}

To fully exploit the symmetry property inherent
in our framework,
we consider a family of test functions $h\in \M_\beta\subset C_\infty\cap L^2(\R^d;dx)$ for some proper $\beta>0$ (see \eqref{e1-3a} below), and then investigate higher-order homogenization problems in the space $L^2(\R^d;dx)$.
 To derive high-order estimates   between $u_\e$ and $\bar u$,
 we inductively introduce the high-order correctors $\psi_n^\e$ and $\phi_n^\e$ (defined by \eqref{e2-13} and \eqref{e2-14} below), which correspond respectively to the slow and fast variables in the homogenization framework.
Then,
for $n\ge 1$, we can define the $n$-th order expansion $v_n^\e:\R^d\to \R$ for $u_\e$ as follows
(see also \eqref{e2-14a} below)
\begin{equation*}
\begin{split}
v_{n+1}^\e(x)
&=v_{n}^\e(x)-\psi_{n+1}^\e(x)-\e\left\langle \nabla \psi_{n+1}^\e(x), \phi_0^\e\left(\e^{-1} x \right)\right\rangle-
\varphi(\e)^{-1}\phi_{n+1}^\e\left(x,\e^{-1} x \right)\\
&=u_\e(x)-\left(\bar u(x)+\sum_{k=2}^{n+1} \psi_k^\e(x)\right)-\e\left\langle
\left(\nabla \bar u(x)+\sum_{k=2}^{n+1} \nabla \psi_k^\e(x)\right),\phi_0^\e\left(\e^{-1} x \right) \right\rangle-
\varphi(\e)^{-1}\sum_{k=1}^{n+1} \phi_k^\e\left(x,\e^{-1} x \right).
\end{split}
\end{equation*}
In particular,
the correctors
$\psi_n^\e$ and $\phi_n^\e$ serve respectively to eliminate the lower-order terms associated with the slow and fast variables arising from the preceding inductive step, thereby generating new residual terms with higher-order errors.
We also refer the reader to Remark \ref{r1-3} below for detailed explanation
about different terms in $v_{n+1}^\e$.
Our main results, Theorems \ref{thm1}, \ref{thm2} and \ref{thm3}, shows in a sense that the error term $v_{n+1}^\eps$  whose $L^2$-norm on $\R^d$  is of a higher order
in $\eps$ than that of $v_n^\e$.  See also Examples \ref{ex1-1}-\ref{ex1-3} below.

\medskip

Set
\begin{equation}\label{ex1-1-1}
  \xi_0(\e):=\e^{2-\alpha}\I_{\{\alpha\in (1,2)\}}+\e (\log \e )^2 \I_{\{\alpha=1\}}
+\e^{\alpha}\I_{\{\alpha\in (0,1)\}}.
\end{equation}

For the sake of simplicity,
we set $j(z)=|z|^{-d-\alpha_0}$ for all $z\in \R^d$ satisfying $|z|\le 2$, where $\alpha_0\in (1,2)$ unless particularly emphasized.
(Note that $|z|\le 2$ can be replaced by $|z|\le C$ for any $C>0$.
Here we take $C=2$ for simplicity when encountering the logarithmic function  later.)

\medskip
The first example is the most typical
case of jumping kernel of polynomial type.

\begin{example}\label{ex1-1}\it
Suppose that $j(z)=\frac{1}{|z|^{d+\alpha}}$
for all $z\in \R^d$ with $|z|>2$.
\begin{itemize}
\item [(1)]
$($Subcritical $\alpha$-stable regime$)$
If $\alpha\in (0,2)$,  ,
then it holds that
\begin{equation}\label{ex1-1-1b}
\|v_n^\e\|_{L^2(\R^d;dx)}\le C_0(n) \xi_0(\e)^n,\quad n\ge 1,\ \e\in (0,1).
\end{equation}

\item [(2)] $($Critical regime$)$  If $\alpha=2$,
then
it holds that
$$
\|v_n^\e\|_{L^2(\R^d;dx)}\le C_0(n)\left(\frac{1}{|\log \e|}\right)^{n},\quad \ n\ge 1,\ \e\in (0,1).
$$
\item [(3)]
$($Supercritical diffusive regime$)$
If $\alpha>2$, then it holds that
$$
\|v_n^\e\|_{L^2(\R^d;dx)}\le C_0(n) \left(\max\{\e,\e^{\alpha-2}\}\right)^n,\quad \ n\ge 1,\ \e\in (0,1).
$$
\end{itemize}
\end{example}

\medskip

The next example is when
$j(z)$ has an additional multiplicative factor $(\log |z|)^m$.

\begin{example}\label{ex1-2}\it
Suppose that $j(z)= \frac{(\log |z|)^m}{|z|^{d+\alpha}}$
for all $z\in \R^d$ with $|z|>2$.
\begin{itemize}
\item [(1)]
$($Subcritical $\alpha$-stable regime$)$
 If $\alpha\in (0,2)$ and $m\in \R$,
 then it holds that
\begin{equation}\label{ex1-1-2}
\|v_n^\e\|_{L^2(\R^d;dx)}\le C_0(n,m)
 \left(\frac{1}{|\log \e|}\right)^n,\quad n\ge 1,\ \e\in (0,1).
\end{equation}
\item [(2)] $($Critical regime$)$  If $\alpha=2$ and $m\in (-1,\infty)$,
then
$$
\|v_n^\e\|_{L^2(\R^d;dx)}\le
C_0(n,m)\left(\frac{1}{|\log \e|}\right)^{\min\{1,m+1\}n},\quad \ n\ge 1,\ \e\in (0,1).
$$
Furthermore, if $\alpha=2$ and $m=-1$,
then
$$
\|v_n^\e\|_{L^2(\R^d;dx)}\le C_0(n)\left(\frac{1}{|\log \log \e^{-1}|}\right)^{n},\quad \ n\ge 1,\ \e\in (0,1).
$$
\item [(3)]
$($Supercritical diffusive regime$)$
 If $\alpha=2$ and $m<-1$,
  then
$$
\|v_n^\e\|_{L^2(\R^d;dx)}\le C_0(n,m)\left(\frac{1}{|\log \e|}\right)^{(m+1)n},
\quad \ n\ge 1,\ \e\in (0,1);$$
If $\alpha>2$ and $m\in \R$,  then
$$
\|v_n^\e\|_{L^2(\R^d;dx)}\le C_1(n,m)\left(\max\{\e, \e^{\alpha-2}|\log \e|^m\}\right)^n,
\quad \ n\ge 1,\ \e\in (0,1).$$
\end{itemize}
\end{example}

\medskip
For the subcritical $\alpha$-stable regime, we have an additional example where the polynomial is mixed order as follows.

\begin{example}\label{ex1-3}\it$($Mixed $\alpha$-stable regime$)$
\begin{itemize}
\item [(1)] If there are
$0<\alpha_1 <\alpha<2$ such that   $ j(z)=\frac{1}{|z|^{d+\alpha}+|z|^{d+\alpha_1}}$ for all $z\in \R^d$ with $|z|>2$,
then
$$
\|v_n^\e\|_{L^2(\R^d;dx)}\le C_0(n) \left( \max\{ \xi_0(\e),\e^{\alpha-\alpha_1}\}\right)^n,\quad n\ge 1,\ \e\in (0,1).
$$

\item [(2)] If there are $\alpha\in (0,2)$ and
$\alpha<\alpha_1$
such that $j(z)=
\frac{1}{|z|^{d+\alpha}}+\frac{1}{|z|^{d+\alpha_1}}
$
 for all $z\in \R^d$ with $|z|>2$,
then
$$
\|v_n^\e\|_{L^2(\R^d;dx)}\le C_0(n)\left(\max\{\xi_0(\e),
\e^{\alpha_1-\alpha}\}\right)^n,\quad n\ge 1,\ \e\in (0,1).$$
\end{itemize}
\end{example}

\medskip
In the above examples, the constants $C_0(n), C_0(n,m) $ and $C_1(n,m)$ are independent of $\e$. Moreover,
it hold that $\lim_{m \to 0}C_0(n,m)>0$ for the subcritical $\alpha$-stable regime,
$\lim_{m\downarrow -1} C_0(n,m)=\infty$
for the critical regime, and
$\lim_{m \uparrow -1}C_0(n,m)=\infty$
for the supercritical diffusive regime.

\subsection{Remarks and novel contributions}
For models related to
the
second-order differential operators (or discrete elliptic differential operators), important progress has been made on
the
quantitative theory of (both periodic and stochastic) homogenizations recently;
 see \cite{ABK, AK, AKM1, AKM, CMOW, GNO, GO, KLS, KLS1, KLS2, Sh} and the references cited therein
for more details.
Moveover,
higher-order
convergence rates of homogenization have been further explored.
In a series of seminal works \cite{KL,KL1,KL2}, Kim and Lee established higher-order convergence rates of periodic homogenization for fully nonlinear elliptic equations, parabolic equations with oscillatory initial data, and viscous Hamilton-Jacobi equations, respectively. Gu \cite{G} obtained
higher-order convergence rates of homogenization for divergence-form elliptic equations in finite-range dependent ergodic environments. Bella, Fehrman, Fischer and Otto \cite{BFFO} derived high-order error estimates in weak norms for the stochastic homogenization of elliptic equations. Kleptsyna, Piatnitski and Popier \cite{KPP} further analyzed
higher-order convergence rates of periodic homogenization for a class of non-autonomous second-order differential operators.
For  optimal convergence rates of periodic homogenization for second-order differential operators
(with non-divergence form), we refer the reader to \cite{GST, ST} and the references therein.

We close this section with several remarks on the aforementioned examples and our main results,
emphasizing the key innovations and novel contributions of this paper.

\begin{itemize}
\item [(i)] In contrast to the high-order estimates for differential operators established in \cite{G,KL,KL1,KL2}, a core challenge in our setting stems from the diversity and complexity of the jumping kernel
for
the non-local operator $\sL$. To address this, we develop a systematic framework to dissect the kernel's contributions across small, intermediate, and large spatial scales, tailored to the distinct regimes specified earlier. This multi-scale analysis is encapsulated by the terms
$\gamma_1(\e)$--$\gamma_{9}(\e)$
detailed in Section \ref{section2}. As demonstrated by Examples \ref{ex1-1}--\ref{ex1-3}, the convergence rate is ultimately governed by the interplay of these scale-dependent effects; notably, even within the same regime, the dominant terms driving the estimates can differ
substantially  among different jumping kernels.

    \item [(ii)]
As mentioned above, in order to derive
high-order estimates, we inductively introduce the high-order correctors $\psi_n^\e$ and $\phi_n^\e$
by \eqref{e2-13} and \eqref{e2-14} below,
which correspond, respectively, to the slow and fast variables in the homogenization framework.
From a probabilistic perspective, the non-local nature of $\sL$ necessitates a quantitative analysis of homogenization effects for two distinct variables: $x$, denoting the current position of the stochastic process $(X_t^\e)_{t\ge0}$, and $z$, representing the magnitude of jumps in the process $(X_t^\e)_{t\ge0}$. This dual-variable dependence induces regime-specific coupling mechanisms between the two correctors -- a key distinction from the differential operator setting in \cite{G,KL,KL1,KL2}, where homogenization only needs to account for the position variable $x$.
In fact, via careful analysis on the interaction of these two variables, we will develop some
systematic methods to obtain quantitative estimates of $\psi_n^\e$ and $\phi_n^\e$, where the effects of
different types of jumping kernels have been taken into account.

    \item [(iii)]    For the subcritical $\alpha$-stable regime, Theorem \ref{thm1} below,
Example \ref{ex1-1}(1), Example \ref{ex1-2}(1)
and Example \ref{ex1-3}
 together reveal that symmetric non-local operators with homogeneous jumping kernels matching the limit operator outside a compact set (i.e., there exists a constant $C_0>0$ such that $j(z)=\frac{1}{|z|^{d+\alpha}}$ for all $z\in \R^d$ with $|z|>C_0$)
attains
 the optimal (fastest) convergence rate, with the $n$-th order expansion yielding a rate of $\xi_0(\e)^n$. In the case $n=1$, this result recovers the Berry-Esseen bounds
   established in \cite{CNX,X}
and \cite[Theorem 4]{CNXYZ}
  for i.i.d. sequences of symmetric  random variables in the domain of attraction of the $\alpha$-stable law.
 (For
i.i.d.
 sequence of symmetric  random variables  whose total variational measure away from the symmetric $\alpha$-stable law has finite second moment,
 it is shown in \cite{B} the Berry-Esseen bound can be improved to the order $\eps$ for $0<\alpha <1$.)
 By contrast, for general non-local operators, discrepancies between the original and limit jumping kernels can lead to a deterioration in convergence speed, as exemplified by
Example \ref{ex1-2}(1) and Example \ref{ex1-3}.
            We further emphasize that the assertion in \eqref{ex1-1-2} is consistent with \cite[Proposition 1]{YP}, which establishes analogous convergence rates for sums of one-dimensional i.i.d.\ $\alpha$-stable random variables with logarithmic perturbations, as such sums converge to an $\alpha$-stable law.

                          \item [(iv)] As elaborated in Remark \ref{r2-2} below, the regularity of the solution $\bar u$ to the limit equation \eqref{e1-3} exerts a fundamental influence on the convergence rate. This dependence originates from the homogenization error associated with the jump-size variable $z$, a distinctive feature of symmetric non-local operators $\sL$ that may not be extended to their non-symmetric counterparts.

                                     \item [(v)] There is an interesting aspect in Example \ref{ex1-3}.
 Indeed, both for (1) and (2),  $j(z)\asymp |z|^{-d-\alpha}$ for $|z|\ge 2$
 and $j(z)=|z|^{-d-\alpha_0}$ for $|z|\le 2$.
The estimates of $\|v_n^\e\|_{L^2(\R^d;dx)}$  all contain a term $\e^{|\alpha-\alpha_1|}$, which slows down to order $\eps^0=1$ as $\alpha_1\to \alpha$.
Some explanations are in order for this counterintuitive phenomenon as well as that of Example \ref{ex1-2}.
For this, let us look at a simple example, where $\sL$ is the L\'evy generator of \eqref{e1-1} with $K(x, y) \equiv 1$ and $j(z)= \frac{1}{|z|^{d+\alpha}}+\frac{1}{|z|^{d+\alpha_1}}$
for all $z\in \R^d\setminus \{0\}$ with $1<\alpha <\alpha_1$; that is, $\sL$ is the infinitesimal generator of a L\'evy process $(X_t)_{t\ge0}$ that is the independent sum of
an isotropic $\alpha$-stable process $Y:=(Y_t)_{t\ge0}$ and an isotropic $\alpha_1$-stable process $Z:=(Z_t)_{t\ge0}$ on $\R^d$.
 Although   $j(z)$ is not set to be  $\frac{1}{|z|^{d+\alpha}}$ on $\{|z|\leq 2\}$, its scaled process $(X^\eps_t)_{t\ge0}:= (\eps X_{\eps^{-\alpha}t})_{t\ge0}$
 shares the same asymptotic behavior as $\eps \to 0$. Indeed, by the $\alpha$-stable scaling,
 $$
 X^\eps_t = \eps Y_{\eps^{-\alpha}t} + \eps  Z_{\eps^{-\alpha}t} \overset{d}=
 Y_t + \eps^{1-(\alpha/\alpha_1)}  \widetilde Z_t,\quad t\ge0,
 $$
 which converges in law to the isotropic $\alpha$-stable process $Y$ as $\eps \to 0$. Here the notation $\overset{d}=$ means having the same distribution, and    $(\widetilde Z_t)_{t\ge0}$ is an independent copy of $Z$.
 Note that
 $$
\sL_\e f(x)={\rm p.v.}\int_{\R^d}\left(f(y)-f(x)\right) \left( \frac{1}{|x-y|^{d+\alpha}} + \frac{\eps^{\alpha_1-\alpha} }{|x-y|^{d+\alpha_1}}
\right) \, dy
\quad \hbox{ for } f\in C_c^2 (\R^d).
 $$
 Denote the resolvents of the L\'evy processes $Y$ and $Z$ by $\{\bar G^{(\alpha)}_\lambda; \lambda >0\}$ and
 $\{\bar G^{(\alpha_1)}_\lambda; \lambda >0\}$, respectively. Then we have, at least at the heuristic level,
  the following Duhamel's formula for $f\in C^\infty_c(\R^d)$:
 \begin{equation}\begin{split}
 G^{(\eps)}_\lambda f
 &= \bar G^{ (\alpha)}_\lambda f + \eps^{\alpha_1-\alpha} \bar G^{(\alpha)}_\lambda (\Delta^{\alpha_1/2} G^{(\eps)}_\lambda  f)
 =\bar G^{ (\alpha)}_\lambda f + \sum_{n=1}^\infty \eps^{(\alpha_1-\alpha)n} \bar G^{(\alpha)}_\lambda (\Delta^{\alpha_1/2} \bar G^{(\alpha)}_\lambda )^n f \\
 &=\bar G^{ (\alpha)}_\lambda f + \sum_{n=1}^\infty \eps^{(\alpha_1-\alpha)n} \bar G^{(\alpha)}_\lambda ( \bar G^{(\alpha)}_\lambda \Delta^{\alpha_1/2} )^n f ,
 \label{e:1.18}
 \end{split}\end{equation}
 where $\Delta^{\alpha_1/2}$ denotes the infinitesimal generator of the isotropic $\alpha_1$-stable process $Y$, which is a constant multiple
 of the fractional Laplacian of order $\alpha_1/2$. The higher-order expansion \eqref{e:1.18} of $G^{(\eps)}_\lambda f$ in $\eps$ indicates that the estimates
 in Example \ref{ex1-3} are reasonable and sharp.
 \end{itemize}

\ \

The remainder of this paper is organized as follows. In
Section \ref{section2},
we present the main results of
this paper,
which are partitioned into three subsections corresponding to the subcritical $\alpha$-stable, critical, and supercritical diffusive regimes specified in Assumption \ref{a1-1}. In each subsection, following the introduction of necessary notations, we provide a detailed definition of the $n$-th order expansion $v_n^\e$ for $u^\e$, and establish the corresponding upper bounds. Section \ref{section3} compiles preliminary results that underpin the proofs of our main theorems, with explicit bounds derived for $\Lambda_\e^i$ ($1\le i\le 7$) and $\Theta_\e^i$ ($1\le i\le 3$). Finally, Section \ref{section4} is dedicated to the rigorous proofs of the main results and Examples \ref{ex1-1}--\ref{ex1-3}.

\section{Main Results}\label{section2}

In this section, we will present main results of our paper. It is divided into three subsections according to
the subcritical $\alpha$-stable, critical, and supercritical diffusive regimes
specified in Assumption \ref{a1-1}.
For every $\beta>0$, let
\begin{equation}\label{e1-3a}
\M_\beta:=\{f\in C_b^\infty(\R^d):\|f\|_{\M_{\beta},k}\le C_0(k) \hbox{ for every }k\ge 0\},
\end{equation}
where
$
\|f\|_{\M_{\beta},k}:=\sup_{x\in \R^d}\{|\nabla^k f(x)|(1+|x|)^{d+\beta}\}$ for all $k\ge0,
$
and $\nabla^k f$ denotes the $k$-th order gradient for the function $f$. In particular,
$
|\nabla^k f(x)|\le \|f\|_{\M_{\beta},k}(1+|x|)^{-d-\beta}$ for all $x\in \R^d$ and $k\ge 0.
$
Observe that $\M_\beta\subset  C_\infty(\R^d)\cap L^2(\R^d; dx)$.

\subsection{Subcritical $\alpha$-stable regime with $0<\alpha <2$}\label{S:2.1}

Let $\beta >0$.
For every $f\in \M_\beta$ and $x,z\in \R^d$, let
\begin{equation}\label{e1-4}
\delta f(x;z):=f(x+z)-f(x)-\langle \nabla f(x), z \I_{\{|z|\le 1\}}\rangle.
\end{equation}
For every $f\in \M_{\beta}$,
$x\in \R^d$,  $y\in\T^d$ and $\eps>0$, we define
\begin{equation}\label{e1-5a}
\begin{split}
\Lambda_\e^1 f(x,y)&=\int_{\R^d}\delta f(x;z)
K(y,y+\e^{-1} z )\left(\varphi(\e)\e^{-d}j (\e^{-1} z )-\frac{1}{|z|^{d+\alpha}}\right)\,dz,\\
\Lambda_\e^2 f(x,y)&=\int_{\R^d}\delta f(x;z) \frac{K(y,y+\e^{-1} z )-\bar K(y)}{|z|^{d+\alpha}} \,dz,\\
\Theta_\e^1 f(x,y)&=\int_{\R^d}\delta f(x;z)\frac{\bar K(y)-\bar K}{|z|^{d+\alpha}}\,dz,\end{split}
\end{equation}
where $\bar K(y):=\displaystyle\int_{\T^d}K(y,y+z)\,dz$.
For every $f\in \M_{\beta}$, $\phi\in C^1(\T^d)$,
$x\in \R^d$,  $y\in \T^d$ and $\eps>0$, we set
\begin{equation}\label{e1-6}
\Gamma_\e(f,\phi)(x,y):=\varphi(\e)\e^{-d}\int_{\R^d}(f(x+z)-f(x))
(\phi(y+\e^{-1} z )-\phi(y))K(y,y+\e^{-1} z )j (\e^{-1} z )\,dz,
\end{equation}
\begin{equation}\label{l2-1-3a}
\begin{split}
\Upsilon_1^\e f(x,y):= &\sum_{i=1}^2\Lambda_{\e}^i f(x,y)+\Theta_{\e}^1f(x,y)+
\e\sum_{i=1}^d\Gamma_\e(\partial_{x_i}f,\phi_{0,i}^\e)(x,y)\\
&+\e\left\langle \bar \sL_\alpha (\nabla f)(x)+\sum_{i=1}^2\Lambda_\e^i(\nabla f)
(x,y)+\Theta_\e^1(\nabla f)
(x,y), \phi_0^\e(y)\right\rangle\\
&+\e^{2}\varphi(\e)\left\langle
\nabla^2 f(x), \phi_0^\e(y)\otimes \Phi_\e(y)\right\rangle-\lambda\e
\left\langle \nabla f(x), \phi_0^\e(y)\right\rangle
\end{split}
\end{equation}
and
\begin{equation}\label{l2-1-3}
F_1^\e(x,y)  := \Upsilon_\e^1 \bar u(x,y),   \quad \bar F_1^\e(x):=\int_{\T^d}F_1^\e(x,y)\,dy,
\end{equation}
where $\phi_0^\e$ is the solution to \eqref{e2-2},
$\bar \sL_\alpha$ is given in \eqref{e2-1a-}, and $\bar u$ is the solution to \eqref{e1-3}.
According
to Assumption \ref{a1-2}, there exists a unique  function
 $\phi_1^\e:\R^d \times \T^d \to \R$
such that
\begin{equation}\label{e2-3aa}
-\sL \phi_1^\e(x,\cdot)(y)=F_1^\e(x,y)-\bar F_1^\e(x) \hbox{ for all } x\in \R^d \hbox{ and }y\in \T^d
\hbox{ with } \int_{\T^d}\phi_1^\e(x,y)\,dy=0,
\end{equation}
thanks to the fact that $\displaystyle\int_{\T^d}(F_1^\e(x,y)-\bar F_1^\e(x))\,dy=0$.
We note that the smoothness  of $\bar u$ depends on the function $h$ given in \eqref{e1-3}. As
 shown in Lemma \ref{l1-1} below, for $h\in \M_\beta$, $\bar u\in \M_{\beta_0}$, where $\beta_0:=\min\{\alpha,\beta\}$.
We
 then introduce the
 first order expansion
\begin{equation}\label{e2-7}
v_1^\e(x):=u_\e(x)-\bar u(x)-\e\left\langle\nabla \bar u(x),
\phi_0^\e\left(\e^{-1} x \right)\right\rangle-\varphi(\e)^{-1}\phi_1^\e\left(x,\e^{-1} x \right),\quad x\in \R^d.
\end{equation}

Inductively, for $n\ge 1$, $x\in \R^d$ and $y\in \T^d$, define
\begin{equation}\label{l2-2-2}
\begin{split}
G_{n}^\varepsilon(x,y)&=\varepsilon^{-d}\int_{\R^d} \delta\phi_n^\varepsilon(\cdot, y)(x;z)
K\left(y,y+\e^{-1} z\right)j \left(\e^{-1} z \right)\,dz+\e\left\langle \nabla_x \phi_n^\e\left(x,y\right),\Phi_\e\left(y\right)\right\rangle\\
&\quad+\varepsilon^{-d}\int_{\R^d} \delta^\varepsilon_2\phi_n^\varepsilon(x,y;z)K\left(y,y+\e^{-1} z\right)j \left(\e^{-1} z \right)\,dz
-\lambda \varphi(\e)^{-1}\phi_n^\e(x,y),\\
\bar G_n^\e(x)&=\int_{\T^d}G_n^\e(x,y)\,dy,\\
F_{n+1}^\e(x,y)&=\Upsilon_1^\e \psi_{n+1}^\e(x,y),\quad \bar F_{n+1}^\e(x)=\int_{\T^d} F_{n+1}^\e(x,y)\,dy.
\end{split}
\end{equation}
Here,  for every $n\ge1$, $\psi_{n+1}^\e:\R^d \to \R$ is the unique solution to
the following resolvent equation
\begin{equation}\label{e2-13}
\lambda \psi_{n+1}^\e(x)-\bar \sL_\alpha \psi_{n+1}^\e(x)=\bar F_n^\e(x)+\bar G_n^\e(x), \quad x\in \R^d,
\end{equation}
and
$\phi_n^\e:\R^d\times \T^d \to \R$ is the unique solution to the following equation
\begin{equation}\label{e2-14}
 -\sL \phi_{n+1}^\e(x,\cdot)(y)=G_n^\e(x,y)-\bar G_n^\e(x)+F_{n+1}^\e(x,y)-\bar F_{n+1}^\e(x),\quad x\in \R^d,\ y\in \T^d
\end{equation}
with $ \int_{\T^d}\phi_{n+1}^\e(x,y)\,dy
=0$;
while the operator $\delta_2^\e$ is defined by
\begin{equation}\label{l2-2-2a}
\delta_2^\varepsilon \phi_n^\e(x,y;z)=\phi_n^\e\left(x+z,y+\e^{-1} z\right)-
\phi_n^\e\left(x,y+\e^{-1} z\right)-\phi_n^\e(x+z,y)+\phi_n^\e(x,y),
\end{equation}
and $\delta\phi_1^\varepsilon(\cdot, y)(x;z)$ denotes the operator $\delta f(x;z)$ defined by \eqref{e1-4} acting on the variable $x\in \R^d$.
Note that the above definition \eqref{e2-14}  for $\phi_{n+1}^\e$ with $n\geq 1$ is consistent with \eqref{e2-3aa} for $\phi_1^\e$ if we set
  $G_0^\e(x,y)=\bar G_0^\e(x)=0$. According to \eqref{e2-14}, \eqref{l2-2-1}  and \eqref{e:add--} below,
$$
G_n^\e(x,x/\e) =\varphi(\e)^{-1}\sL_\e\phi_n^\e(\cdot,\cdot/\e)(x)-\sL\phi_n^\e(x,\cdot/\e)(x)
- \lambda \varphi(\e)^{-1}\phi_n^\e(x,x/\e).
$$
This indicates that $G^\e_n(x,y)$ reflects the homogenized error between the scaled operator $\sL_\e$ and the original operator $\sL$ acting on the function
 $\phi^\e_n(x,y)$.
For $n\ge 1$,  we will define the $(n+1)$-th expansion
$v_{n+1}^\e:\R^d\to \R$ as follows
\begin{equation}\label{e2-14a}
\begin{split} v_{n+1}^\e(x)
&=v_{n}^\e(x)-\psi_{n+1}^\e(x)-\e\left\langle \nabla \psi_{n+1}^\e(x), \phi_0^\e\left(\e^{-1} x \right)\right\rangle-
\varphi(\e)^{-1}\phi_{n+1}^\e\left(x,\e^{-1} x \right)\\
&=u_\e(x)-\bar u(x)-\sum_{k=2}^{n+1} \psi_k^\e(x)-\e\left\langle
\left(\nabla \bar u(x)+\sum_{k=2}^{n+1} \nabla \psi_k^\e(x)\right),\phi_0^\e\left(\e^{-1} x \right) \right\rangle-\varphi(\e)^{-1}\sum_{k=1}^{n+1} \phi_k^\e\left(x,\e^{-1} x \right).
\end{split}\end{equation}

\begin{remark}\label{r1-3}
Let us
give some explanation on the high order expansion $v_n^\e$. Here we only take
the subcritical $\alpha$-stable regime as an example. For $n=1$,
by \eqref{l2-1-1}, \eqref{e2-9a} and \eqref{l2-2-3} below, it holds that
\begin{align*}
\left(\lambda-\sL_\e\right)v_1^\e(x)=\bar F_1^\e\left(x\right)+G_1^\e\left(x,\e^{-1} x \right),
\end{align*}
which along with \eqref{e1-2a}  induces
\begin{align}\label{p2-1-1--}
\|v_1^\e\|_{L^2(\R^d;dx)}\le c_1\left(\xi_{1,\beta_0}(\e)+\xi_{2,\beta_0}(\e)\right),
\end{align}
where $\xi_{1,\beta_0}(\e)$ and $\xi_{2,\beta_0}(\e)$ are defined by \eqref{e1-8a} with $\beta_0:=\min\{\alpha,\beta\}$
for some $\beta>0$.

To obtain estimates for higher order, we will introduce the correctors $\psi_n^\e$ and $\phi_n^\e$ by induction for
$n\ge 2$
by
\begin{align*}
\left(\lambda-\sL_\e\right)\left(\psi_n^\e(\cdot)+\e\left\langle \nabla \psi_n^\e(\cdot), \phi_0^\e\left(\e^{-1}\cdot\right)\right\rangle\right)(x)
=\left(\lambda-\bar \sL_{\alpha}\right)\psi_n^\e(x)-F_n^\e\left(x,\e^{-1} x \right)=
\bar F_{n-1}^\e(x)+\bar G_{n-1}^\e(x)-F_n^\e\left(x,\e^{-1} x \right)
\end{align*}
and
\begin{align*}
\varphi(\e)^{-1}\left(\lambda-\sL_\e\right)\phi_n^\e\left(\cdot,
\e^{-1} \cdot
\right)(x)=
G_{n-1}^\e\left(x,\e^{-1} x \right)-\bar G_{n-1}^\e(x)+F_{n}^\e\left(x,\e^{-1} x \right)-\bar F_{n}^\e(x)
-G_{n}^\e\left(x,\e^{-1} x \right);
\end{align*}
see \eqref{e2-1} and \eqref{e:add--} below.
The correctors $\psi_n^\e$ and $\phi_n^\e$ are introduced to eliminate the lower-order terms $\bar F_{n-1}^\e(x)$ and $G_{n-1}^\e\left(x,\e^{-1} x \right)$
from the previous inductive step, respectively, which will in turn generate new terms
$\bar F_n^\e\left(x\right)$ and $G_n^\e\left(x,\e^{-1} x \right)$
that posses higher-order errors.
Furthermore, according to \eqref{e2-9a}, the error term $\xi_{1,\beta_0}(\e)$ in \eqref{p2-1-1--} corresponds to $\bar F_1^\e(x)$ (which are
$\xi_{3,\beta}(\e)$ and $\xi_{4,\beta}(\e)$ for the critical regime and the supercritical diffusive regime respectively; see \eqref{e3-9} and \eqref{t1-3-4}); on the other hand, by Lemma \ref{l2-2},
the error term $\xi_{2,\beta_0}(\e)$ is related to $G_1^\e\left(x,\e^{-1} x \right)$, which can be viewed the interaction between the fast and slow variables in
$\varphi(\e)^{-1}\sL_\e\phi_1^\e\left(\cdot,
\e^{-1} \cdot
\right)(x)$.
\end{remark}

For $\eps >0$, define
\begin{equation}\label{e1-7a}
\Pi_\e(z):=\left|\varphi(\e)\e^{-d}j \left(\e^{-1} z \right)-\frac{1}{|z|^{d+\alpha}}\right|,\quad z\in \R^d.
\end{equation}
For  $\eps>0$ and $\beta >0$, set
\begin{equation}\label{e1-7}
\begin{split}
&\gamma_{1,\beta}(\e):=\int_{\R^d}\left(1\wedge |z|^2\right)\Pi_\e(z)\,dz+\sup_{x\in \R^d}\left\{(1+|x|)^{d+\beta}\sup_{z\in \R^d:|z|\ge \frac{1+|x|}{2}}\Pi_\e(z)\right\},\\
&\gamma_2(\e):=\varphi(\e)\left(\e^2\int_{\{|z|\le 4\sqrt{d}\}}|z|^2j(z)\,dz+\e^2\int_{\{1<|z|\le \e^{-1}\}}|z|j(z)\,dz+\e\int_{\{|z|>\e^{-1}\}}j(z)\,dz\right),\\
&\gamma_{3,\beta}(\e):=\varphi(\e)
\left(\e^{-d+1}\sup_{x\in \R^d}\left\{(1+|x|)^{d+\beta}\sup_{z\in \R^d:|z|\ge \frac{1+|x|}{2}}j \left(\e^{-1} z \right)\right\}+
\e\int_{\{|z|>\e^{-1}\}}j(z)\,dz\right),\\
&\gamma_4(\e):= 1+\int_{\{1<|z|\le \e^{-1}\}}|z|j(z)\,dz.
\end{split}
\end{equation}
Define
\begin{equation}\label{e1-8a}
\begin{split}
  \xi_{1,\beta}(\e)&=(
\xi_0(\e)
+\gamma_{1,\beta}(\e))(1+\e\gamma_4(\e))+\left(\gamma_2(\e)+\gamma_{3,\beta}(\e)\right)\gamma_4(\e)
+\e\gamma_4(\e)+\e^2\varphi(\e)\gamma_4(\e)^2,\\
\xi_{2,\beta}(\e)&=\varphi(\e)^{-1}\e^{-1}\left(\gamma_2(\e)+\gamma_{3,\beta}(\e)\right)
+\e\gamma_4(\e)+\varphi(\e)^{-1},
\end{split}
\end{equation} where $\xi_0(\xi)$ is defined by \eqref{ex1-1-1}.

\begin{theorem}\label{thm1}
Suppose that Assumption {\rm\ref{a1-1}}{\rm(i)} and Assumption {\rm\ref{a1-2}} hold, and
that there exists $\beta>0$ such that
\begin{equation}\label{t1-1-1}
\lim_{\e \to 0}\xi_{1,\beta}(\e)=0,
\end{equation}
where $\xi_{1,\beta}(\e)$ is defined by \eqref{e1-8a}. Fix this $\beta$.
For any $n\ge1$, let $v_n^\e$ be defined by \eqref{e2-14a} above with $h\in \M_{\beta}$. Then, for any $n\ge1$, there exists a constant $C_0(n)>0$ such that for all $\e>0$,
\begin{equation}\label{p2-1-1}
\begin{split}
\|v_{n}^\e\|_{L^2(\R^d;dx)}\le
C_0(n)\left(\xi_{1,\beta_0}(\e)+\xi_{2,\beta_0}(\e)\right)^n,
\end{split}
\end{equation}
where $\beta_0:=\min\{\alpha,\beta\}$.
\end{theorem}

\begin{remark}\label{r1-2}
\begin{itemize}
\item[{\rm(i)}] Note that $\xi_{1,\beta_1}(\e)\le \xi_{1,\beta_2}(\e)$ and $\M_{\beta_2}\subset \M_{\beta_1}$ for
every $\beta_1\le \beta_2$.
Hence, when \eqref{t1-1-1}
holds for some $\beta>0$, it
holds
for all $\beta'\in (0,\beta]$.
As explained in Remark \ref{r1-3}, the error term $\xi_{1,\beta_0}(\e)$ is induced by
the function $\bar F_1^\e(x)$; see \eqref{e2-9a} below. Therefore, it is natural to assume that \eqref{t1-1-1} is satisfied; otherwise, $\lim_{\e \to 0}\|u_\e-\bar u\|_{L^2(\R^d;dx)}=0$ may not hold true.

\item[{\rm(ii)}] As indicated by Remarks \ref{r1-1} and \ref{r2-1} below, if there is $\theta\in [0,1)$ such that
$$
\int_{\{|z|\le 1\}}|z|^{1+\theta}j(z)\,dz<\infty,
$$
then one can relax Assumption \ref{a1-2}
to
 the following assumption.

\noindent \emph{Assumption {\rm(H):}
For every $f\in C(\T^d)$ with $\int_{\T^d}f(y)\,dy=0$, there exists a unique multivariate 1-periodic
function $\phi_f\in C^ \theta(\T^d)\cap
\D(\sL)$ such that $\sL \phi_f=f$ on $\T^d
$ in the  pointwise sense having $\int_{\T^d}\phi_f(y)\,dy=0$. Moveover,
$\|\phi_f\|_{\theta}\le C_0\|f\|_\infty$,
where the norm $\|\phi_f\|_{\theta}$ is defined by \eqref{r1-1-1a} below and $C_0>0$ is independent of $f$.}

\noindent
Under
Assumption {\rm\ref{a1-1}}{\rm(i)}, Assumption {\rm (H)} and \eqref{t1-1-1}, we can prove
\begin{equation}\label{r1-2-1}
\|v_{n}^\e\|_{L^2(\R^d;dx)}\le
C_0(n)(\xi_{1,\beta_0}(\e)+\tilde \xi_{2,\beta_0}(\e))^n,\ n\ge 1,\ \e\in (0,1),
\end{equation}
where $\tilde \xi_{2,\beta_0}$ is defined by \eqref{r2-1-3} below.
We note that, if there
exist
constants
 $c_0>0$ and $\alpha_0\in (0,1]$ such that $j(z)=\frac{c_0}{|z|^{d+\alpha_0}}$ for all $z\in \R^d$ with $|z|\le 1$, then Assumption {\rm(H)} above holds true.
\end{itemize}
\end{remark}

\subsection{Critical regime}\label{S:2.2}
In this subsection, we assume that Assumption {\rm\ref{a1-1}}{\rm(ii)} and Assumption {\rm\ref{a1-2}} hold.
We also make the following assumption.

\begin{assumption}\label{a2-1} \it
There exist a function $a_0:\R^d \to [0,\infty)$ and a constant $C>0$ such that for all $z_1,z_2\in \R^d$ with $|z_1|\wedge |z_2|\ge 4$ and $|z_1-z_2|\le 1$,
\begin{equation}\label{a2-1-1}
|j(z_1)-j(z_2)|\le Ca_0(z_1)|z_1-z_2|
\end{equation}
\end{assumption}

For every $f\in \M_{\beta}$, $x\in \R^d$ and $y\in \T^d$, define
\begin{equation}\label{e3-3a}
\begin{split}
 \Lambda_\e^3 f(x,y)&=\e^{-d}\varphi(\e)\int_{\R^d}\delta f(x;z)\left(K\left(y,y+\e^{-1} z \right)-\bar K(y)\right)j \left(\e^{-1} z \right)\,dz,\\
\Lambda_\e^4 f(x,y)&=\bar K(y)\left(\e^{-d}\varphi(\e)\int_{\R^d}\delta f(x;z)j \left(\e^{-1} z \right)dz-\frac{1}{2}
\left\langle \nabla^2\bar u(x), A\right\rangle\right),\\
\Theta_\e^2 f(x,y)&=\frac{1}{2}\left\langle \nabla^2 f(x), \bar K(y)A-\bar A\right\rangle,\\
\Upsilon_\e^2 f(x,y)&=\sum_{i=3}^4\Lambda_{\e}^i f(x,y)+\Theta_{\e}^2f(x,y)+
\e\sum_{i=1}^d\Gamma_\e\left(\partial_{x_i}f,\phi_{0,i}^\e\right)\left(x,y\right)\\
&\quad+\e\left\langle \bar \sL_2 \left(\nabla f\right)(x)+\sum_{i=3}^4\Lambda_\e^i\left(\nabla f\right)
\left(x,y\right)+\Theta_\e^2\left(\nabla f\right)
\left(x,y\right), \phi_0^\e\left(y\right)\right\rangle\\
&\quad+\e^{2}\varphi(\e)\left\langle
\nabla^2 f(x), \phi_0^\e\left(y\right)\otimes \Phi_\e\left(y\right)\right\rangle-\lambda\e
\left\langle \nabla f(x), \phi_0^\e(y)\right\rangle,
\end{split}
\end{equation}
where $\delta f(x;z)$ is defined
by \eqref{e1-4}, $A$ is defined by \eqref{a1-1-3}, $\bar A$ is given in
\eqref{e2-1a-}, and some other notations are those given in Subsection \ref{S:2.1}.
In particular, $\Upsilon_\e^2f(x,y)$ for the critical case is defined in the same way as
$\Upsilon_1^\e f(x,y)$ in \eqref{l2-1-3a} for the subcritical $\alpha$-stable regime with the operators $\Lambda_\e^1$, $\Lambda_\e^2$, $\Theta_\e^1$ and $\bar \sL_\alpha$ replaced by
$\Lambda_\e^3$, $\Lambda_\e^4$, $\Theta_\e^2$ and $\bar \sL_2$ respectively.

Let $\bar u$ be the solution to \eqref{e1-3}. Define
\begin{equation}\label{e3-4}
F_1^\e(x,y) =\Upsilon_\e^2 \bar u(x,y),\,\,\bar F_1^\e(x)=\int_{\T^d}F_1^\e(x,y)dy,\quad x\in \R^d,\ y\in \T^d.
\end{equation}
Then, with $F_1^\e$ and $\bar F_1^\e$ as \eqref{e3-4} at hand,  we can define $v_1^\e:\R^d \to \R$, $\psi_1^\e:\R^d\times\T^d\to \R^d$, and, inductively, for all $n\ge1$, define
$G_n^\e:\R^d\times \T^d \to \R$,
$\bar G_n^\e:\R^d \to \R$, $F_{n+1}^\e:\R^d \times \T^d \to \R$,
$\bar F_{n+1}^\e:\R^d \to \R$, $\phi_n^\e:\R^d\times \T^d \to \R$ and $\psi_{n+1}^\e:\R^d \to \R$
by the same way as those for \eqref{e2-3aa}--
\eqref{e2-14}
in  the subcritical $\alpha$-stable regime. The only
difference is that here we use
$\Upsilon_\e^2$ instead of $\Upsilon_\e^1$. With all the notations above, we can define $v^\e_n:\R^d\to \R$ for all $n\ge1$ as \eqref{e2-14a}.

For fixed $\beta>0$, let
\begin{equation}\label{e3-5a}
\begin{split}
&\gamma_5(\e):=\varphi(\e)\left(\e^2\int_{\{1\le |z|\le \e^{-1}\}}|z|^2a_0(z)\,dz+\int_{\{|z|>\e^{-1}\}}a_0(z)\,dz\right),\\
&\gamma_{6,\beta}(\e):=\varphi(\e)\left(\e^{-d}\sup_{x\in \R^d}\left\{(1+|x|)^{d+\beta} \sup_{z\in \R^d:|z|\ge \frac{1+|x|}{2}}a_0\left(\e^{-1} z \right)\right\}
+\int_{\{|z|>\e^{-1}\}}a_0(z)\,dz\right),\\
&\gamma_7(\e):=\left|\varphi(\e)\int_{\{|z|\le \e^{-1}\}}(z\otimes z)j(z)\,dz-A\right|,\\
&\gamma_8(\e):=\varphi(\e)\e^3\int_{\{|z|\le \e^{-1}\}}|z|^3j(z)\,dz,
\end{split}
\end{equation}
and
\begin{equation}\label{e3-6}
\xi_{3,\beta}(\e):=(\gamma_2(\e)+\e^{-1}\gamma_{3,\beta}(\e)
+\gamma_5(\e)+\gamma_{6,\beta}(\e)+\gamma_7(\e)+\gamma_8(\e))(1+\e\gamma_4(\e))+\e\gamma_4(\e)+\e^2\varphi(\e)\gamma_4(\e)^2.
\end{equation}
where $A$ is the $d\times d$ matrix given in \eqref{a1-1-3}, $a_0:\R^d \to \R_+$ is the function given in \eqref{a2-1-1}, and
$\gamma_2(\e)$, $\gamma_{3,\beta}(\e)$ and $\gamma_4(\e)$ are defined in \eqref{e1-7}.

\begin{theorem}\label{thm2}
Suppose that Assumption {\rm\ref{a1-1}}{\rm(ii)}, Assumption {\rm\ref{a1-2}} and Assumption {\rm\ref{a2-1}} hold, and that
there exists $\beta>0$ such that
\begin{equation}\label{t1-2-1}
\lim_{\e \to 0}\xi_{3,\beta}(\e)=0,
\end{equation}
where $\xi_{3,\beta}(\e)$ is defined by \eqref{e3-6}. Fix this $\beta$, and let $v_n:\R^d\to \R$ be defined as \eqref{e2-14a} with $h\in \M_{\beta}$.
Then, for any $n\ge1$, there is $C_0(n)>0$ such that for all $\e\in (0,1)$,
\begin{equation}\label{t1-2-2}
\|v_{n}^\e\|_{L^2(\R^d;dx)}\le
C_0(n)\left(\xi_{2,\beta}(\e)+\xi_{3,\beta}(\e)\right)^n,
\end{equation}
where $\xi_{2,\beta}(\e)$ and $\xi_{3,\beta}(\e)$ are defined by \eqref{e1-8a} and \eqref{e3-6}, respectively.
\end{theorem}

\subsection{Supercritical diffusive regime} \label{S:2.3}

In this subsection, we assume that Assumption {\rm\ref{a1-1}}{\rm(iii)} and Assumption {\rm\ref{a1-2}} hold.
As mentioned before, in this case we take $\varphi(\e)=\e^{-2}$ and write the solution to \eqref{e2-2} as $\phi_0$ (since it is independent of $\e$).
For every $f\in \M_{\beta}$,  set
$$
\hat \delta f(x;z):=f(x+z)-f(x)-\langle \nabla f(x), z\rangle,\quad x,z\in \R^d,
$$
and, for any $x\in \R^d$ and $y\in \T^d$, define
\begin{equation}\label{e2-3-4}
\begin{split}
\Lambda_\e^5 f(x,y)&=\e^{-2}\int_{\R^d}\hat\delta f(x;\e z)K(y,y+z)j(z)\,dz-\frac{1}{2}
\left\langle \nabla^2 f(x), \int_{\R^d}(z\otimes z) K(y,y+z)j(z)\,dz\right\rangle,\\
\Lambda_\e^6 f(x,y)&=\e^{-1}\sum_{i=1}^d\Gamma_\e\left(\partial_{x_i}f, \phi_{0,i}\right)(x,y)-
\left\langle \nabla^2 f(x), \int_{\R^d} z\otimes \left(\phi_0(y+z)-\phi_0(y)\right)K(y,y+z)j(z)\,dz\right\rangle,\\
\Lambda_\e^7 f(x,y)&=\e^{-1}\left\langle\hbox{p.v.}
\int_{\R^d}\left(\nabla f(x+\e z)-\nabla f(x)\right)K(y,y+z)j(z)\,dz, \phi_0(y)\right\rangle -\left\langle \nabla^2 f(x),\Phi_0(y)\otimes \phi_0(y)\right\rangle,\\
\Theta_\e^3 f(x,y)&=\frac{1}{2}\left\langle \nabla^2 f(x), \bar A_0(y)-\bar A_0\right\rangle.
\end{split}
\end{equation}
Here, $\Gamma_\e(f,\phi)$ is defined by
\eqref{e1-6}, $\bar A_0$ by \eqref{e2-3-2}, and
\begin{equation}\label{e2-3-7a}
\bar A_0(y)=\int_{\R^d}(z\otimes z+2 z\otimes \left(\phi_0(y+z)-\phi_0(y)\right)
)K(y,y+z)j(z)\,dz+2\Phi_0(y)\otimes \phi_0(y),\quad y\in \T^d,
\end{equation}
where $\Phi_0$ is defined in \eqref{e2-1a}
(in the supercritical diffusive regime we also write $\Phi_\e$ defined by \eqref{e2-1a} as $\Phi_0$ since it is independent of $\e$),
and $\phi_0$ is the solution to \eqref{e2-2}. With these terms above, set
\begin{equation}\label{e2-3-8}
\Upsilon_\e^3 f(x,y):=\sum_{i=5}^7 \Lambda_\e^i f(x,y)+\Theta_\e^3 f(x,y)-\lambda \e\left\langle \nabla f(x), \phi_0(y)\right\rangle,
\quad x\in \R^d,\ y\in \T^d.
\end{equation}
For the solution $\bar u$ to \eqref{e1-3}, set
\begin{equation}\label{e2-3-9}
F_1^\e(x,y):=\Upsilon_\e^3 \bar u(x,y),\,\, \bar F_1^\e(x):=\int_{\T^d}F_1^\e(x,y)\,dy,\quad x\in \R^d,\ y\in \T^d.
\end{equation}
By Assumption \ref{a1-2} again, we can find a unique function $\phi_1^\e:\R^d \times \T^d \to \R$
such that
\begin{equation}\label{e2-3-10}
-\sL \phi_1^\e(x,\cdot)(y)=F_1^\e(x,y)-\bar F_1^\e(x) \hbox{  for } x\in \R^d,\ y\in \T^d
\hbox{ with }\ \int_{\T^d}\phi_1^\e(x,y)\,dy=0,
\end{equation}
due to the fact $\int_{\T^d}(F_1^\e(x,y)-\bar F_1^\e(x))\,dy=0$.
Now we can define the first order expansion
$$
v_1^\e(x)=u_\e(x)-\bar u(x)-\e\left\langle\nabla \bar u(x),
\phi_0\left(\e^{-1} x \right)\right\rangle-\e^2\phi_1^\e\left(x,\e^{-1} x \right),\quad x\in \R^d.
$$

Furthermore, by induction, for every $n\ge 1$, we can define $G_n^\e:\R^d \times \T^d \to \R$ and $\bar G_n^\e:\R^d \to \R$
by the same way as  \eqref{l2-2-2} with $\Phi_\e=\Phi$ given by \eqref{e2-1a}, and
define
\begin{equation}\label{e2-3-11}
F_{n+1}^\e(x,y)=\Upsilon_3^\e \psi_{n+1}^\e(x,y),\,\,  \bar F_{n+1}^\e(x)=\int_{\T^d} F_{n+1}^\e(x,z)\,dz,\quad x\in \R^d,\ y\in \T^d.
\end{equation}
Here, $\psi_{n+1}^\e$ and $\phi_{n+1}^\e$ are the solutions to
\eqref{e2-13} and \eqref{e2-14} respectively with $\bar \sL_\alpha$ replaced by $\bar \sL_{>2}$.  With all the notations above, we can define $v^\e_n:\R^d\to \R$ for all $n\ge1$ as \eqref{e2-14a}.

For a fixed $\beta>0$, define
\begin{equation}\label{e2-3-12}
 \xi_{4,\beta}(\e)=\e+\e^{-1}\gamma_{3,\beta}(\e)
+\gamma_9(\e),
\end{equation}where $$\gamma_9(\e)=\e\int_{\{|z|\le \e^{-1}\}}|z|^3j(z)\,dz+
\int_{\{|z|>\e^{-1}\}}|z|^2j(z)\,dz.$$

\begin{theorem}\label{thm3}
Suppose that Assumption {\rm\ref{a1-1}(iii)} and Assumption {\rm\ref{a1-2}} hold, and that there exists $\beta>0$ such that
\begin{equation}\label{t1-3-1}
\lim_{\e \to 0}\xi_{4,\beta}(\e)=0.
\end{equation} Fix this $\beta$, and let $v_n:\R^d\to \R$  be defined as \eqref{e2-14a} with $h\in \M_{\beta}$.
For any $n\ge1$, there is a constant $C(n)>0$ such that for all $\e\in (0,1)$
\begin{equation}\label{t1-3-2}
\|v_{n}^\e\|_{L^2(\R^d;dx)}\le
C(n)\left(\xi_{2,\beta}(\e)+\xi_{4,\beta}(\e)\right)^n,
\end{equation}
where $\xi_{2,\beta}(\e)$ and $\xi_{4,\beta}(\e)$ are defined by \eqref{e1-8a} and \eqref{e2-3-12} respectively.
\end{theorem}

\section{Preliminaries}\label{section3}

This section is again split into three parts. In each part, we will establish explicit bounds for $\Lambda_\e^i f(x,y)$ and $\Theta_\e^i f(x,y)$, which are crucial for the proofs of main results in the previous section.

\subsection{Subcritical $\alpha$-stable regime}
Throughout this part, we always assume that Assumption \ref{a1-1}(i) and Assumption \ref{a1-2} hold.
For a fixed $\beta>0$, set $\beta_0:=\alpha\wedge \beta=\min\{\alpha,\beta\}$.

\begin{lemma}\label{l1-1}
Assume that $g\in \M_{\beta}$ for some $\beta>0$, and let $u_{g,\alpha}\in C_b^\infty(\R^d)$ be the
pointwise solution to the following resolvent equation:
$$
\lambda u_{g,\alpha}(x)-\bar \sL_\alpha u_{g,\alpha}(x)=g(x),\quad x\in \R^d,
$$ where $\bar \sL_\alpha$ is defined in \eqref{e2-1a-}.
Then, $u_{g,\alpha}\in \M_{\beta_0}$, and there exists a constant $C_1:=C_1(\beta)>0$ $($which is independent of $g$$)$ such that for all $k\ge0$,
$$
\|u_{g,\alpha}\|_{\M_{\beta_0},k}\le C_1\|g\|_{\M_{\beta},k}.
$$
\end{lemma}
\begin{proof} Let $\{\bar T_t\}_{t\ge 0}$ and $\{\bar p_\alpha(t,x)\}_{t>0,x\in \R^d}$
denote the semigroup and the heat kernel associated with the operator $\bar \sL_\alpha$, respectively.
Since $g\in \M_{\beta}$,  for every $k\ge 0$ and $x\in \R^d$,
\begin{align*}
\nabla^k u_{g,\alpha}(x)&=\int_0^\infty e^{-\lambda t}\nabla^k \bar T_t g(x)\,dt=
\int_{0}^\infty e^{-\lambda t}\nabla^k\left(\int_{\R^d}\bar p_\alpha(t,\cdot-y)g(y)\,dy\right)(x)\,dt\\
&=\int_{0}^\infty e^{-\lambda t} \int_{\R^d}\bar p_\alpha(t,x-y)\nabla^k g(y)\,dy \,dt.
\end{align*}
It is well known that (see e.g. \cite{CK})
$$
\bar p_\alpha(t,x)\le c_1\left(t^{-d/\alpha}\wedge \frac{t}{|x|^{d+\alpha}}\right),\quad  t>0,\ x\in \R^d.
$$
Then, for all $k\ge0$ and $x\in \R^d$,
$$
|\nabla^k  u_{g,\alpha}(x)| \le c_2\|g\|_{\M_{\beta},k}\int_{0}^\infty e^{-\lambda t}
 \int_{\R^d}\left(t^{-d/\alpha}\wedge \frac{t}{|x-y|^{d+\alpha}}\right)(1+|y|)^{-d-\beta}\,dy\,dt.
$$

Set
\begin{align*}
\int_{\R^d}\left(t^{-d/\alpha}\wedge \frac{t}{|x-y|^{d+\alpha}}\right)(1+|y|)^{-d-\beta}\,dy&=\left(\int_{\{|y|\le |x|/2\}}+\int_{\{|y|>|x|/2\}}\right)
\left(t^{-d/\alpha}\wedge \frac{t}{|x-y|^{d+\alpha}}\right)(1+|y|)^{-d-\beta}\,dy\\
&=:I_1(t)+I_2(t).
\end{align*}
Since $|x-y|\ge |x|-|y|\ge |x|/2$ for every $x,y\in \R^d$ with $|y|\le |x|/2$,
\begin{align*}
I_1(t)&\le
\begin{cases}
\displaystyle\int_{\R^d}\left(t^{-d/\alpha}\wedge \frac{t}{|x-y|^{d+\alpha}}\right)dy\le c_3(1+|x|)^{-d-\alpha},\ \ &\ |x|\le 2,\\
c_4t|x|^{-d-\alpha}\displaystyle\int_{\R^d}(1+|y|)^{-d-\beta}dy\le c_5t(1+|x|)^{-d-\alpha},\ \ &\ |x|>2.
\end{cases}
\end{align*}
On the other hand,
\begin{align*}
I_2(t)&\le c_6(1+|x|)^{-d-\beta}\int_{\R^d}\left(t^{-d/\alpha}\wedge \frac{t}{|x-y|^{d+\alpha}}\right)dy
\le c_7(1+|x|)^{-d-\beta}.
\end{align*}
Putting all the estimates above together yields that for all  $k\ge0$ and $x\in \R^d$,
\begin{align*}
|\nabla^k u_{g,\alpha}(x)|&\le c_8\|g\|_{\M_\beta,k}\left[(1+|x|)^{-d-\alpha}+(1+|x|)^{-d-\beta}\right]  \left(\int_0^\infty(1+t)e^{-\lambda t}\,dt\right)\\
&\le c_9\|g\|_{\M_{\beta},k}(1+|x|)^{-d-\beta_0}.
\end{align*}
This proves the desired assertion.
\end{proof}
\begin{lemma}\label{l1-5} Let $\Lambda_\e^1 f(x,y)$  be defined in \eqref{e1-5a}.
For every $\beta>0$, there exists a constant $C_2:=C_2(\beta)>0$ such that for every $f\in \M_\beta$, $x\in \R^d$ and $y\in \T^d$,
\begin{equation}\label{l1-5-1}
|\Lambda_\e^1 f(x,y)|\le C_2
\gamma_{1,\beta}(\e)
\left(\sum_{i=0}^2\|f\|_{\M_{\beta},i}\right)
(1+|x|)^{-d-\beta},
\quad
\end{equation}
where  $\gamma_{1,\beta}(\e)$ is defined in \eqref{e1-7}.
\end{lemma}
\begin{proof}
Applying the mean value theorem, we have
for every $f\in \M_\beta$ and $x,z\in \R^d$,
\begin{equation}\label{l1-5-2}
|\delta f(x;z)| \le c_1\left(\sum_{i=0}^2\|f\|_{\M_{\beta},i}\right)(1+|x|)^{-d-\beta}\left(|z|^2\I_{\{|z|\le 1\}}
+\I_{\{|z|>1\}}\right)  +c_1\|f\|_{\M_{\beta},0}(1+|x+z|)^{-d-\beta}\I_{\{|z|>1\}}.
\end{equation}
Thus, for every $f\in \M_\beta$, $x\in \R^d$ and $y\in \T^d$,
\begin{align*}
|\Lambda_\e^1 f(x,y)|\le& c_2\left(\sum_{i=0}^2\|f\|_{\M_{\beta},i}\right)(1+|x|)^{-d-\beta}\int_{\R^d}\left(1\wedge |z|^2\right)\Pi_\e(z)\,dz+c_2\|f\|_{\M_{\beta},0} \int_{\{|z|>1\}}(1+|x+z|)^{-d-\beta}\Pi_\e(z)\,dz,
\end{align*}
where $\Pi_\e(z)$ is defined by \eqref{e1-7a}. Meanwhile,
\begin{align*}
&\int_{\{|z|>1\}}(1+|x+z|)^{-d-\beta}\Pi_\e(z)\,dz\\
&\le \left(\int_{\{|z|>1,|z|\le ({1+|x|})/{2}\}}+\int_{\{|z|>({1+|x|})/{2}\}}\right)(1+|x+z|)^{-d-\beta}
\Pi_\e(z)\,dz\\
&\le c_3(1+|x|)^{-d-\beta}\int_{\{|z|>1\}}\Pi_\e(z)\,dz+c_3\sup_{z\in \R^d:|z|\ge {(1+|x|)}/{2}}
\Pi_\e(z)\int_{\{|z|> 1\}}(1+|x+z|)^{-d-\beta}\,dz\\
&\le c_4
\gamma_{1,\beta}(\e)
(1+|x|)^{-d-\beta},
\end{align*}
where the last inequality follows from the definition $\gamma_{1,\beta}(\e)$.

Combining the above estimates yields the desired conclusion.
\eqref{l1-5-1}.
\end{proof}

\begin{lemma}\label{l1-2} Let $\Lambda_\e^2 f(x,y)$  be defined in \eqref{e1-5a}. For every $\beta>0$, there exists a positive constant $C_3:=C_3(\beta)$ such that for every $f\in \M_{\beta}$,  $x\in \R^d$ and $y \in \T^d$,
\begin{equation}\label{l1-2-1}
|\Lambda_\e^2 f(x,y)|\le C_3
\xi_0(\e)
\left(\sum_{i=0}^2\|f\|_{\M_{\beta},i}\right)(1+|x|)^{-d-\beta_0},
\end{equation}
where
$\xi_0(\e)$
is defined by \eqref{ex1-1-1}.
\end{lemma}
\begin{proof}
Let $$ H(x,z):=\frac{f(x+z)-f(x)-\langle\nabla f(x), z\I_{\{|z|\le 1\}}\rangle}{|z|^{d+\alpha}},\quad x,z\in \R^d \hbox{ with } z\neq 0,$$
and  $H(x,0)=0$ for $x\in \R^d$.
Then,
\begin{equation}\label{l1-2-4}
\begin{split}
\Lambda_\e^2 f(x,y)&=\sum_{v\in \e\Z^d}\int_{\Pi_{i=1}^d(v_i,v_i+\e]}
H(x,z)\left(K\left(y,y+\e^{-1} z \right)-\bar K(y)\right)\,dz \\
&=\sum_{v\in \e\Z^d}\int_{\Pi_{i=1}^d(v_i,v_i+\e]}
\left(H(x,z)-H(x,v)\right) K\left(y,y+\e^{-1} z \right)\,dz\\
&\quad  +\sum_{v\in \e\Z^d}\int_{\Pi_{i=1}^d(v_i,v_i+\e]}\left(H(x,v)-H(x,z)\right)\bar K(y)\,dz,
\end{split}\end{equation}
where in the second equality we  used the fact that
$$\int_{\Pi_{i=1}^d(v_i,v_i+\e]}K\left(y,y+\e^{-1} z \right)\,dz=
\e^d\bar K(y),\quad \ v\in \e\Z^d. $$

Note that $f\in \M_\beta$. By \eqref{l1-5-2},
for every $x\in \R^d$, $v\in \e\Z^d$
with $|v|\le 4\sqrt{d}\e$ and $z\in \Pi_{i=1}^d(v_i,v_i+\e]$,
\begin{align*}
|H(x,z)-H(x,v)|&\le |H(x,z)|+|H(x,v)|\\
&\le c_1\|f\|_{\M_{\beta},2}(1+|x|)^{-d-\beta}
\left(|z|^{2-d-\alpha}+|v|^{2-d-\alpha}\right).
\end{align*}
On the other hand, for every $x\in \Z^d$, $v\in \e\Z^d$
and $z\in \Pi_{i=1}^d(v_i,v_i+\e]$, we can write
\begin{align*}
 |H(x,z)-H(x,v)|
&\le \left|f(x+z)-f(x)-\langle \nabla f(x),z\I_{\{|z|\le 1\}}\rangle\right|\left|\frac{1}{|z|^{d+\alpha}}-\frac{1}{|v|^{d+\alpha}}\right|\\
&\quad +\frac{1}{|v|^{d+\alpha}} \left|f(x+z)-f(x+v)
-\langle \nabla f(x), z\I_{\{|z|\le 1\}}
-v\I_{\{|v|\le 1\}}\rangle\right|\\
&=:I_1(x,z,v)+I_2(x,z,v).
\end{align*}
According to \eqref{l1-5-2} and the mean value theorem, we deduce that
for every $x\in \Z^d$,  $v\in \e\Z^d$ with $|v|>4\sqrt{d}\e$
and $z\in \Pi_{i=1}^d(v_i,v_i+\e]$,
\begin{align*}
 I_1(x,z,v)
&\le  \frac{c_2\e}{|z|^{d+\alpha+1}}\left(\sum_{i=0}^2\|f\|_{\M_{\beta},i}\right)
\Big[(1+|x|)^{-d-\beta}\left(|z|^2\I_{\{2\sqrt{d}\e<|z|\le 1\}}+\I_{\{|z|>1\}}\right)+(1+|x+z|)^{-d-\beta}\I_{\{|z|>1\}}\Big]
\end{align*} and
\begin{align}
I_2(x,z,v)&\le \frac{c_3\e}{|z|^{d+\alpha}}\Bigg[\left(\int_0^1\left|\nabla f(x+z+s(v-z))-\nabla f(x)\right|ds\right)\I_{\{|v|\le 1-4\sqrt{d}\e\}}\nonumber\\
&\quad \quad\qquad\,\,+\left(\int_0^1\left|\nabla f(x+z+s(v-z))\right|ds\right)\I_{\{|v|>1+4\sqrt{d}\e\}}\Bigg]\nonumber\\
&\quad+ \frac{c_3}{|z|^{d+\alpha}}\left(|f(x+v)|+|f(x+z)|+|\nabla f(x)|\right) \I_{\{1-4\sqrt{d}\e<|v|\le 1+4\sqrt{d}\e\}}\nonumber\\
&\le \frac{c_4\e}{|z|^{d+\alpha}}\left(\sum_{i=0}^2\|f\|_{\M_{\beta},i}\right)
\Bigg[(1+|x|)^{-d-\beta}|z|\I_{\{2\sqrt{d}\e<|z|\le 1-2\sqrt{d}\e\}} +(1+|x+z|)^{-d-\beta}\I_{\{|z|>1+2\sqrt{d}\e\}}\Bigg]\nonumber\\
&\quad+ \frac{c_4}{|z|^{d+\alpha}}\left(\sum_{i=0}^1\|f\|_{\M_{\beta},i}\right)(1+|x|)^{-d-\beta} \I_{\{1-6\sqrt{d}\e<|z|\le 1+6\sqrt{d}\e\}}.\label{eq:eohbbsd}
\end{align}

Furthermore, it holds that
\begin{equation}\label{l1-2-5}
\begin{split}
&\int_{\{|z|>1\}}|z|^{-d-\alpha}(1+|x+z|)^{-d-\beta}\,dz\\
&\le \left(\int_{\{|z|>1,|z|\le ({1+|x|})/{2}\}}+\int_{\{|z|>({1+|x|})/{2}\}}\right)|z|^{-d-\alpha}(1+|x+z|)^{-d-\beta}\,dz\\
&\le c_5(1+|x|)^{-d-\beta}\left(\int_{\{|z|>1\}}|z|^{-d-\alpha}\,dz\right)+c_5(1+|x|)^{-d-\alpha}\left(\int_{\R^d}(1+|x+z|)^{-d-\beta}\,dz\right)\\
&\le c_6\left[(1+|x|)^{-d-\alpha}+(1+|x|)^{-d-\beta}\right]\le c_7(1+|x|)^{-d-\beta_0}.
\end{split}
\end{equation}

Noting that $\xi_0(\e)$ appears by integrating \eqref{eq:eohbbsd},
we obtain the assertion \eqref{l1-2-1} by putting all the estimates above together
into \eqref{l1-2-4}.
\end{proof}

\begin{remark}\label{r2-2}
Intuitively speaking, the estimate for $\Lambda_\e^2 f$ in \eqref{l1-2-1} reflects the quantitative homogenized
error,  in terms of the coefficient
$K(y,y+z)$,
for the variable $z$ which describes the jump-size of the associated
Feller process $(X^\e_t)_{t\ge0}$. According to the proof, we know that such estimate may become
slower, when $f$ is less regular. Below, we take $\alpha\in (1,2)$ for example.  Suppose that $\|f\|_{\M_{\beta},0}+\|f\|_{\M_{\beta},1}<\infty$, and that there
exist a constant $\theta\in (\alpha-1,1)$ such that
\begin{equation}\label{r2-2-1}
|\nabla f(x)-\nabla f(x+z)|\le c_1(1+|x|)^{-d-\beta}|z|^{\theta},\quad x\in \R^d,\ z\in \R^d\ {\rm with}\ |z|\le 1,
\end{equation} where the constant $c_1$ may depend on $f$.
Then, by the proof of Lemma \ref{l1-2}, we can prove that for every $x\in \R^d$, $v\in \e\Z^d$
with $|v|\le 4\sqrt{d}\e$ and $z\in \Pi_{i=1}^d(v_i,v_i+\e]$,
$$
|H(x,z)-H(x,v)|
\le c_2(1+|x|)^{-d-\beta}
\left(|z|^{1+\theta-d-\alpha}+|v|^{1+\theta-d-\alpha}\right);$$ that for every $x\in \Z^d$,  $v\in \e\Z^d$ with $|v|>4\sqrt{d}\e$
and $z\in \Pi_{i=1}^d(v_i,v_i+\e]$,
\begin{align*}
&\quad I_1(x,z,v)
\le  \frac{c_2\e}{|z|^{d+\alpha+1}}
\Big[(1+|x|)^{-d-\beta}\left(|z|^{1+\theta}\I_{\{2\sqrt{d}\e<|z|\le 1\}}+\I_{\{|z|>1\}}\right)+(1+|x+z|)^{-d-\beta}\I_{\{|z|>1\}}\Big]
\end{align*} and
\begin{align*}
I_2(x,z,v)
&\le \frac{c_2\e}{|z|^{d+\alpha}}
\Bigg[(1+|x|)^{-d-\beta}|z|^\theta\I_{\{2\sqrt{d}\e<|z|\le 1-2\sqrt{d}\e\}} +(1+|x+z|)^{-d-\beta}\I_{\{|z|>1+2\sqrt{d}\e\}}\Bigg]\\
&\quad+ \frac{c_3}{|z|^{d+\alpha}}(1+|x|)^{-d-\beta}\I_{\{1-6\sqrt{d}\e<|z|\le 1+6\sqrt{d}\e\}},
\end{align*}
where $H(x,z)$, $I_1(x,z,v)$ and $I_2(x,z,v)$ are defined by the way as those in the proof of Lemma \ref{l1-2}.
With these estimates at hand, we can get that for all $x\in \R^d$ and $y \in \T^d$,
\begin{equation*}
|\Lambda_\e^2 f(x,y)|\le c_4\e^{1+\theta-\alpha}(1+|x|)^{-d-\beta_0},
\end{equation*}
which is worse than \eqref{l1-2-1}.
We note that the constants $c_i$ ($2\le i\le 4$) here may depend on $\|f\|_{\M_{\beta},0}$, $\|f\|_{\M_{\beta},1}$ and $c_1$.
We also note that, though the estimate for $\Lambda_\e^1 f$ in \eqref{l1-5-1} also partly indicates the quantitative homogenized
error
for the variable $z$ describing the jump-size, it essentially only depends on the density function $j(z)$ not the coefficient
$K(y,y+z)$.
\end{remark}

\begin{lemma}\label{l1-4} Let $\Theta_\e^1 f(x,y)$  be defined in \eqref{e1-5a}. For any $\beta>0$, there exists a positive constant $C_4:=C_4(\beta)$ such that for every $f\in \M_\beta$,  $x\in \R^d$ and $y\in \T^d$,
\begin{equation}\label{l1-4-1}
|\Theta_\e^1 f(x,y)|\le C_4\left(\sum_{k=0}^2\|f\|_{\M_{\beta},k}\right)(1+|x|)^{-d-\beta_0}.
\end{equation}
\end{lemma}
\begin{proof}
According to \eqref{l1-5-2} and \eqref{l1-2-5}, we can prove \eqref{l1-4-1} directly.
\end{proof}

\begin{lemma}\label{l1-3} Let $\Gamma_\e^1(f,\phi)(x,y)$  be defined by \eqref{e1-6}.
For every $\beta>0$, there exists a constant $C_5:=C_5(\beta)>0$ such that for every $f\in \M_\beta$, $\phi\in C^1(\T^d)$, $x\in \R^d$ and $y\in \T^d$,
\begin{equation}\label{l1-3-1}
 |\Gamma_\e(f,\phi)(x,y)|\le C_5\e^{-1}\left(\gamma_2(\e)+\gamma_{3,\beta}(\e)\right)\left(\sum_{k=0}^1\|f\|_{\M_{\beta},k} \right)
\left(\|\phi\|_\infty+\|\nabla \phi\|_\infty\right)(1+|x|)^{-d-\beta},
\end{equation} where $\gamma_2(\e)$ and $\gamma_{3,\beta}(\e)$ are defined  in \eqref{e1-7}.
\end{lemma}
\begin{proof}
By the
change of variables,
we find
\begin{align*}
\Gamma_\e(f,\phi)(x,y)=\varphi(\e)\int_{\R^d}\left(f(x+\e z)-f(x)\right)
\left(\phi\left(y+z\right)-\phi\left(y\right)\right)K\left(y,y+z\right)j(z)\,dz.
\end{align*}
By the mean value theorem, we have
\begin{align*}
\left|f(x+\e z)-f(x)\right|&\le c_1
\Big(\e\|f\|_{\M_{\beta},1}(1+|x|)^{-d-\beta}|z|\I_{\{|z|\le \e^{-1}\}}\\
&\quad\quad\quad+
\|f\|_{\M_{\beta},0}\left((1+|x|)^{-d-\beta}+(1+|x+\e z|)^{-d-\beta}\right)\I_{\{|z|>\e^{-1}\}}\Big)
\end{align*} and
\begin{align}\label{l1-3-4}
\left|\phi(y+z)-\phi(y)\right|&\le c_1
\left(\|\nabla \phi\|_\infty |z|\I_{\{|z|\le 1\}}+\|\phi\|_\infty \I_{\{|z|>1\}}\right).
\end{align}
Then,
\begin{align*}
 |\Gamma_\e(f,\phi)(x,y)|
&\le   c_2\varphi(\e)
\left(\|f\|_{\M_{\beta},0}+\|f\|_{\M_{\beta},1}\right)
\left(\|\phi\|_\infty+\|\nabla \phi\|_\infty\right)\\
&\quad\times\Bigg[(1+|x|)^{-d-\beta}\bigg(\e\int_{\{|z|\le 1\}}|z|^{2}j(z)\,dz
+\e\int_{\{1<|z|\le \e^{-1}\}}|z|j(z)\,dz+
\int_{\{|z|>\e^{-1}\}}j(z)\,dz\bigg)\\
&\quad\quad\quad +\int_{\{|z|>\e^{-1}\}}j(z)(1+|x+\e z|)^{-d-\beta}\,dz\Bigg].
\end{align*}

Furthermore, we get
\begin{equation}\label{l1-3-5}
\begin{split}
&\varphi(\e)\int_{\{|z|>\e^{-1}\}}j(z)(1+|x+\e z|)^{-d-\beta}\,dz\\
&=\varphi(\e)
\e^{-d}\int_{\{|z|>1\}}j \left(\e^{-1} z \right)(1+|x+z|)^{-d-\beta}\,dz\\
&\le \varphi(\e)
\e^{-d}\left(\int_{\{|z|>1,|z|\le ({1+|x|})/{2}\}}+\int_{\{|z|>({1+|x|})/{2}\}}\right)(1+|x+z|)^{-d-\beta}j \left(\e^{-1} z \right)\,dz\\
&\le c_3(1+|x|)^{-d-\beta}\Bigg[\varphi(\e)\int_{\{|z|>\e^{-1}\}}j(z)\,dz\\
&\qquad\quad+\varphi(\e)\e^{-d}\sup_{x\in \R^d}\left\{(1+|x|)^{d+\beta} \sup_{z\in \R^d:|z|\ge \frac{1+|x|}{2}}j \left(\e^{-1} z \right)\right\}\int_{\{|z|>1\}}(1+|x+z|)^{-d-\beta}dz\Bigg]\\
&\le c_4(1+|x|)^{-d-\beta}\e^{-1}\gamma_{3,\beta}(\e),
\end{split}
\end{equation}
where the last  inequality follows from the definition of $\gamma_{3,\beta}(\e)$.

Hence, combining the above estimates
yields the desired conclusion \eqref{l1-3-1}.
\end{proof}

\begin{remark}\label{r1-1}
Let
\begin{equation}\label{r1-1-1a}
\|\phi\|_{\theta}:=\begin{cases}\sup_{x,y\in \T^d}\frac{\left|\phi(x)-\phi(y)\right|}{|x-y|^{\theta}},\quad &\theta\in(0,1),\\
\|\phi\|_\infty=\sup_{x\in \T^d}|\phi(x)|,\quad& \theta=0.\end{cases}
\end{equation}
Using
$
\left|\phi(y+z)-\phi(y)\right|\le 2\|\phi\|_\infty
$
for $\theta=0$ or
$
\left|\phi(y+z)-\phi(y)\right|\le \|\phi\|_{\theta}|z|^\theta
$
for $\theta\in (0,1)$ in the place of \eqref{l1-3-4}, and following the proof of Lemma \ref{l1-3}, we can get that, \emph{if there is a constant $\theta\in [0,1)$ so that
\begin{equation}\label{r1-1-1}
\int_{\{|z|\le 1\}}|z|^{1+\theta}j(z)\,dz<\infty,
\end{equation} then, for every $\beta>0$, there exists a constant $C_6:=C_6(\beta)>0$ such that for every $f\in \M_\beta$, $\phi\in C^1(\T^d)$, $x\in \R^d$ and $y\in \T^d$,
\begin{equation}\label{r1-1-2}
|\Gamma_\e(f,\phi)(x,y)|\le  C_6\e^{-1}\left(\tilde \gamma_2(\e)+\gamma_{3,\beta}(\e)\right)\left(\sum_{i=0}^1\|f\|_{\M_{\beta},i} \right)
(1+|x|)^{-d-\beta}
\begin{cases}
 \|\phi\|_\infty+\|\phi\|_\theta ,\ &\ \theta\in (0,1),\\
\|\phi\|_\infty,\ &\ \theta=0,
\end{cases}
\end{equation}
where
\begin{equation}\label{r1-1-3}
\tilde \gamma_2(\e):=\varphi(\e) \left(\e^2\int_{\{|z|\le 1\}}|z|^{1+\theta}j(z)\,dz+\e^2\int_{\{1<|z|\le \e^{-1}\}}|z|j(z)\,dz+\e\int_{\{|z|>\e^{-1}\}}j(z)\,dz\right).
\end{equation}}
\end{remark}

Then, we have the following statement.
\begin{lemma}\label{l1-6}
Suppose that \eqref{t1-1-1} holds for some $\beta>0$. Let $\Upsilon_\e^1 f(x,y)$ be defined by \eqref{l2-1-3a}. Then there exists a positive constant $C_7:=C_7(\beta)$ such that for every $f\in \M_\beta$, $x\in \R^d$ and $y\in \T^d$,
\begin{equation}\label{l1-6-1}
|\Upsilon_\e^1 f(x,y)|\le C_7\left(\sum_{i=0}^2\|f\|_{\M_{\beta},i}\right)(1+|x|)^{-d-\beta_0}.
\end{equation}
\end{lemma}
\begin{proof}
Let $\Phi_\e(y)$ be defined in  \eqref{e2-1a}. It is clear that there is a constant $c_1>0$ such that
\begin{equation}\label{e:PHI}
\|\Phi_\e\|_\infty\le c_1\left(1+\int_{\{1<|z|\le \e^{-1}\}}|z|j(z)\,dz\right)=  c_1\gamma_4(\e).
\end{equation}
Thus, by \eqref{e2-2a}, the solution $\phi_0^\e$ to \eqref{e2-2} satisfies
\begin{equation}\label{l1-6-3}
\|\phi_0^\e\|_\infty+\|\nabla \phi_0^\e\|\le c_2\|\Phi_\e\|_\infty
\le c_3\gamma_4(\e).
\end{equation}

According to \eqref{l1-5-1}, \eqref{l1-2-1},  \eqref{l1-4-1}, \eqref{l1-3-1} and \eqref{l1-6-3}, we have
\begin{equation}\label{l1-6-2}
\begin{split}
|\Upsilon_\e^1 f(x,y)|&\le c_4\left[|\Theta_\e^1 f(x,y)|+\left(\sum_{i=0}^{2}\|f\|_{\M_{\beta},i}\right)
\xi_{1,\beta}(\e)(1+|x|)^{-d-\beta_0}\right]\\
&\le c_5\left(\sum_{i=0}^{2}\|f\|_{\M_{\beta},i}\right)\left(1+\xi_{1,\beta}(\e)\right)(1+|x|)^{-d-\beta_0}\\
&\le c_6\left(\sum_{i=0}^{2}\|f\|_{\M_{\beta},i}\right)(1+|x|)^{-d-\beta_0},
\end{split}
\end{equation}
where in the first inequality we used the definition of $\xi_{1,\beta}(\e)$, and the last inequality follows from  \eqref{t1-1-1}.
\end{proof}

\subsection{Critical regime}
In this subsection, we suppose that Assumption \ref{a1-1}(ii), Assumption \ref{a1-2} and Assumption \ref{a2-1} hold.

\begin{lemma}\label{l3-1}
Let $g\in \M_{\beta}$ for some $\beta>0$, and let $u_g\in C_b^\infty(\R^d)$ be the solution to following equation:
$$
\lambda u_g(x)-\bar \sL_2 u_g(x)=g(x),\quad x\in \R^d,
$$ where $\sL_2$ is given by \eqref{e2-1a-}.
Then, $u_g\in \M_{\beta}$, and there exists a constant $C_1:=C_1(\beta)>0$ $($which is independent of $g$$)$ such that for all $k\ge0$,
$$
\|u_g\|_{\M_{\beta},k}\le C_1\|g\|_{\M_{\beta},k} .
$$
\end{lemma}
\begin{proof}
As explained in the proof of Lemma \ref{l1-1}, it holds that for all $k\ge0$ and $x\in \R^d$,
$$
\nabla^k u_g(x)=\int_{0}^\infty e^{-\lambda t} \int_{\R^d}\bar p(t,x-y)\nabla^k g(y)\,dy \,dt,
$$
where $\{\bar p(t,x)\}_{t>0,x\in \R^d}$ denotes the heat kernel associated with $\bar \sL_2$.

Since
$$
\bar p(t,x)\le c_1t^{-d/2}\exp\left(-\frac{c_2|x|^2}{t}\right),\quad t>0,\ x\in \R^d,
$$ for all $x\in \R^d$ with $|x|\ge2$,
\begin{align*}
\int_0^\infty t^{-d/2} \exp \left(-\frac{c_2|x|^2}{t}\right) e^{-\lambda t}\,dt\le &c_3\int_{|x|^2}^\infty t^{-d/2}e^{-\lambda t}\,dt+c_3|x|^{-d}\int_0^{|x|^2} \exp \left(-\frac{c_4|x|^2}{t}\right) e^{-\lambda t}\,dt\\
\le &c_5 \exp(-c_6|x|)\le c_7(1+|x|)^{-\beta}.
\end{align*}
With this at hand, one can follow the proof of Lemma \ref{l1-1} and  prove the desired conclusion.
\end{proof}
\begin{remark}\label{r3-1a}
 By Lemma \ref{l5-2} below,
we know that the $d\times d$ matrix $\bar A_0$ defined by \eqref{e2-3-2} is strictly positive definite
 under
Assumption {\rm\ref{a1-1}(iii)} and Assumption {\rm\ref{a1-2}}.
Then, according to the proof of Lemma \ref{l3-1} and the expression of $\bar \sL_{>2}$, the assertion of Lemma \ref{l3-1} still holds true with $\bar \sL_{>2}$ in place of $\bar \sL_2$.   \end{remark}

\begin{lemma}\label{l3-2}  Let $\Lambda_\e^3 f(x,y)$ be defined in \eqref{e3-3a}.
For every $\beta>0$, there exists a positive constant $C_2:=C_2(\beta)$ such that for every $f\in \M_{\beta}$, $x\in \R^d$ and $y \in \T^d$,
\begin{equation}\label{l3-2-1}
|\Lambda_\e^3 f(x,y)|\le C_2\left(\gamma_2(\e)+\e^{-1}\gamma_{3,\beta}(\e)+\gamma_{5}(\e)+\gamma_{6, \beta}(\e)\right)\left(\sum_{i=0}^2\|f\|_{\M_{\beta},i}\right)(1+|x|)^{-d-\beta},
\end{equation} where $\gamma_2(\e)$ and $\gamma_{3,\beta}(\e)$ are defined in \eqref{e1-7}, and $\gamma_{5}(\e)$ and $\gamma_{6, \beta}(\e)$ are defined in \eqref{e3-5a}.
\end{lemma}
\begin{proof}
For  $\e>0$, set
$$
H_\e(x,z):=\left(f(x+z)-f(x)-\langle\nabla f(x), z\I_{\{|z|\le 1\}}\rangle\right)j \left(\e^{-1} z \right),\quad x,z\in \R^d \hbox{ with }z\neq0,
$$
and $H_\e(x,0)=0$ for every $x\in \R^d$.

Following the argument of \eqref{l1-2-4}, we can get
\begin{align*}
\Lambda_\e^3 f(x,y)
&=\e^{-d}\varphi(\e)\sum_{v\in \e\Z^d}\int_{\Pi_{i=1}^d(v_i,v_i+\e]}
\left(H_\e(x,z)-H_\e(x,v)\right) K\left(y,y+\e^{-1} z \right)\,dz\\
&\quad + \e^{-d}\varphi(\e)\sum_{v\in \e\Z^d}\int_{\Pi_{i=1}^d(v_i,v_i+\e]}\left(H_\e(x,v)-H_\e(x,z)\right)\bar K(y)\,dz.
\end{align*}
By the mean value theorem, it holds that
for every $x\in \R^d$, $v\in \e\Z^d$ with
$|v|\le 4\sqrt{d}\e$ and $z\in \Pi_{i=1}^d(v_i,v_i+\e]$,
\begin{align*}
|H_\e(x,z)-H_\e(x,v)|&\le |H_\e(x,z)|+|H_\e(x,v)|\\
&\le c_1\|f\|_{\M_{\beta},2}(1+|x|)^{-d-\beta}
\left(|z|^2j (\e^{-1} z )+|v|^{2}j (\e^{-1} v)\right).
\end{align*}
On the other hand, for every $x\in \Z^d$, $v\in \e\Z^d$
and $z\in \Pi_{i=1}^d(v_i,v_i+\e]$,
\begin{align*}
 |H_\e(x,z)-H_\e(x,v)|
&\le \left|f(x+v)-f(x)-\langle \nabla f(x),v\I_{\{|v|\le 1\}}\rangle\right||j (\e^{-1} z)-j(\e^{-1} v)|\\
&\quad +j \left(\e^{-1} z \right) \left|f(x+z)-f(x+v)-\langle \nabla f(x), z\I_{\{|z|\le 1\}}-v\I_{\{|v|\le 1\}}\rangle\right|\\
&=:I_{1,\e}(x,z,v)+I_{2,\e}(x,z,v).
\end{align*}
Thus, by the mean value theorem and \eqref{a2-1-1}, we obtain that
for every $x\in \Z^d$,  $v\in \e\Z^d$ with $|v|>4\sqrt{d}\e$
and $z\in \Pi_{i=1}^d(v_i,v_i+\e]$,
\begin{align*}
& I_{1,\e}(x,z,v)\\
&\le c_2a_0\left(\e^{-1} z \right)
\bigg(\|f\|_{\M_{\beta},2}(1+|x|)^{-d-\beta}|v|^2\I_{\{4\sqrt{d}\e<|v|\le 1\}}+
\left(|f(x+z)|+|f(x)|\right)\I_{\{|v|>1\}}\bigg)\\
&\le c_3a_0\left(\e^{-1} z \right)\left(\sum_{i=0}^2\|f\|_{\M_{\beta},i}\right)
\Big((1+|x|)^{-d-\beta}\left(|z|^2\I_{\{2\sqrt{d}\e<|z|\le 1\}}+\I_{\{|z|>1\}}\right)+(1+|x+z|)^{-d-\beta}\I_{\{|z|>1\}}\Big).
\end{align*}
Following the arguments for \eqref{eq:eohbbsd} and \eqref{l1-3-5}, we derive that
for every $x\in \Z^d$,  $v\in \e\Z^d$ with $|v|>4\sqrt{d}\e$
and $z\in \Pi_{i=1}^d(v_i,v_i+\e]$,
\begin{align*}
I_{2,\e}(x,z,v)
&\le c_4\e j \left(\e^{-1} z \right)\left(\sum_{i=0}^2\|f\|_{\M_{\beta},i}\right)\Bigg((1+|x|)^{-d-\beta}|z|\I_{\{2\sqrt{d}\e<|z|\le 1-2\sqrt{d}\e\}}+(1+|x+z|)^{-d-\beta}\I_{\{|z|>1+2\sqrt{d}\e\}}\Bigg)\\
&\quad +c_4  j \left(\e^{-1} z \right)\left(\sum_{i=0}^1\|f\|_{\M_{\beta},i}\right)(1+|x|)^{-d-\beta}\I_{\{1-6\sqrt{d}\e\le |z|\le 1+6\sqrt{d}\e\}},
\end{align*} and
\begin{equation*}
\varphi(\e)\int_{\{|z|>\e^{-1}\}}j(z)(1+|x+\e z|)^{-d-\beta}dz
\le c_5(1+|x|)^{-d-\beta}\e^{-1}\gamma_{3,\beta}(\e)
\end{equation*}
as well as
$$
 \varphi(\e)\int_{\{|z|>\e^{-1}\}}a_0(z)(1+|x+\e z|)^{-d-\beta}dz
\le c_5(1+|x|)^{-d-\beta}\gamma_{6,\beta}(\e).
$$

Hence, putting all the estimates together yields that
\begin{align*}
\left|\Lambda_\e^3 f(x,y)\right|& \le c_6\left(\sum_{i=0}^2\|f\|_{\M_{\beta},i}\right)(1+|x|)^{-d-\beta}
\left(\gamma_2(\e)+\varepsilon^{-1}\gamma_{3,\beta}(\e)+\gamma_{5}(\e)+\gamma_{6,\beta}(\e)\right).
\end{align*}
The  proof is complete.
\end{proof}

\begin{lemma}\label{l3-3} Let $\Lambda_\e^4 f(x,y)$ be defined in \eqref{e3-3a}.
For every $\beta>0$, there exists a positive constant $C_3:=C_3(\beta)$ such that for every $f\in \M_\beta$, $x\in \R^d$ and $y\in \T^d$,
\begin{equation}\label{l3-3-1}
|\Lambda_\e^4 f(x,y)|\le C_3\left(\e^{-1}\gamma_{3,\beta}(\e)+\gamma_7(\e)+\gamma_8(\e)\right)\left(\sum_{i=0}^3\|f\|_{\M_{\beta},i}\right)
(1+|x|)^{-d-\beta}.
\end{equation}
\end{lemma}
\begin{proof}
By the
change of variables,
\begin{align*}
\varphi(\e)\e^{-d}\int_{\R^d}\delta f(x;z)j \left(\e^{-1} z \right)dz=
\varphi(\e)\int_{\R^d}\delta f(x;\e z)j \left(z\right)dz.
\end{align*}
Applying the Taylor expansion, we find that for every $x,z\in \R^d$,
\begin{align*}
\delta f(x;\e z)=
\begin{cases}
\frac{1}{2}\e^2\left\langle \nabla^2 f(x), z\otimes z\right\rangle
+J_\e(f)(x;z),\ &  |z|\le \e^{-1},\\
f(x+\e z)-f(x),\ &   |z|> \e^{-1},
\end{cases}
\end{align*}
where the remaining term $J_\e(f)(x;z)$ satisfies that
for every $x,z\in \R^d$ with $|z|\le \e^{-1}$,
\begin{equation}\label{l3-3-2}
|J_\e(f)(x;z)|\le \frac{\e^3}{6}\left(\sup_{x'\in \R^d:|x'-x|\le 1}|\nabla^3 f(x')|\right)|z|^3
\le c_1\e^3\|f\|_{\M_{\beta},3}(1+|x|)^{-d-\beta}|z|^3.
\end{equation}
Combining the above estimates yields that
\begin{align*}
\varphi(\e)\e^{-d}\int_{\R^d}\delta f(x;z)j \left(\e^{-1} z \right)\,dz
&=\frac{1}{2}\varphi(\e)\e^2\left\langle \nabla^2 f(x),
\int_{\{|z|\le \e^{-1}\}}(z\otimes z)j(z)\,dz\right\rangle
+\varphi(\e)\int_{\{|z|\le \e^{-1}\}}J_\e(f)(x;z)j(z)\,dz\\
&\quad +\varphi(\e)\int_{\{|z|>\e^{-1}\}}\left(f(x+\e z)-f(x)\right)j(z)\,dz\\
&=:\frac{1}{2}\varphi(\e)\e^2\left\langle \nabla^2 f(x),
\int_{\{|z|\le \e^{-1}\}}(z\otimes z)j(z)\,dz\right\rangle+I_{1,\e}(f)(x)+I_{2,\e}(f)(x).
\end{align*}

Furthermore, according to \eqref{l3-3-2} and \eqref{l1-3-5}, it is easy to see that
$$
|I_{1,\e}(f)(x)| \le c_2\gamma_8(\e)\|f\|_{\M_{\beta},3}(1+|x|)^{-d-\beta}
$$ and
$$
|I_{2,\e}(f)(x)|\le c_2\e^{-1}\gamma_{3,\beta}(\e)\|f\|_{\M_{\beta},0}(1+|x|)^{-d-\beta}.
$$
Therefore,
\begin{align*}
|\Lambda_\e^4 f(x,y)|
&=\left|\bar K(y) \left(\varphi(\e)\e^{-d}\int_{\R^d}\delta f(x;z)j \left(\e^{-1} z \right)dz
-\frac{1}{2}\left\langle \nabla^2 f(x), A\right\rangle\right)\right|\\
&\le \left|\frac{1}{2}\left\langle \nabla^2 f(x), \varphi(\e)\e^2\int_{\{|z|\le \e^{-1}\}}(z\otimes z)j(z)\,dz-A \right\rangle\right|
+\left(\sup_{y\in \T^d}\bar K(y)\right)\left(|I_{1,\e}(f)(x)|+|I_{2,\e}(f)(x)|\right)\\
&\le c_3\left(\e^{-1}\gamma_{3,\beta}(\e)+\gamma_7(\e)+\gamma_8(\e)\right)
\left(\sum_{i=0}^3\|f\|_{\M_{\beta},i}\right)(1+|x|)^{-d-\beta}.
\end{align*}
We then complete the proof.
\end{proof}

\begin{lemma}\label{l3-4}
Suppose that \eqref{t1-2-1} holds for some $\beta>0$. Let $\Upsilon_\e^2 f(x,y)$ be defined in \eqref{e3-3a}. Then there exists a positive constant $C_4:=C_4(\beta)$ such that for every $f\in \M_\beta$, $x\in \R^d$ and $y\in \T^d$,
\begin{equation}\label{l3-4-1}
|\Upsilon_\e^2 f(x,y)|\le C_4\left(\sum_{i=0}^3\|f\|_{\M_{\beta},i}\right)(1+|x|)^{-d-\beta}.
\end{equation}
\end{lemma}
\begin{proof}
By the proof of Lemma \ref{l1-3}, we know that  \eqref{l1-3-1} still holds true for the critical regime. That is,
for every $\beta>0$, there exists a positive constant $c_1:=c_1(\beta)$ such that for
every $f\in \M_\beta$, $\phi\in C^1(\T^d)$, $x\in \R^d$ and $y\in \T^d$,
$$
\left|\Gamma_\e(f,\phi)(x,y)\right| \le c_1\e^{-1}\left(\gamma_2(\e)+\gamma_{3,\beta}(\e)\right)
\left(\sum_{i=0}^1
\|f\|_{\M_\beta,i}
\right)(1+|x|)^{-d-\beta}
\left(\|\phi\|_\infty+\|\nabla \phi\|_\infty\right).
$$
On the other hand, it is easily seen from the definition of $\Theta_\e^2 f(x,y)$ that
for every $f\in \M_\beta$, $x\in \R^d$ and $y\in \T^d$,
$$
\left|\Theta_\e^2 f(x,y)\right|\le c_2
\|f\|_{\M_\beta,2}(1+|x|)^{-d-\beta}.
$$
Then, putting both estimates above, \eqref{l1-6-3}, \eqref{l3-2-1} and \eqref{l3-3-1} together yields that                                                                                                                             \begin{align*}
|\Upsilon_\e^2 f(x,y)|&\le c_3\left[|\Theta_\e^2 f(x,y)|+\left(\sum_{i=0}^{3}\|f\|_{\M_{\beta},i}\right)
\xi_{3,\beta}(\e)(1+|x|)^{-d-\beta}\right]\\
&\le c_4\left(\sum_{i=0}^{3}\|f\|_{\M_{\beta},i}\right)\left(1+\xi_{3,\beta}(\e)\right)(1+|x|)^{-d-\beta}\\
&\le c_5\left(\sum_{i=0}^{3}\|f\|_{\M_{\beta},i}\right)(1+|x|)^{-d-\beta},
\end{align*}
where the last inequality follows from  \eqref{t1-2-1}. The proof is complete.
\end{proof}

\subsection{Supercritical diffusive regime}
In this part we suppose that Assumption \ref{a1-1}(iii) and Assumption \ref{a1-2} hold. Recall
that in the supercritical diffusive regime, we write $\phi_0^\e$ as $\phi_0$ (which is defined by \eqref{e2-2})
since it is independent of $\e$.

 We begin with the following lemma, which indicates that the condition \eqref{a1-1-2a} ensures the non-degeneracy of
the limit diffusion matrix $\bar A_0$ defined by \eqref{e2-3-2}.

\begin{lemma}\label{l5-2}
Suppose that Assumption {\rm\ref{a1-1}(iii)} and Assumption {\rm\ref{a1-2}} hold. Then the limit diffusion matrix $\bar A_0$ defined by \eqref{e2-3-2} is
strictly positive definite.
\end{lemma}
\begin{proof}
The proof is similar
to
that of \cite[Remark 4.2]{CCKW}.
 For the  convenience of the reader, we
 give a detailed proof here.
Suppose that the desired conclusion does not hold. Then there
is some $\xi \in \R^d \setminus \{ 0\}$ so
 that $ \langle \bar A_0 \xi, \xi\rangle=0. $
This along with the definition \eqref{e2-3-2} yields that
\begin{equation*}
\int_{\T^d}\int_{\R^d}\left|\langle z+\phi_0(y+z)-\phi_0(y), \xi\rangle \right|^2 K(y,y+z)j(z)\,dz\,dy=0.
\end{equation*}
By the continuity of  $(y,z)\mapsto z+\phi_0(y+z)-\phi_0(y)$
 and the strict positivity of the function $K$, we have
\begin{equation}\label{l5-2-1}
\langle z+\phi_0(y+z)-\phi_0(y), \xi\rangle=0,\quad \ y\in \T^d,\ z\in {\rm supp}[ j].
\end{equation}
  Since $\xi \not=0$, there is some $1\leq i \leq d$ so that $  \langle e_i, \xi  \rangle \not=0$.
 Let $\{z_k^i\}_{k\ge 1}$
be the sequence in \eqref{a1-1-2a}.
 Without loss of generality, we assume
\begin{equation}\label{e:3.25}
 \lim_{k\to\infty} \frac{z_k^i- z_\infty^i} { |z_k^i- z_\infty^i|} =e_i.
\end{equation}
 According to \eqref{l5-2-1},
 for all $k\ge1$ and $y\in \T^d$,
$$
\langle \phi_0(y+z_k^i)-\phi_0(y),\xi\rangle=-\langle z_k^i,\xi\rangle. 
$$
  Letting $k\to \infty$ yields that for any $y\in \T^d$,
$$
\langle \phi_0(y+z_\infty^i)-\phi_0(y),\xi\rangle=-\langle z_\infty^i,\xi\rangle.  
$$
Thus
\begin{equation}\label{e:3.28}
\langle \phi_0(y+z_k^i)- \phi_0(y+z_\infty^i),\xi\rangle=-\langle  z_k^i- z_\infty^i,\xi\rangle
\quad \hbox{for every } y\in \T^d
\hbox{ and } k\geq 1.
\end{equation}
 Since $\phi_0\in C^1(\T^d)$, dividing \eqref{e:3.28} by
  $| z_k^i- z_\infty^i|$
 on both sides   and letting $k\to \infty$,    we have by \eqref{e:3.25} that
\begin{equation*}
 \frac{\partial}{\partial e_i}
 \langle \phi_0(y+ z^i_\infty),\xi\rangle=-\langle e_i, \xi \rangle  \not= 0 \quad \hbox{for every } y\in \T^d.
\end{equation*}
Here  $\frac{\partial}{\partial e_i} $
stands for the directional derivative in the direction of $e_i$.
The above  contradicts with the fact that the function $y\mapsto \langle \phi_0(y + z^i_\infty  ),\xi\rangle$ is multivariate periodic
in $y$. So the desired conclusion holds.
\end{proof}

\begin{lemma}\label{l4-1} Let $\Lambda_\e^5 f(x,y)$ be defined in \eqref{e2-3-4}.
For every $\beta>0$, there exists  $C_1:=C_1(\beta)>0$ such that for every $f\in \M_{\beta}$, $x\in \R^d$ and $y\in \T^d$,
\begin{equation}\label{l4-1-1}
|\Lambda_\e^5 f(x,y)|\le C_1\left(\e^{-1}\gamma_{3,\beta}(\e)+\gamma_9(\e)\right)\left(\sum_{i=0}^3 \|f\|_{\M_{\beta},i}\right)
(1+|x|)^{-d-\beta}.
\end{equation}
\end{lemma}
\begin{proof}
For any $f\in \M_{\beta}$, $x\in \R^d$ and $y\in \T^d$,
\begin{align*}
\Lambda_\e^5 f(x,y)&=\e^{-2}\int_{\{|z|\le \e^{-1}\}}\hat \delta f(x;\e z)K(y,y+z)j(z)\,dz-
\frac{1}{2}\left\langle \nabla^2 f(x), \int_{\R^d} (z\otimes z) K(y,y+z)j(z)\,dz\right\rangle\\
&\quad +\e^{-2}\int_{\{|z|>\e^{-1}\}}\hat \delta f(x;\e z)K(y,y+z)j(z)\,dz\\
&=:I_1^\e(x,y)+I_2^\e(x,y).
\end{align*}

By Taylor's expansion, we get
\begin{align*}
 |I_1^\e(x,y)|
&\le c_1\e\left(\sup_{x'\in \R^d:|x-x'|\le 1}|\nabla^3 f(x')|\right)
\left(\int_{\{|z|\le \e^{-1}\}}|z|^3j(z)\,dz\right)
+c_1 |\nabla^2 f(x)| \left(\int_{\{|z|>\e^{-1}\}}|z|^2 j(z)\,dz\right)\\
&\le c_2\left(\sum_{i=2}^3\|f\|_{\M_{\beta},i}\right)\gamma_9(\e)
(1+|x|)^{-d-\beta}.
\end{align*}
On the other hand, one can see that \eqref{l1-3-5} holds with $\varphi(\e)=\e^{-2}$, and so
\begin{align*}
|I_2^\e(x,y)|&\le c_3\e^{-2}\|f\|_{\M_{\beta},0}
\left((1+|x|)^{-d-\beta} \int_{\{|z|>\e^{-1}\}}j(z)\,dz+\int_{\{|z|>\e^{-1}\}}j(z)(1+|x+\e z|)^{-d-\beta}\,dz\right)\\
&\le c_4\e^{-1}\gamma_{3,\beta}(\e)\|f\|_{\M_{\beta},0}(1+|x|)^{-d-\beta},
\end{align*}
where $\gamma_{3,\beta}(\e)$ is defined by \eqref{e1-7} with $\varphi(\e)=\e^{-2}$.

Hence, the conclusion \eqref{l4-1-1} follows
by combining the above estimates.
\end{proof}

\begin{lemma}\label{l4-2}
Let $\Lambda_\e^6 f(x,y)$ be defined in \eqref{e2-3-4}.
For every $\beta>0$, there exists $C_2:=C_2(\beta)>0$ such that for every $f\in \M_{\beta}$, $x\in \R^d$ and $y\in \T^d$,
\begin{equation}\label{l4-2-1}
|\Lambda_\e^6 f(x,y)|\le C_2\left(\e+\gamma_{3,\beta}(\e)+\gamma_9(\e)\right)\left(\sum_{i=1}^3\|f\|_{\M_{\beta},i}\right)\|\phi_0\|_\infty
(1+|x|)^{-d-\beta}.
\end{equation}
\end{lemma}
\begin{proof}
For any $x\in \R^d$ and $y\in \T^d$, we write
 $\Lambda_\e^6 f(x,y)=I_1^\e(x,y)+I_2^\e(x,y),$ where
\begin{align*}
I_1^\e(x,y)&:=\e^{-1}\sum_{i=1}^d\int_{\{|z|\le \e^{-1}\}}\left(\partial_{x_i}f(x+\e z)-\partial_{x_i}f(x)\right)
\left(\phi_{0,i}(y+z)-\phi_{0,i}(y)\right)K(y,y+z)j(z)\,dz\\
&\quad-\left\langle \nabla^2 f(x), \int_{\R^d} z\otimes \left(\phi_0(y+z)-\phi_0(y)\right)K(y,y+z)j(z)\,dz\right\rangle,\\
I_2^\e(x,y)&:=\e^{-1}\sum_{i=1}^d\int_{\{|z|> \e^{-1}\}}\left(\partial_{x_i}f(x+\e z)-\partial_{x_i}f(x)\right)
\left(\phi_{0,i}(y+z)-\phi_{0,i}(y)\right)K(y,y+z)j(z)\,dz.
\end{align*}

By the mean value theorem, for all $x,z\in \R^d$ with $|z|\le \e^{-1}$,
\begin{align*}
\partial_{x_i}f(x+\e z)-\partial_{x_i}f(x)=
\e\left\langle \nabla \partial_{x_i}f(x), z\right\rangle+J_{\e}(x,z),
\end{align*}
where
$$
|J_{\e}(x,z)| \le c_1\e^2
\left(\sup_{x'\in \R^d:|x'-x|\le 1}|\nabla^3 f(x')|\right)|z|^2\le c_2\|f\|_{\M_{\beta},3}(1+|x|)^{-d-\beta}\e^2|z|^2.
$$
This  along with  $\int_{\R^d}|z|^2j(z)\,dz<\infty$ yields that
\begin{align*}
|I_1^\e(x,y)|&\le c_3\left(\e^{-1}\|\phi_0\|_\infty\int_{\{|z|\le \e^{-1}\}}|J_\e(x,z)|j(z)\,dz+\|\phi_0\|_\infty|\nabla^2 f(x)|
\int_{\{|z|>\e^{-1}\}}|z|j(z)\,dz\right)\\
&\le c_4\|\phi_0\|_\infty\left(\sum_{i=2}^3\|f\|_{\M_{\beta},i}\right)\left(\e\int_{\R^d}|z|^2j(z)\,dz+ \int_{\{|z|>\e^{-1}\}}|z| j(z)\,dz\right)(1+|x|)^{-d-\beta}\\
&\le c_5\|\phi_0\|_\infty\left(\sum_{i=2}^3\|f\|_{\M_{\beta},i}\right)(1+|x|)^{-d-\beta} (\e+\gamma_9(\e)).
\end{align*}
On the other hand, according to \eqref{l1-3-5}, we can derive
\begin{align*}
|I_2^\e(x,y)|&\le c_6\e^{-1}\|f\|_{\M_{\beta},1}\|\phi_0\|_\infty
\left((1+|x|)^{-d-\beta} \int_{\{|z|>\e^{-1}\}}j(z)\,dz+\int_{\{|z|>\e^{-1}\}}
j(z)(1+|x+\e z|)^{-d-\beta}\,dz\right)\\
&\le c_7\|f\|_{\M_{\beta},1}\|\phi_0\|_\infty\gamma_{3,\beta}(\e)(1+|x|)^{-d-\beta}.
\end{align*}

Combining the above estimates,
we can get \eqref{l4-2-1}.
\end{proof}

\begin{lemma}\label{l4-3} Let $\Lambda_\e^7 f(x,y)$ be defined by \eqref{e2-3-4}.
For every $\beta>0$, there exists $C_3:=C_3(\beta)>0$ such that for every $f\in \M_{\beta}$, $x\in \R^d$ and $y\in \T^d$,
\begin{equation}\label{l4-3-1}
|\Lambda_\e^7 f(x,y)|\le C_3\left(\e+\gamma_{3,\beta}(\e)+\gamma_9(\e)\right)\left(\sum_{i=1}^3\|f\|_{\M_{\beta},i}\right)(1+|x|)^{-d-\beta}.
\end{equation}
\end{lemma}
\begin{proof}
By the mean value theorem, for all $x,z\in \R^d$ with $|z|\le \e^{-1}$,
$$
\nabla f(x+\e z)-\nabla f(x)=\e \nabla^2 f(x) \cdot z+J_\e(x,z),
$$
where
$$
|J_\e(x,z)| \le c_1\e^2
\left(\sup_{x'\in \R^d:|x'-x|\le 1}|\nabla^3 f(x')|\right)|z|^2\\
 \le c_2\|f\|_{\M_{\beta},3}(1+|x|)^{-d-\beta}\e^2|z|^2.$$
Then,
\begin{align*}
|\Lambda_\e^7f(x,y)|&\le c_3\e^{-1}\|\phi_0\|_\infty
\left(\int_{\{|z|\le \e^{-1}\}}|J_\e(x,z)|j(z)\,dz\right)+c_3\|\phi_0\|_\infty|\nabla^2 f(x)| \int_{\{|z|>\e^{-1}\}}|z|j(z)\,dz\\
&\quad +c_3\e^{-1}\|\phi_0\|_\infty\|f\|_{\M_{\beta},1}\left(\int_{\{|z|>\e^{-1}\}}\left((1+|x|)^{-d-\beta}+(1+|x+\e z|)^{-d-\beta}\right)j(z)\,dz\right)\\
&\le c_4\left(\sum_{i=1}^3\|f\|_{\M_\beta,i}\right)\|\phi_0\|_\infty(1+|x|)^{-d-\beta}\left(\e\int_{\{|z|\le \e^{-1}\}}|z|^2j(z)\,dz
+\int_{\{|z|>\e^{-1}\}}|z|j(z)\,dz\right)\\
&\quad +c_4\|f\|_{\M_\beta,1}\|\phi_0\|_\infty\e^{-1}\int_{\{|z|>\e^{-1}\}}j(z)(1+|x+\e z|)^{-d-\beta}\,dz\\
&\le c_5\left(\e+\gamma_{3,\beta}(\e)+\gamma_9(\e)\right)\left(\sum_{i=1}^3\|f\|_{\M_\beta,i}\right)(1+|x|)^{-d-\beta},
\end{align*} where the last inequality follows from \eqref{l1-3-5}.
This proves the desired assertion.
\end{proof}

\begin{lemma}\label{l4-4}
Suppose that \eqref{t1-3-1} holds for some $\beta>0$. Let $\Upsilon_\e^3 f(x,y)$ be defined by \eqref{e2-3-8}. Then, there exists a positive constant $C_4:=C_4(\beta)$ such that for every $f\in \M_\beta$, $x\in \R^d$ and $y\in \T^d$,
\begin{equation}\label{l4-4-1}
|\Upsilon_\e^3 f(x,y)|\le c_1\left(\sum_{i=0}^3\|f\|_{\M_{\beta},i}\right)(1+|x|)^{-d-\beta}.
\end{equation}
\end{lemma}
\begin{proof}
It is easy to see that
for every $f\in \M_\beta$, $x\in \R^d$ and $y\in \T^d$,
$$
\left|\Theta_\e^3 f(x,y)\right|\le c_1\|f\|_{\M,\beta,2}(1+|x|)^{d-\beta}.
$$
Combining it with \eqref{l4-1-1}, \eqref{l4-2-1} and \eqref{l4-3-1}, we find that                                                                                                                             \begin{align*}
|\Upsilon_\e^3 f(x,y)|
&\le c_2\left[|\Theta_\e^3 f(x,y)|+\left(\sum_{i=0}^{3}\|f\|_{\M_{\beta},i}\right)
\xi_{4,\beta}(\e)(1+|x|)^{-d-\beta}\right]\\
&\le c_3\left(\sum_{i=0}^{3}\|f\|_{\M_{\beta},i}\right)\left(1+\xi_{4,\beta}(\e)\right)(1+|x|)^{-d-\beta}\\
&\le c_4\left(\sum_{i=0}^{3}\|f\|_{\M_{\beta},i}\right)(1+|x|)^{-d-\beta},
\end{align*}
where the last inequality follows from \eqref{t1-3-1}. The proof is complete.
\end{proof}

\section{Proofs of main results and examples}\label{section4}

In this section, we first prove the main results from Section \ref{section2} according to the subcritical \(\alpha\)-stable regime, critical regime and supercritical diffusive regime. We then present the proofs for all examples introduced in Section \ref{section1}.

\subsection{Subcritical $\alpha$-stable regime}
Recall that for
every $f\in \M_{\beta}$ and $x\in \R^d$,
\begin{equation}\label{e2-1}
\begin{split}
\sL_\e f(x)&= \varphi(\e)\e^{-d}\int_{\R^d}\delta f(x;z)K\left(\e^{-1} x ,\e^{-1}({x+z})\right)j \left(\e^{-1} z \right)\,dz+
\e\varphi(\e)\left\langle \nabla f(x), \Phi_\e\left(\e^{-1} x \right)\right\rangle\\
&=\bar \sL_\alpha f(x)+\sum_{i=1}^2\Lambda_\e^i f\left(x,\e^{-1} x \right)+\Theta_\e^1 f\left(x,\e^{-1} x \right)+\e\varphi(\e)\left\langle \nabla f(x), \Phi_\e\left(\e^{-1} x \right)\right\rangle,
\end{split}
\end{equation}
where $\Lambda_\e^1 f(x,y)$, $\Lambda_\e^2 f(x,y)$, $\Theta_\e^1 f(x,y)$ and $\Phi_\e:\T^d \to \R^d$ are defined by \eqref{e1-5a} and \eqref{e2-1a} respectively.
It is easy to see that for any $f\in \M_{\beta}$, $k\ge0$, $i=1,2$, $x\in \R^d$ and $y\in \T^d$,
\begin{equation}\label{e2-8a}
\nabla^k\Lambda_\e^i f(\cdot,y)(x)=\Lambda_\e^i\left(\nabla^k f\right)(x,y),\,\,\nabla^k\Theta_\e^1 f(\cdot,y)(x)=\Theta_\e^1\left(\nabla^k f\right)(x,y),\,\,
\nabla^k\Gamma_\e(f,\phi)(\cdot,y)(x)=\Gamma_\e\left(\nabla^k f,\phi\right)(x,y).
\end{equation}
Fix $\beta>0$ such that \eqref{t1-1-1} is satisfied. For $h\in \M_{\beta}$, let $\bar u$ be the solution to \eqref{e1-3}. By Lemma \ref{l1-1}, $\bar u \in \M_{\beta_0}$, and so for all $x\in \R^d$ and $y\in \T^d$,
$
\nabla^k F_1^\e(\cdot,y)(x)=\Upsilon_1^\e\left(\nabla^k \bar u\right)(x,y).
$
This, along with \eqref{l1-6-1}, yields that $F_1^\e(\cdot,y)\in \M_{\beta_0}$
and, for all $\e\in (0,1)$ and $k\ge0$,
\begin{equation}\label{e2-4}
\sup_{y\in \T^d}\|F_1^\e(\cdot,y)\|_{\M_{\beta_0},k}
\le c_1(k) \sum_{i=k}^{k+2}\|\bar u\|_{\M_{\beta_0},i}.
\end{equation}
Indeed, as indicated by the argument for \eqref{l1-6-2}, the leading term for the estimate of
$\|F^\e_1(\cdot,y)\|_{\M_{\beta_0},k}$ is $\|\Theta_\e^1\bar u(\cdot,y)\|_{\M_{\beta_0},k}$. Furthermore,
by \eqref{e2-3aa}, it holds that for all $k\ge0$, $x\in \R^d$ and $y\in \T^d$,
$
\sL \nabla_x^k\phi_1^\e(x,\cdot)(y)=-\nabla_x^k F_1^\e(x,y)+\nabla_x^k \bar F_1^\e(x).$
This along with Assumption \ref{a1-2} and \eqref{e2-4}
gives
\begin{equation}\label{e2-11}
\sup_{\e\in (0,1)}\left(\sup_{y\in \T^d}\left|\nabla_x^k \phi_1^\e(x,y)\right|
+\sup_{y\in \T^d}\left|\nabla_y\nabla_x^k \phi_1^\e(x,y)\right|\right)\le c_2(k)(1+|x|)^{-d-\beta_0},\quad x\in \R^d,k\ge0,
\end{equation}
where $\nabla_x$ and $\nabla_y$ denote the gradient for the variables $x\in \R^d$ and $y\in \T^d$ respectively.

\begin{lemma}\label{l2-2}
Let $\phi_1^\e:\R^d \times\T^d\to \R$ be the solution to \eqref{e2-3aa}, and define $\Psi^\e:\R^d \to \R$ by
$\Psi^\e(x):=\phi_1^\e\left(x,\e^{-1} x \right)$. Then,
\begin{equation}\label{l2-2-1}
\varphi(\e)^{-1}\sL_\e \Psi^\e(x)=-F_1^\e\left(x,\e^{-1} x \right)+\bar F_1^\e(x)+G_1^\e\left(x,\e^{-1} x \right)
+\lambda \varphi(\e)^{-1}\phi_1^\e\left(x,\e^{-1} x \right),\quad x\in \R^d,
\end{equation}
where $G_1^\e:\R^d\times \T^d\to \R$ is defined by \eqref{l2-2-2}.
Moreover, for every $k\ge 0$, there exist positive constants $C_1(k)$ and $C_2(k)$ such that for all $\e\in (0,1)$, $x\in \R^d$ and $y\in \T^d$,
\begin{equation}\label{l2-2-3}
\begin{split}
|\nabla_x^k  G_1^\e(x,y)|&\le C_1(k)\xi_{2,\beta_0}(\e)\sum_{i=k}^{k+2}\left(\sup_{y\in \T^d}\|\phi_1^\e(\cdot,y)\|_{\M_{\beta_0},i}
+\sup_{y\in \T^d}\|\nabla_y\phi_1^\e(\cdot,y)\|_{\M_{\beta_0},i}\right)
(1+|x|)^{-d-\beta_0}\\
 &\le C_2(k)\xi_{2,\beta_0}(\e)(1+|x|)^{-d-\beta_0},
\end{split}
\end{equation}
where $\xi_{2,\beta_0}(\e)$ is defined by \eqref{e1-8a}.
\end{lemma}
\begin{proof}
For any $x\in \R^d$ and $\varepsilon\in (0,1)$,
\begin{align*}
&\varphi(\e)^{-1}\sL_\varepsilon \Psi^\e(x)\\
&= \varepsilon^{-d}{\rm p.v.}\int_{\R^d}(\phi_1^\varepsilon(x+z,\e^{-1}(x+z))-\phi_1^\varepsilon(x,\e^{-1} x))
K(\e^{-1} x,\e^{-1}(x+z))j (\e^{-1} z )\,dz\\
&=\varepsilon^{-d}{\rm p.v.}\int_{\R^d} \left(\phi_1^\varepsilon(x,\e^{-1}(x+z))-\phi_1^\varepsilon(x,\e^{-1} x)\right)
K(\e^{-1} x,\e^{-1}(x+z))j (\e^{-1} z )\,dz\\
&\quad +\varepsilon^{-d}\int_{\R^d}\left(\phi_1^\e\left(x+z,\e^{-1} x\right)-\phi_1^\e\left(x,\e^{-1} x\right)-
\left\langle z, \nabla_x\phi_1^\e\left(\cdot,\e^{-1} x\right)(x)\right\rangle\I_{\{|z|\le 1\}}\right)
K(\e^{-1} x,\e^{-1}(x+z))j(\e^{-1} z )\,dz\\
&\quad+
\e\left\langle \nabla_x \phi_1^\e\left(\cdot,
\e^{-1}x
\right)(x),\Phi_\e\left(\e^{-1} x \right)\right\rangle+\varepsilon^{-d}\int_{\R^d}\delta_2^\varepsilon \phi_1^\varepsilon\left(x,\e^{-1} x;z\right)
K(\e^{-1} x,\e^{-1}(x+z))j (\e^{-1} z )\,dz.
\end{align*}
Hence, by the
change of variables,
\eqref{l2-2-2} and  \eqref{e2-3aa}, we can write
\begin{align*}
\varphi(\e)^{-1} \sL_\varepsilon \Psi^\e(x)&=\sL\phi_1^\varepsilon(x,\cdot)\left(\e^{-1} x\right)+G_1^\varepsilon\left(x,\e^{-1} x\right)
+\lambda \varphi(\e)^{-1}\phi_1^\e(x,\e^{-1} x )\\
&=
-F_1^\e\left(x,\e^{-1} x \right)+\bar F_1^\e(x)+G_1^\e\left(x,\e^{-1} x \right)+\lambda \varphi(\e)^{-1}\phi_1^\e(x,\e^{-1}x),
\end{align*}
which
completes the proof of \eqref{l2-2-1}.

By
\eqref{e2-11} we know $\phi_1(\cdot,y)\in \M_{\beta_0}$.
According to the mean value theorem, we have
\begin{align*}
 |\delta \phi_1^\e(\cdot,y)(x;\e z)|
&\le c_1\left(\sum_{i=0}^2
\sup_{y\in \T^d}\|\phi_1^\e(\cdot,y)\|_{\M_{\beta_0},i}\right)(1+|x|)^{-d-\beta_0}\left(\e^2|z|^2\I_{\{|z|\le 1\}}
+\e|z|\I_{\{1<|z|\le \e^{-1}\}}+\I_{\{|z|>\e^{-1}\}}\right)\\
& \quad +c_1\sup_{y\in \T^d}\|\phi_1^\e(\cdot,y)\|_{\M_{\beta_0},0}(1+|x+\e z|)^{-d-\beta_0}\I_{\{|z|>\e^{-1}\}}.
\end{align*}
This along with \eqref{l1-3-5} yields that
\begin{align*}
&\e^{-d}\left|\int_{\R^d}\delta\phi_1^\varepsilon(\cdot, y)(x;z)K\left(y,y+\e^{-1} z\right)j \left(\e^{-1} z \right)\,dz\right|\\
&=\left|\int_{\R^d}\delta\phi_1^\varepsilon(\cdot, y)(x;\e z)K\left(y,y+z\right)j(z)\,dz\right|\\
&\le c_2\e^{-1}\varphi(\e)^{-1}\left(\gamma_2(\e)+\gamma_{3,\beta_0}(\e)\right)\left(\sum_{i=0}^2
\sup_{y\in \T^d}\|\phi_1^\e(\cdot,y)\|_{\M_{\beta_0},i}\right)(1+|x|)^{-d-\beta_0}.
\end{align*}
On the other hand, for fixed $x,z\in \R^d$ and $y\in \T^d$,  define $
H^\e(t,s):=\phi_1^\e\left(x+t\e z,y+s z\right).$
It is easy to see that
\begin{align*}
\delta_2^\e\phi_1^\e(x,y;\e z)&=
H^\e(1,1)-H^\e(1,0)-H^\e(0,1)+H^\e(0,0)
=\int_0^1\int_0^1 \frac{\partial^2 H^\e(t,s)}{\partial t \partial s}\,dt\,ds\\
&=\e\int_0^1\int_0^1 \left\langle \nabla_x\nabla_y \phi_1^\e\left(x+t\e z,y+s z\right), z\otimes z\right\rangle \,dt\,ds
\end{align*} and
\begin{align*}
\delta_2^\e\phi_1^\e(x,y;\e z)&=
H^\e(1,1)-H^\e(1,0)-H^\e(0,1)+H^\e(0,0)
=\int_0^1\frac{\partial H^\e(t,1)}{\partial t}\,dt-\int_0^1\frac{\partial H^\e(t,0)}{\partial t}\,dt\\
&=\e\left(\int_0^1\left\langle \nabla_x\phi_1^\e\left(x+t\e z,y+z\right), z\right\rangle \,dt-
\int_0^1\left\langle \nabla_x\phi_1^\e\left(x+t\e z,y\right), z\right\rangle \,dt\right).
\end{align*}
These together with the definition of $\delta_2^\e\phi_1^\e(x,y;\e z)$ in turn yield that
\begin{equation}\label{l2-2-4}
\begin{split}
\left|\delta_2^\e\phi_1^\e(x,y;\e z)\right|\le& c_3
\sum_{i=0}^1\left(
\sup_{y\in \T^d}\|\phi_1^\e(\cdot,y)\|_{\M_{\beta_0},i}+\sup_{y\in \T^d}\|\nabla_y\phi_1^\e(\cdot,y)\|_{\M_{\beta_0},i}\right)(1+|x|)^{-d-\beta_0}\\
&\quad\times\left(\e|z|^2\I_{\{|z|\le 1\}}+\e|z|\I_{\{1<|z|\le \e^{-1}\}}
+\I_{\{|z|>\e^{-1}\}}\right)\\
&+c_3\sup_{y\in \T^d}\|\phi_1^\e(\cdot,y)\|_{\M_{\beta_0},0}(1+|x+\e z|)^{-d-\beta_0}\I_{\{|z|>\e^{-1}\}}.
\end{split}
\end{equation}
Hence, by \eqref{l1-3-5}, we can verify directly that for every $x\in \R^d$ and $y\in \T^d$
\begin{align*}
&\e^{-d}\left|\int_{\R^d}\delta_2^\varepsilon \phi_1^\varepsilon\left(x,y;z\right)
K\left(y,y+\e^{-1} z\right)j \left(\e^{-1} z \right)\,dz\right|\\
&=
\left|\int_{\R^d}\delta_2^\varepsilon \phi_1^\varepsilon\left(x,y;\e z\right)
K\left(y,y+z\right)j(z)\,dz\right|\\
&\le c_4\e^{-1}\varphi(\e)^{-1}\left(\gamma_2(\e)+\gamma_{3,\beta_0}(\e)\right)\sum_{i=0}^2
\left(\sup_{y\in \T^d}\|\phi_1^\e(\cdot,y)\|_{\M_{\beta_0},i}
+\sup_{y\in \T^d}\|\nabla_y\phi_1^\e(\cdot,y)\|_{\M_{\beta_0},i}
\right)(1+|x|)^{-d-\beta_0}.
\end{align*}
Therefore, putting all the estimates above together with \eqref{e2-3aa}, \eqref{e:PHI} and \eqref{e2-11}, we obtain that
\begin{align*}
|G_1^\e(x,y)|&\le c_4\left(\e^{-1}\varphi(\e)^{-1}\left(\gamma_2(\e)+\gamma_{3,\beta_0}(\e)\right)+\e\gamma_4(\e)+\varphi(\e)^{-1}\right)\\
&\quad\times\sum_{i=0}^2 \left(
\sup_{y\in \T^d}\|\phi_1^\e(\cdot,y)\|_{\M_{\beta_0},i}+\sup_{y\in \T^d}\|\nabla_y\phi_1^\e(\cdot,y)\|_{\M_{\beta_0},i}\right)(1+|x|)^{-d-\beta_0}\\
&\le c_5\xi_{2,\beta_0}(\e)(1+|x|)^{-d-\beta_0}.
\end{align*}

Furthermore, for every $k\ge 1$,
\begin{align*}
\nabla^k_x G_1^\e(x,y)=&\varepsilon^{-d}\int_{\R^d} \delta\nabla_x^k\phi_1^\varepsilon(\cdot, y)(x;z)K\left(y,y+\e^{-1} z\right)j \left(\e^{-1} z \right)\,dz\\
&+\varepsilon^{-d}\int_{\R^d} \delta^\varepsilon_2\nabla_x^k\phi_1^\varepsilon(x,y;z)K\left(y,y+\e^{-1} z\right)j \left(\e^{-1} z \right)\,dz
+\e\left\langle \nabla_x^{k} \phi_1^\e\left(x,y\right),\Phi_\e\left(y\right)\right\rangle\\
&-\lambda\varphi(\e)^{-1}\nabla_x^{k}\phi_1^\e\left(x,y\right).
\end{align*}
Then, applying \eqref{e2-11} and following the same arguments as above, we can prove \eqref{l2-2-3}.
\end{proof}

\begin{remark}\label{r2-1} Following the proof of Lemma \ref{l2-2}, we can prove
the next
assertion. \emph{Suppose that there exists $\theta\in [0,1)$ such that \eqref{r1-1-1} holds.  Then, for all $k\ge0$, there are constants $C_3(k)$ and $C_4(k)$ such that
\begin{align*}
&|\nabla_x^k G_1^\e(x,y)|\\
&\le C_3(k)(1+|x|)^{-d-\beta_0}\begin{cases} \tilde \xi_{2,\beta_0}(\e)\sum_{i=k}^{k+1}\sup\limits_{y\in \T^d}\|\phi_1^\e(\cdot,y)\|_{\M_{\beta_0},i}
,\,\, &\theta=0,\\
\tilde\xi_{2,\beta_0}(\e)\sum_{i=k}^{k+1}\left(\sup\limits_{y\in \T^d}\|\phi_1^\e(\cdot,y)\|_{\M_{\beta_0},i}
+\sup\limits_{y,z\in \T^d}
\left\|\frac{\left|\nabla_x\phi_1^\e(\cdot,y+z)-\nabla_x\phi_1^\e(\cdot,y)\right|}{|z|^\theta}\right\|_{\M_{\beta_0},i}\right),\,\,&\theta\in (0,1)\end{cases}\\
&\le C_4(k)\tilde \xi_{2,\beta_0}(\e)(1+|x|)^{-d-\beta_0},
\end{align*}
where
\begin{align}\label{r2-1-3}
\tilde \xi_{2,\beta_0}(\e):=\e^{-1}\varphi(\e)\left(\tilde \gamma_2(\e)+\gamma_{3,\beta_0}(\e)\right)+\e\gamma_4(\e)+\varphi(\e)^{-1},
\end{align}
and $\tilde \gamma_2(\e)$ is defined by \eqref{r1-1-3}.}
Indeed, take $k=0$ for example. In order to prove the assertion above,
instead of \eqref{l2-2-4}, we need
to apply the following
estimates: for $\theta=0$,
\begin{align*}
|\delta_2^\e\phi_1^\e(x,y;\e z)|\le& c_1\left(\sum_{i=0}^1
\sup_{y\in \T^d}\|\phi_1^\e(\cdot,y)\|_{\M_{\beta_0},i}\right)(1+|x|)^{-d-\beta_0}\left(\e|z|\I_{\{0<|z|\le \e^{-1}\}}
+\I_{\{|z|>\e^{-1}\}}\right)\\
&+c_1\sup_{y\in \T^d}\|\phi_1^\e(\cdot,y)\|_{\M_{\beta_0},0}(1+|x+\e z|)^{-d-\beta_0}\I_{\{|z|>\e^{-1}\}};
\end{align*}
for $\theta\in (0,1)$,
\begin{align*}
|\delta_2^\e\phi_1^\e(x,y;\e z)|
&\le c_1
\sum_{i=0}^1\left(
\sup_{y\in \T^d}\|\phi_1^\e(\cdot,y)\|_{\M_{\beta_0},i}+\sup_{y,z\in \T^d}
\left\|\frac{\left|\nabla_x\phi_1^\e(\cdot,y+z)-\nabla_x\phi_1^\e(\cdot,y)\right|}{|z|^\theta}\right\|_{\M_{\beta_0},i}\right)(1+|x|)^{-d-\beta_0}\\
&\quad\quad\times\left(\e|z|^{1+\theta}\I_{\{|z|\le 1\}}+\e|z|\I_{\{1<|z|\le \e^{-1}\}}
+\I_{\{|z|>\e^{-1}\}}\right)\\
&\quad +c_1\sup_{y\in \T^d}\|\phi_1^\e(\cdot,y)\|_{\M_{\beta_0},0}(1+|x+\e z|)^{-d-\beta_0}\I_{\{|z|>\e^{-1}\}}.
\end{align*}
\end{remark}

\begin{lemma}\label{l2-1}
Let $v_1^\e$ be defined by \eqref{e2-7}. Then,
\begin{equation}\label{l2-1-1}
\lambda v_1^\e(x)-\sL_\e v_1^\e(x)= G_1^\e\left(x,\e^{-1} x \right)+\bar F_1^\e(x),
\end{equation}
where $G_1^\e:\R^d\times \T^d \to \R$, $F_1^\e:\R^d\times \T^d\to \R$ and $\bar F_1^\e:\R^d \to \R$ are defined respectively by
\eqref{l2-2-2} and \eqref{l2-1-3}.
\end{lemma}

\begin{proof}
According to \eqref{e2-1}, we
have
\begin{equation}\label{e:add-}
\begin{split}
&\sL_\e\left(\e\left\langle \nabla \bar u(\cdot), \phi_0^\e
\right(\e^{-1} \cdot )
\rangle\right)(x)\\
&=\e\left\langle \sL_\e \nabla \bar u(x), \phi_0^\e\left(\e^{-1} x \right)\right\rangle+
\e\left\langle \nabla \bar u(x), \sL_\e\phi_0^\e
(\e^{-1} \cdot )(x)
 \right\rangle +\sum_{i=1}^d\e\Gamma_\e\left(\partial_{x_i}\bar u,\phi_{0,i}^\e\right)\left(x,\e^{-1} x \right)\\
&=\e\left\langle \bar \sL_\alpha \nabla \bar u(x)+\sum_{i=1}^2\Lambda_\e^i \nabla \bar u
\left(x,\e^{-1} x \right)+\Theta_\e^1 \nabla \bar u
\left(x,\e^{-1} x \right), \phi_0^\e\left(\e^{-1} x \right)\right\rangle\\
&\quad +\e^{2}\varphi(\e)\left\langle
\nabla^2 \bar u(x), \phi_0^\e\left(\e^{-1} x \right)\otimes \Phi_\e\left(\e^{-1} x \right)\right\rangle\\
&\quad -
\e\varphi(\e)\left\langle
\nabla \bar u(x),
\Phi_\e\left(\e^{-1} x \right)\right\rangle
+\sum_{i=1}^d\e\Gamma_\e\left(\partial_{x_i}\bar u,\phi_{0,i}^\e\right)\left(x,\e^{-1} x \right),
\end{split}
\end{equation}
where the last equality follows from
\begin{equation*}
\sL_\e\phi_0^\e
(\e^{-1} \cdot )(x)
=\varphi(\e)\sL \phi_0^\e(\cdot)\left(\e^{-1}x\right)
= -\varphi(\e)\Phi_\e\left(\e^{-1} x \right)
\end{equation*}
that can be verified directly by \eqref{e2-2} and the
change of variables.

Therefore, combining this with \eqref{e2-1}, \eqref{e1-3-}, \eqref{e1-3} and  \eqref{l2-2-1},
we find that
\begin{align*}
&\lambda v_1^\e(x)-\sL_\e v_1^\e(x)\\
&=\left(\lambda u_\e(x)-\sL_\e u_\e(x)\right)-
\left(\lambda \bar u(x)-\bar \sL_\alpha \bar u(x) \right)
+\sum_{i=1}^2\Lambda_\e^i \bar u\left(x,\e^{-1} x \right)+\Theta_\e^1 \bar u\left(x,\e^{-1} x \right)+\e\varphi(\e)\left\langle\nabla \bar u(x), \Phi_\e\left(\e^{-1} x \right)\right\rangle\\
&\quad-
\lambda\e
\left\langle \nabla \bar u(x), \phi_0^\e\left(\e^{-1} x \right)\right\rangle
+\e\left\langle \bar \sL_\alpha \nabla \bar u(x)+\sum_{i=1}^2\Lambda_\e^i \nabla \bar u
\left(x,\e^{-1} x \right)+\Theta_\e^1 \nabla \bar u
\left(x,\e^{-1} x \right), \phi_0^\e\left(\e^{-1} x \right)\right\rangle\\
&\quad +\e^{2}\varphi(\e)\left\langle
\nabla^2 \bar u(x), \phi_0^\e\left(\e^{-1} x \right)\otimes \Phi_\e\left(\e^{-1} x \right)\right\rangle-
\e\varphi(\e)\left\langle \nabla \bar u(x), \Phi_\e\left(\e^{-1} x \right)\right\rangle+
\sum_{i=1}^d\e\Gamma_\e\left(\partial_{x_i}\bar u,\phi_{0,i}^\e\right)\left(x,\e^{-1} x \right)\\
&\quad -\lambda \varphi(\e)^{-1}\phi_1^\e\left(x,\e^{-1} x \right)+\varphi(\e)^{-1}\sL_\e\phi_1^\e\left(\cdot,
\e^{-1} \cdot
\right)(x)\\
&=F_1^\e\left(x, \e^{-1}x\right)-\lambda \varphi(\e)^{-1}\phi_1^\e\left(x,\e^{-1} x \right)+\varphi(\e)^{-1}\sL_\e\phi_1^\e\left(\cdot,
\e^{-1} \cdot
\right)(x)\\
&=G_1^\e\left(x,\e^{-1} x \right)+\bar F_1^\e(x).
\end{align*}
The proof is complete.
\end{proof}

Now, it is a position to present the
\begin{proof}[Proof of Theorem $\ref{thm1}$] The proof is split into two steps.

{\bf Step 1}\,\, Note that  $\int_{\T^d}\Theta_\e^1\bar u(x,y)\,dy=0$. Thus, according to the argument for \eqref{l1-6-2}
(to obtain estimates for corresponding terms except $\Theta_\e^1 \bar u$),
\begin{equation}\label{e2-9a}
 |\nabla^k \bar F_1^\e(x)|
\le c_1\left(\sum_{i=k}^{k+2}\|\bar u\|_{\M_{\beta},i}\right)\xi_{1,\beta_0}(\e)(1+|x|)^{-d-\beta_0},\quad \e\in (0,1),\ x\in \R^d,\ k\ge 0.
\end{equation}
By \eqref{e2-11} and \eqref{l2-2-3},
\begin{equation}\label{p2-1-1a}
\begin{split}
|\nabla^k \bar G_1^\e(x)|&\le c_2\sum_{i=k}^{k+2}\left(\sup_{y\in \T^d}\|\phi_1^\e(\cdot,y)\|_{\M_{\beta_0},i}
+\sup_{y\in \T^d}\|\nabla_y\phi_1^\e(\cdot,y)\|_{\M_{\beta_0},i}\right)\xi_{2,\beta_0}(\e)(1+|x|)^{-d-\beta_0},\\
&\le c_3(k)\xi_{2,\beta_0}(\e)(1+|x|)^{-d-\beta_0}, \quad \e\in (0,1),\ x\in \R^d,\ k\ge 0.
\end{split}
\end{equation}

According to \eqref{e2-13} and Lemma \ref{l1-1}, for every $\e\in (0,1)$, $n\ge 1$ and $k\ge 0$,
\begin{equation}\label{p2-1-2}
\|\psi_{n+1}^\e\|_{\M_{\beta_0},k}\le c_3(k)\left(\|\bar F_n^\e\|_{\M_{\beta_0},k}+\|\bar G_n^\e\|_{\M_{\beta_0},k}\right).
\end{equation}
Furthermore, by the arguments for \eqref{e2-4} and \eqref{e2-9a}, we can derive that for every $n\ge 1$,
\begin{equation}\label{p2-1-5}
\begin{split}
\sup_{y\in \T^d}\|F_{n+1}^\e(\cdot,y)\|_{\M_{\beta_0},k}
&\le c_4(k)\sum_{i=k}^{k+2}\|\psi_{n+1}^\e\|_{\M_{\beta_0},i}\\
&\le c_5(k)\sum_{i=k}^{k+2}\left(\|\bar F_n^\e\|_{\M_{\beta_0},i}+\|\bar G_n^\e\|_{\M_{\beta_0},i}\right)
\end{split}
\end{equation} and
\begin{equation}\label{p2-1-5a}
\begin{split}
\|\bar F_{n+1}^\e\|_{\M_{\beta_0},k}
&\le c_4(k)\xi_{1,\beta_0}(\e)\sum_{i=k}^{k+2}\|\psi_{n+1}^\e\|_{\M_{\beta_0},i}\\
&\le c_5(k)\xi_{1,\beta_0}(\e)\sum_{i=k}^{k+2}\left(\|\bar F_n^\e\|_{\M_{\beta_0},i}+\|\bar G_n^\e\|_{\M_{\beta_0},i}\right).
\end{split}
\end{equation}
In particular, as mentioned above, the leading order in
\eqref{p2-1-5} comes from the term $\Theta_\e^1 \psi_{n+1}^\e\left(x,y\right)$. However, since
$\displaystyle\int_{\T^d}\Theta_\e^1 \psi_{n+1}^\e(x,y)\,dy=0$, for the estimate \eqref{p2-1-5a} this term vanishes.

Note that by \eqref{e2-14}, for every $n\ge 1$, $k\ge0$, $x\in \R^d$ and $y\in \T^d$,
$$
-\sL \nabla_x^k \phi_{n+1}^\e(x,\cdot)(y)=\nabla_x^k G_n(x,y)-\nabla_x^k \bar G_n(x)+\nabla^k_x  F_{n+1}^\e(x,y)-\nabla_x^k \bar F_{n+1}^\e(x).
$$
This along with  Assumption \ref{a1-2} yields that for every $\e\in (0,1)$ and $k\ge 0$,
\begin{align*}
&\sup_{y\in \T^d}|\nabla_x^k\phi_{n+1}^\e(x,y)|+\sup_{y\in \T^d}|\nabla_y\nabla_x^k \phi_{n+1}^\e(x,y)|\\
&\le
c_6(k)\left(\sup_{y\in \T^d}\|G_n^\e(\cdot,y)\|_{\M_{\beta_0},k}+
\|\bar G_n^\e\|_{\M_{\beta_0},k}+
\sup_{y\in \T^d}\|F_{n+1}^\e(\cdot,y)\|_{\M_{\beta_0},k}+
\|\bar F_{n+1}^\e\|_{\M_{\beta_0},k}\right)(1+|x|)^{-d-\beta_0}\\
&\le c_7(k)\left(\sup_{y\in \T^d}\|G_n^\e(\cdot,y)\|_{\M_{\beta_0},k}+\|\bar G_n^\e\|_{\M_{\beta_0},k}
+\|\bar F_{n}^\e\|_{\M_{\beta_0},k}\right)(1+|x|)^{-d-\beta_0},
\end{align*}
where the last inequality follows from \eqref{p2-1-5} and \eqref{p2-1-5a}.
Thus, applying the estimate above and following the argument for \eqref{l2-2-3}, we have for all $\e\in (0,1)$,
\begin{equation}\label{p2-1-3a}
\begin{split}
&\sup_{y\in \T^d}\|G_{n+1}^\e(\cdot,y)\|_{\M_{\beta_0},k}+\|\bar G_{n+1}^\e\|_{\M_{\beta_0},k}\\
&\le c_7(k)\xi_{2,\beta_0}(\e)\sum_{i=k}^{k+2}\left(\sup_{y\in \T^d}
\|\phi_n^\e(\cdot,y)\|_{\M_{\beta_0},i}
+\sup_{y\in \T^d}\|\nabla_y\phi_n^\e(\cdot,y)\|_{\M_{\beta_0},i}\right)\\
&\le c_8(k)\xi_{2,\beta_0}(\e)\sum_{i=k}^{k+2}\left(\sup_{y\in \T^d}\|G_n^\e(\cdot,y)\|_{\M_{\beta_0},i}
+\|\bar G_n^\e\|_{\M_{\beta_0},i}
+\|\bar F_{n}^\e\|_{\M_{\beta_0},i}\right).
\end{split}
\end{equation}

Using the inductive inequalities  \eqref{p2-1-5}, \eqref{p2-1-5a} and \eqref{p2-1-3a}
for $\|G_{n}^\e(\cdot,y)\|_{\M_{\beta_0},k}$, $\|F_{n}^\e(\cdot,y)\|_{\M_{\beta_0},k}$,
$\|\bar G_{n}^\e(\cdot)\|_{\M_{\beta_0},k}$ and $\|\bar F_{n}^\e\|_{\M_{\beta_0},k}$, as well as
the initial estimates \eqref{e2-4}, \eqref{e2-9a}, \eqref{l2-2-3} and \eqref{p2-1-1a},
we can obtain that for every $k\ge 0$, $n\ge 1$,
\begin{equation}\label{p2-1-4}
\begin{split}
 \sup_{y\in \T^d}\|G_{n}^\e(\cdot,y)\|_{\M_{\beta_0},k}+
\|\bar G_{n}^\e\|_{\M_{\beta_0},k}&\le
c_9(k,n)
\xi_{2,\beta_0}(\e)\left(\xi_{1,\beta_0}(\e)+\xi_{2,\beta_0}(\e)\right)^{n-1},\\
 \sup_{y\in \T^d}\|F_{n}^\e(\cdot,y)\|_{\M_{\beta_0},k}&\le c_9(k,n)
\left(\xi_{1,\beta_0}(\e)+\xi_{2,\beta_0}(\e)\right)^{n-1},\\
 \|\bar F_{n}^\e\|_{\M_{\beta_0},k}&\le c_9(k,n)\xi_{1,\beta_0}(\e)\left(\xi_{1,\beta_0}(\e)+\xi_{2,\beta_0}(\e)\right)^{n-1}.
\end{split}
\end{equation}

{\bf Step 2}\,\,
Set
\begin{align*}
\tilde v_n^\e(x):=v_{n+1}^\e(x)-v_n^\e(x)=-\psi_{n+1}^\e(x)-\e\left\langle \nabla \psi_{n+1}^\e(x), \phi_0^\e\left(\e^{-1} x \right)\right\rangle-
\varphi(\e)^{-1}\phi_{n+1}^\e\left(x,\e^{-1} x \right).
\end{align*}
Following the arguments for \eqref{l2-2-1} and using \eqref{e2-14}, we have
\begin{equation}\label{e:add--}\begin{split}
&\lambda \varphi(\e)^{-1}\phi_{n+1}^\e\left(x,\e^{-1} x \right)-\varphi(\e)^{-1}\sL_\e \phi_{n+1}^\e
(\cdot, \e^{-1} \cdot )(x) \\
&=G_n^\e\left(x,\e^{-1} x \right)-\bar G_n^\e(x)+F_{n+1}^\e\left(x,\e^{-1} x \right)-\bar F_{n+1}^\e(x)
-G_{n+1}^\e\left(x,\e^{-1} x \right).\end{split}
\end{equation}
Then, according to the equality above, \eqref{e2-13},
\eqref{e2-1} and the argument for \eqref{e:add-},
we can get that for every $n\ge 1$,
\begin{align*}
  \lambda \tilde v_n^\e(x)-\sL_\e \tilde v_n^\e(x)
&=-\bar F_{n}^\e(x)-\bar G_n^\e(x)
+\sum_{i=1}^2\Lambda_\e^i \psi_{n+1}^\e\left(x,\e^{-1} x \right)+\Theta_\e^1 \psi_{n+1}^\e\left(x,\e^{-1} x \right)
+\e\varphi(\e)\left\langle \nabla\psi_{n+1}^\e(x), \Phi_\e\left(\e^{-1} x \right)\right\rangle\\
&\quad+\e\left\langle \bar \sL_\alpha(\nabla \psi_{n+1}^\e)(x)+
\sum_{i=1}^2\Lambda_\e^i(\nabla \psi_{n+1}^\e)\left(x,\e^{-1} x \right)
+\Theta_\e^1(\nabla \psi_{n+1}^\e)\left(x,\e^{-1} x \right), \phi_0^\e\left(\e^{-1} x \right)\right\rangle\\
&\quad+\e^{2}\varphi(\e)\left\langle \nabla^2 \psi_{n+1}^\e(x), \phi_0^\e\left(\e^{-1} x \right)\otimes\Phi_\e\left(\e^{-1} x \right)\right\rangle
+\e\sum_{i=1}^d \Gamma_\e\left(\partial_{x_i}\psi_{n+1}^\e, \phi_{0,i}^\e\right)\left(x,\e^{-1} x \right)\\
&\quad -\e\varphi(\e)\left\langle \nabla\psi_{n+1}^\e(x), \Phi_\e\left(\e^{-1} x \right)\right\rangle
-\lambda\e\left\langle \nabla \psi_{n+1}^\e(x), \phi_0^\e\left(\e^{-1} x \right)\right\rangle \\
&\quad -G_n^\e\left(x,\e^{-1} x \right)+\bar G_n^\e(x)-F_{n+1}^\e\left(x,\e^{-1} x \right)+\bar F_{n+1}^\e(x)
+G_{n+1}^\e\left(x,\e^{-1} x \right)\\
&=-\bar F_{n}^\e(x)-G_{n}^\e\left(x,\e^{-1} x \right)+G_{n+1}^\e\left(x,\e^{-1} x \right)+\bar F_{n+1}^\e(x).
\end{align*}
Combining this induction property with the initial estimate \eqref{l2-1-1} yields that for every $n\ge 1$ and $x\in \R^d$,
\begin{equation}\label{t1-1-3}
\lambda v_n^\e(x)-\sL_\e v_n^\e(x)=\bar F_{n}^\e(x)+G_{n}^\e\left(x,\e^{-1} x \right).
\end{equation}
This together with \eqref{e1-2a} and  \eqref{p2-1-4} in turn
gives
that for every $n\ge 1$,
$$
\|v_n^\e\|_{L^2(\R^d;dx)}\le \lambda^{-1}\left\|\bar F_{n}^\e(\cdot)+G_{n}^\e
(\cdot, \e^{-1} \cdot )
\right\|_{L^2(\R^d;dx)}
\le c_1\left(\xi_{1,\beta_0}(\e)+\xi_{2,\beta_0}(\e)\right)^n.
$$
So the proof is finished.
\end{proof}

\subsection{Critical regime}
In this subsection, we give the
\begin{proof}[Proof of Theorem $\ref{thm2}$]
For every $f\in \M_\beta$ and $x\in \R^d$,
\begin{equation}\label{e3-7}
\begin{split}
\sL_\e f(x)&= \varphi(\e)\e^{-d}\int_{\R^d}\delta f(x;z)K\left(\e^{-1} x ,
\e^{-1} (x+z)
\right)j \left(\e^{-1} z \right)dz+
\e\varphi(\e)\left\langle \nabla f(x), \Phi_\e\left(\e^{-1} x \right)\right\rangle\\
&=\bar \sL_2 f(x)+\sum_{i=3}^4\Lambda_\e^i f\left(x,\e^{-1} x \right)+\Theta_\e^2 f\left(x,\e^{-1} x \right)+\e\varphi(\e)\left\langle \nabla f(x), \Phi_\e\left(\e^{-1} x \right)\right\rangle,
\end{split}
\end{equation}
where $\Lambda_\e^3 f$, $\Lambda_\e^4$, $\Theta_\e^2 f$ and $\Phi_\e:\T^d \to \R^d$ are defined by \eqref{e3-3a}
and \eqref{e2-1a} respectively.

Note that \eqref{e2-2a} still holds true.
Similarly to \eqref{e2-8a}, we have,
for all $k\ge0$, $i=3,4$, $x\in \R^d$ and $y\in \T^d$,
$$\nabla^k\Lambda_\e^i f(\cdot,y)(x)=\Lambda_\e^i\left(\nabla^k f\right)(x,y),\,\,
\nabla^k\Theta_\e^2 f(\cdot,y)(x)=\Theta_\e^2\left(\nabla^k f\right)(x,y),\,\,
 \nabla^k\Gamma_\e(f,\phi)(\cdot,y)(x)=\Gamma_\e\left(\nabla^k f,\phi\right)(x,y).
$$
Using these and
Lemma \ref{l3-4} as well as its proof, we can obtain
\begin{equation}\label{e3-9}
\begin{split}
&\sup_{y\in \T^d}\|F_1^\e(\cdot,y)\|_{\M_{\beta},k}\le c_1\left(\sum_{i=k}^{k+3}\|\bar u\|_{\M_{\beta},i}\right)(1+\xi_{3,\beta}(\e))\le c_2(k),\\
&\|\bar F_1^\e\|_{\M_{\beta},k}\le c_1\left(\sum_{i=k}^{k+3}\|\bar u\|_{\M_{\beta},i}\right)\xi_{3,\beta}(\e)\le c_2(k)\xi_{3,\beta}(\e).
\end{split}
\end{equation}
Here, as explained before, the dominated estimate of $F_1^\e(x,y)$ is from the term
$\Theta_\e^2 \bar u(x,y)$, which vanishes in the estimates of $\bar F_1^\e(x)$.

Furthermore, according to the proof of \eqref{l2-2-3}, we know that it still holds for the critical regime.
Hence, applying the arguments for \eqref{p2-1-5}, \eqref{p2-1-5a}
and \eqref{p2-1-3a},
 and using \eqref{l3-4-1}, \eqref{l2-2-3} and Lemma \ref{l3-1},
we can obtain
for every $n\ge 1$
\begin{align*}
\sup_{y\in \T^d}\|F_{n+1}^\e(\cdot,y)\|_{\M_{\beta},k}
&\le c_3(k,n)\sum_{i=k}^{k+3}\left(\|\bar F_n^\e\|_{\M_{\beta},i}+\|\bar G_n^\e\|_{\M_{\beta},i}\right),\\
\|\bar F_{n+1}^\e\|_{\M_{\beta},k}
&\le
c_3(k,n)
\xi_{3,\beta}(\e)\sum_{i=k}^{k+3}\left(\|\bar F_n^\e\|_{\M_{\beta},i}+\|\bar G_n^\e\|_{\M_{\beta},i}\right),\\
\sup_{y\in \T^d}\|G_{n+1}^\e(\cdot,y)\|_{\M_{\beta},k}+\|\bar G_{n+1}^\e\|_{\M_{\beta},k}&\le c_3(k,n)\xi_{2,\beta}(\e)\sum_{i=k}^{k+2}\left(\sup_{y\in \T^d}\|G_n^\e(\cdot,y)\|_{\M_{\beta},i}
+\|\bar G_n^\e\|_{\M_{\beta},i}
+\|\bar F_{n}^\e\|_{\M_{\beta},i}\right).
\end{align*}
Using
these inductive inequalities and the initial estimate \eqref{e3-9}, we can show that
\begin{equation}\label{t1-2-3}
\begin{split}
 \sup_{y\in \T^d}\|G_{n}^\e(\cdot,y)\|_{\M_{\beta},k}+
 \|\bar G_{n}^\e\|_{\M_{\beta},k}&\le c_4(k,n)
\xi_{2,\beta}(\e)\left(\xi_{2,\beta}(\e)+\xi_{3,\beta}(\e)\right)^{n-1},\\
\sup_{y\in \T^d}\|F_{n}^\e(\cdot,y)\|_{\M_{\beta},k}&\le c_4(k,n)
\left(\xi_{2,\beta}(\e)+\xi_{3,\beta}(\e)\right)^{n-1},\\
 \|\bar F_{n}^\e\|_{\M_{\beta},k}&\le c_4(k,n)\xi_{3,\beta}(\e)\left(\xi_{2,\beta}(\e)+\xi_{3,\beta}(\e)\right)^{n-1}.
\end{split}
\end{equation}

Meanwhile, following the same arguments for \eqref{t1-1-3}, and using \eqref{e3-7} and \eqref{l2-2-1}
(similar to \eqref{l2-2-3},
\eqref{l2-2-1} still holds for the critical case), we can obtain that for every $n\ge 1$ and $x\in\R^d$,
$$
\lambda v_n^\e(x)-\sL_\e v_n^\e(x)=\bar F_{n}^\e(x)+G_{n}^\e\left(x,\e^{-1} x \right).
$$
This, along with  \eqref{e1-2a} and  \eqref{t1-2-3}, yields the desired conclusion \eqref{t1-2-2}.
\end{proof}

\subsection{Supercritical diffusive regime}

This part is devoted to the
\begin{proof}[Proof of Theorem $\ref{thm3}$]
Applying the fact
which is similar to \eqref{e2-8a}
that for all $k\ge0$, $x\in \R^d$, $y\in \T^d$ and $i=5,6,7$,
$$\nabla^k\Lambda_\e^i f(\cdot,y)(x)=\Lambda_\e^i\left(\nabla^k f\right)(x,y),\quad
\nabla^k\Theta_\e^3 f(\cdot,y)(x)=\Theta_\e^2\left(\nabla^k f\right)(x,y),$$
as well as the estimates \eqref{l4-1-1}, \eqref{l4-2-1}, \eqref{l4-3-1} and \eqref{l4-4-1}, we derive that
\begin{equation}\label{t1-3-4}
\begin{split}
&\sup_{y\in \T^d}\|F_1^\e(\cdot,y)\|_{\M_{\beta},k}\le c_1\left(\sum_{k=0}^3\|\bar u\|_{\M_\beta,k}\right)\left(1+\xi_{4,\beta}(\e)\right)\le c_2(k),\\
&\|\bar F_1^\e\|_{\M_{\beta},k}\le c_1\left(\sum_{k=0}^3\|\bar u\|_{\M_\beta,k}\right)\xi_{4,\beta}(\e)\le c_2(k)\xi_{4,\beta}(\e).
\end{split}
\end{equation}
Here as explained before, the dominated estimate of $F_1^\e(x,y)$  comes from the term
$\Theta_\e^3 \bar u(x,y)$, which varnishes when considering the estimate for $\bar F_1^\e(x,y)$
since
$\int_{\T^d}\Theta_\e^3 \bar u(x,y)\,dy=0$. Indeed, this fact is due to the following property
\begin{align*}
 \int_{\T^d}\bar A_0(y)\,dy
&=\int_{\T^d}\left(\int_{\R^d}\Big(z\otimes z+2 z\otimes \left(\phi_0(y+z)-\phi_0(y)\right)
\Big)K(y,y+z)j(z)\,dz\right)dy+2\int_{\T^d}\Phi_0(y)\otimes \phi_0(y)\,dy\\
&=\int_{\T^d}\left(\int_{\R^d}\Big(z\otimes z+2 z\otimes \left(\phi_0(y+z)-\phi_0(y)\right)
\Big)K(y,y+z)j(z)\,dz\right)\,dy-2\int_{\T^d}\sL \phi_0(y)\otimes \phi_0(y)\,dy\\
&=\int_{\T^d}\left(\int_{\R^d}\left(z+\phi_0(y+z)-\phi_0(y)\right)\otimes \left(z+\phi_0(y+z)-\phi_0(y)\right)K(y,y+z)
j(z)\,dz\right)\,dy=\bar A_0,
\end{align*}
where the second equality is due to \eqref{e2-2}, and the third one follows from
\begin{align*}
-2\int_{\T^d}\sL \phi_0(y)\otimes \phi_0(y)\,dy=
\int_{\T^d}\int_{\R^d}(\phi_0(y+z)-\phi_0(y))\otimes(\phi_0(y+z)-\phi_0(y))
K(y,y+z)\,dz\,dy
\end{align*}

Similar to the critical case,
\eqref{l2-2-3} still holds true in the present setting.
Hence, repeating the procedure of the proofs for \eqref{p2-1-5}, \eqref{p2-1-5a}
and \eqref{p2-1-3a}, and using \eqref{l4-4-1}, \eqref{l2-2-3} (by the proof of \eqref{l2-2-3}, we know it still holds for diffusive regime)
 and Lemma \ref{l3-1} (according to Remark \ref{r3-1a}, the conclusion of Lemma \ref{l3-1} still holds for diffusive regime),
 we can get that
for every $n\ge 1$
\begin{align*}
\sup_{y\in \T^d}\|F_{n+1}^\e(\cdot,y)\|_{\M_{\beta},k}
&\le c_3(k,n)\sum_{i=k}^{k+3}\left(\|\bar F_n^\e\|_{\M_{\beta},i}+\|\bar G_n^\e\|_{\M_{\beta},i}\right),\\
\|\bar F_{n+1}^\e\|_{\M_{\beta},k}
&\le c_3(k,n)\xi_{4,\beta}(\e)\sum_{i=k}^{k+3}\left(\|\bar F_n^\e\|_{\M_{\beta},i}+\|\bar G_n^\e\|_{\M_{\beta},i}\right),\\
 \sup_{y\in \T^d}\|G_{n+1}^\e(\cdot,y)\|_{\M_{\beta},k}+\|\bar G_{n+1}^\e\|_{\M_{\beta},k}
&\le c_3(k,n)\xi_{2,\beta}(\e)\sum_{i=k}^{k+2}\left(\sup_{y\in \T^d}\|G_n^\e(\cdot,y)\|_{\M_{\beta},i}
+\|\bar G_n^\e\|_{\M_{\beta},i}
+\|\bar F_{n}^\e\|_{\M_{\beta},i}\right).
\end{align*}
By these inductive inequalities and the initial estimate \eqref{t1-3-4}, we can obtain
\begin{equation}\label{t1-3-6}
\begin{split}
\sup_{y\in \T^d}\|G_{n}^\e(\cdot,y)\|_{\M_{\beta},k}+
\|\bar G_{n}^\e\|_{\M_{\beta},k}&\le c_4(k,n)
\xi_{2,\beta}(\e)\left(\xi_{2,\beta}(\e)+\xi_{4,\beta}(\e)\right)^{n-1},\\
\sup_{y\in \T^d}\|F_{n}^\e(\cdot,y)\|_{\M_{\beta},k}&\le c_4(k,n)
\left(\xi_{2,\beta}(\e)+\xi_{4,\beta}(\e)\right)^{n-1},\\
\|\bar F_{n}^\e\|_{\M_{\beta},k}&\le c_4(k,n)\xi_{4,\beta}(\e)\left(\xi_{2,\beta}(\e)+\xi_{4,\beta}(\e)\right)^{n-1}.\\
\end{split}
\end{equation}

Furthermore, for every $f\in \M_\beta$, it holds that
\begin{equation}\label{t1-3-7}
\begin{split}
& \sL_\e f(x)+\e\sL_\e \left\langle (\nabla f)(\cdot),\phi_0\left(
\e^{-1} \cdot
\right)\right\rangle(x)\\
&= \e^{-2}\int_{\R^d}\hat \delta f(x;\e z)K\left(\e^{-1} x ,\e^{-1} x +z\right)j \left(z\right)\,dz+
\e^{-1}\left\langle \nabla f(x), \Phi_0\left(\e^{-1} x \right)\right\rangle\\
&\quad+\left\langle \hbox{p.v.}
\int_{\R^d}\left(\nabla f(x+\e z)-\nabla f(x)\right)K\left(\e^{-1} x ,\e^{-1} x +z\right)j(z)\,dz, \phi_0\left(\e^{-1} x \right)\right\rangle+
\e \left\langle (\nabla f)(x),\sL_\e\phi_0\left(
\e^{-1} \cdot
\right)(x)\right\rangle \\
&\quad +
\e\sum_{i=1}^d \Gamma_\e\left(\partial_{x_i}f,\phi_{0,i}\right)\left(x,\e^{-1} x \right)\\
&=\bar \sL_{>2} f(x)+\sum_{i=5}^7\Lambda_\e^i f\left(x,\e^{-1} x \right)+\Theta_\e^3 f\left(x,\e^{-1} x \right),
\end{split}
\end{equation}
where in the last equality $\Lambda_\e^5 f$, $\Lambda_\e^6 f$, $\Lambda_\e^7 f$ and $\Theta_\e^3 f$ are defined in \eqref{e2-3-4}, $\Phi_0$ is defined in \eqref{e2-1a}, and we used the fact that $$\sL_\e\phi_0\left(
\e^{-1} \cdot
\right)(x)= -\e^{-2}\Phi\left(\e^{-1} x \right).$$
Hence, following the proof of \eqref{t1-1-3}, using \eqref{l2-2-1} and
\eqref{t1-3-7} (noting that \eqref{l2-2-1} holds for the supercritical diffusive regime),
we obtain that for every $n\ge 1$,
\begin{equation*}\label{t1-2-4}
\lambda v_n^\e(x)-\sL_\e v_n^\e(x)=\bar F_{n}^\e(x)+G_{n}^\e\left(x,\e^{-1} x \right),\ \ x\in \R^d.
\end{equation*}
This, along with \eqref{e1-2a} and \eqref{t1-3-6}, yields the desired conclusion \eqref{t1-3-2}.
\end{proof}

\subsection{Proofs of Examples}

\begin{proof}[Proof of Example $\ref{ex1-1}$]
(1) We firstly suppose that
$j(z)=\frac{1}{|z|^{d+\alpha}}$ for all $z\in \R^d$ with $\alpha\in (0,2)$. Then, Assumption \ref{a1-2} holds
when $\alpha\in (1,2)$ and Assumption (H) in Remark \ref{r1-2} holds when $\alpha\in (0,1]$.
In this case, take $\varphi(\e)=\e^{-\alpha}$ in \eqref{a1-1-1}. We can see that $\Pi_\e(z)\equiv 0$ with $\Pi_\e$ being defined by \eqref{e1-7a},
which implies
$\gamma_{1,\beta}(\e)\equiv 0$ for every $\beta>0$.
When $\alpha\in (1,2)$, by choosing $\beta=\alpha$, we have
$
\gamma_2(\e)\le c_1\e^{2-\alpha}$, $\gamma_{3,\alpha}(\e)\le c_1\e$ and $\gamma_4(\e)\le c_1.
$
Thus,
$
\xi_{1,\alpha}(\e)\le c_2\e^{2-\alpha}$ and $\xi_{2,\alpha}(\e)\le c_2\e.$
By this and \eqref{p2-1-1}, we can show the desired assertion for $\alpha\in (1,2)$.
For $\alpha=1$, \eqref{r1-1-1} holds with $\theta>0$. We still take $\beta=\alpha$, and then we can get
$
\tilde \gamma_2(\e)+\gamma_2(\e)\le c_3\e|\log \e|$, $\gamma_{3,\alpha}(\e)\le c_3\e$ and $\gamma_4(\e)\le c_3|\log \e|.
$
Thus,
$
\xi_{1,\alpha}(\e)\le c_4\e|\log \e|$ and $\tilde \xi_{2,\alpha}(\e)\le c_4\e|\log \e|,
$
which along with \eqref{r1-2-1} implies the desired assertion for $\alpha=1$.
For $\alpha\in (0,1)$, \eqref{r1-1-1} holds with $\theta=0$. Taking $\beta=\alpha$, it holds that
$
\tilde \gamma_2(\e)+\gamma_2(\e)\le c_5\e,$ $\gamma_{3,\alpha}(\e)\le c_5\e$ and $\gamma_4(\e)\le c_5\e^{\alpha-1}.
$
Then,
$
\xi_{1,\alpha}(\e)\le c_6\e^\alpha$ and $\tilde \xi_{2,\alpha}(\e)\le c_6\e^\alpha.
$
This along with \eqref{r1-2-1} yields
the desired assertion.

Next, we consider the case that $j(z)=\frac{1}{|z|^{d+\alpha_0}}\I_{\{|z|\le 2\}}+\frac{1}{|z|^{d+\alpha}}\I_{\{|z|> 2\}}$.
Then, by
direct computations,
\begin{align*}
\Pi_\e(z)\le \left(\e^{\alpha_0-\alpha}\frac{1}{|z|^{d+\alpha_0}}+\frac{1}{|z|^{d+\alpha}}\right)\I_{\{|z|\le 2\e\}}.
\end{align*}
This implies that $\gamma_{1,\alpha}(\e)\le c_7\e^{2-\alpha}$.
By the above argument, we see
that the same estimates hold for
$\gamma_2(\e)$-$\gamma_4(\e)$, so
\eqref{ex1-1-1b} is still true.

(2) This is a special case of Example $\ref{ex1-2}$(2) (namely when $m=0$), so the
proof is given below.

(3) With the form of $j(z)$ given here, one can see that, for any $\beta\in (0,2]$,
$$
\gamma_2(\e)\le c_7,\,\,\gamma_{3,\beta}(\e)\le c_7\e^{\alpha-1},\,\, \gamma_4(\e)\le c_7,\,\, \gamma_9(\e)\le c_7\max\{\e,\e^{\alpha-2}\}.
$$
This together with \eqref{t1-3-2} yields the desired assertion.
\end{proof}

\begin{proof}[Proof of Example $\ref{ex1-2}$]
(1) We will only consider the case that $
j(z)= \frac{(\log |z|)^m}{|z|^{d+\alpha}}
$ for all $z\in \R^d$ with $|z|\ge2$ and $m\ge0$, since the other one can be treated similarly.
It is easy to see that \eqref{a1-1-1} holds with $\varphi(\e)=\e^{-\alpha}|\log \e|^{-m}$. Then,
\begin{align*}
\Pi_\e(z)&\le \left(\e^{\alpha_0-\alpha}|\log \e|^{-m}\frac{1}{|z|^{d+\alpha_0}}+\frac{1}{|z|^{d+\alpha}}\right)\I_{\{|z|\le 2\e\}}
+\frac{1}{|z|^{d+\alpha}}\left|\left(\frac{\log |z|}{|\log \e|}+1\right)^m-1\right|\I_{\{|z|>2\e\}}\\
&\le \left(\e^{\alpha_0-\alpha}|\log \e|^{-m}\frac{1}{|z|^{d+\alpha_0}}+\frac{1}{|z|^{d+\alpha}}\right)\I_{\{|z|\le 2\e\}}
+\frac{c_1m }{|z|^{d+\alpha}}\frac{|\log |z||}{|\log \e|}\I_{\{2\e<|z|\le 2\e^{-1}\}}\\
&\quad +\frac{c_1}{|z|^{d+\alpha}}\left(\frac{  \log |z| }{|\log \e|}\right)^m \I_{\{|z|\ge 2\e^{-1}\}},
\end{align*} where $c_1>0$ is independent of $m$.
This yields that for every $0<\beta<\alpha$, there is a constant $c_2>0$ such that for all $m\ge0$,
$ \gamma_{1,\beta}(\e)\le c_2 m|\log \e|^{-1}.
$
On the other hand, we can get that
\begin{align*}
&\gamma_2(\e)\le c_3(m)\left(\e^{2-\alpha}(\log \e^{-1})^{-m}\I_{\{\alpha\in (1,2)\}}+\e|\log \e|\I_{\{\alpha=1\}}+\e^\alpha\I_{\{\alpha\in (0,1)\}}\right),\\
&\gamma_{3,\beta}(\e)\le c_3(m)\e,\,\, \gamma_4(\e)\le c_3(m)\left(1+|\log \e|^{m+1}\I_{\{\alpha=1\}}+\e^{\alpha-1}|\log \e|^m\I_{\{\alpha\in (0,1)\}}\right).
\end{align*}
Here we should note that $\lim_{m\to 0}c_3(m)>0$.
Combining the above estimates
with \eqref{p2-1-1} yields the desired assertion \eqref{ex1-1-2}.

(2) We first prove the case $m\in (-1,\infty)$.
Given
$j(z)$ therein, it holds that
\begin{align*}
\int_{\{2<|z|\le \e^{-1}\}}z_iz_jj(z)\,dz&=
c_1\int_2^{\e^{-1}}\int_{\mathbb{S}^{d-1}}\frac{(\log r)^m r^2 \theta_i \theta_j}{r^{d+2}}r^{d-1}\,\sigma(d\theta)\,dr\\
&=c_1\left(\int_{\mathbb{S}^{d-1}}\theta_i \theta_j\,\sigma(d\theta)\right) \int_{2}^{\e^{-1}}\frac{(\log r)^m}{r}\,dr\\
&=c_2(m)\left(\int_{\mathbb{S}^{d-1}}\theta_i \theta_j\,\sigma(d\theta)\right)\left(|\log \e|^{m+1}-(\log 2)^{m+1}
\right),
\end{align*}
where $d\theta \,dr$ denotes the spherical coordinate of $\R^d$, and $\sigma$ is the rotationally symmetric
Lebesgue measure on the sphere $\mathbb{S}^{d-1}$.
Then, letting
$
A=\{a_{ij}\}_{1\le i,j\le d}$ with $a_{ij}:=\int_{\mathbb{S}^{d-1}}\theta_i \theta_j\,\sigma(d\theta)$ for all $1\le i,j\le d$,
we have
$ \lim_{\e \to 0}|\log \e|^{-(m+1)}\int_{\{|z|\le \e^{-1}\}}(z\otimes z)j(z)\,dz=c_2(m)A, $
which implies that \eqref{a1-1-3} holds with $\varphi(\e)=\e^{-2}|\log \e|^{-(m+1)}$ and
$\gamma_7(\e)\le c_3(m)|\log \e|^{-(m+1)}.$ Indeed, we shall note that $a_{ij}=0$ for every $i\neq j$ due to the rotationally symmetry of $\sigma$; that is, $A$ is a positive definite diagonal $d\times d$ matrix. Here we should note that $\lim_{m\to -1}c_2(m)=\lim_{m\to -1}c_3(m)=\infty$.

Furthermore, it is easy to see that \eqref{a2-1-1} holds with $a_0(z)=\frac{(\log |z|)^m}{|z|^{d+3}}$. With $\varphi(\e)=\e^{-2}|\log \e|^{-(m+1)}$, we can verify that for every $\beta\in (0,2)$,
\begin{align*}
&\gamma_2(\e)\le c_4(m)|\log \e|^{-(m+1)},\,\, \gamma_{3,\beta}(\e)\le c_5(m)\e\max\{|\log \e|^{-(m+1)},\,\, |\log \e|^{-1}\},\ \gamma_4(\e)\le c_5(m),\\
&\gamma_5(\e)\le c_5(m)|\log \e|^{-(m+1)},\,\, \gamma_{6,\beta}(\e)\le c_5(m)\e\max\{|\log \e|^{-(m+1)},|\log \e|^{-1}\},\,\,\gamma_8(\e)\le c_5(m)|\log \e|^{-1}.
\end{align*}
Here the constant $c_4(m)$ depends on $m$ and also satisfies that
$\limsup_{m\downarrow -1}c_4(m)=\infty$.
Combining the above estimates
with \eqref{t1-2-2} yields the desired assertion.

We next prove the case $m=-1$.  For $j(z)$ given here, it holds that
\begin{align*}
\int_{\{2<|z|\le \e^{-1}\}}z_iz_jj(z)\,dz&=
c_6\int_2^{\e^{-1}}\int_{\mathbb{S}^{d-1}}\frac{r^2 \theta_i \theta_j}{r^{d+2}\log r}r^{d-1}\,\sigma(d\theta)\,dr\\
&=c_6\left(\int_{\mathbb{S}^{d-1}}\theta_i \theta_j\,\sigma(d\theta)\right)  \int_{2}^{\e^{-1}}\frac{1}{r\log r}\,dr\\
&=c_6
\left(\int_{\mathbb{S}^{d-1}}\theta_i \theta_j\,\sigma(d\theta)\right) \left(\log \log \e^{-1}-\log\log 2\right),
\end{align*}
and so
\begin{align*}
\lim_{\e \to 0}|\log\log \e^{-1}|^{-1}\int_{\{|z|\le \e^{-1}\}}(z\otimes z)j(z)\,dz
=c_6A,
\end{align*}
where $A$ is defined as
in the case $m\in (-1,\infty)$
that is a positive definite diagonal $d\times d$ matrix.
In particular,  \eqref{a1-1-3} holds with $\varphi(\e)=\e^{-2}|\log\log \e^{-1}|^{-1}$ and
$\gamma_7(\e)\le c_7|\log\log \e^{-1}|^{-1}.$

 Furthermore, it is easy to see that \eqref{a2-1-1} holds with $a_0(z)=\frac{1}{|z|^{d+3}(\log |z|)}$. By taking $\varphi(\e)=\e^{-2}|\log\log \e^{-1}|^{-1}$, we know that for every $\beta\in (0,2)$,
\begin{align*}
&\gamma_2(\e)\le c_8|\log\log \e^{-1}|^{-1},\,\, \gamma_{3,\beta}(\e)\le c_8\e|\log \e|^{-1}|\log\log \e^{-1}|^{-1},\,\, \gamma_4(\e)\le c_8,\\
&\gamma_5(\e)\le c_8|\log\log \e^{-1}|^{-1},\,\, \gamma_{6,\beta}(\e)\le c_8\e|\log \e|^{-1}|\log\log \e^{-1}|^{-1},\,\, \gamma_8(\e)\le c_8|\log \e|^{-1}|\log\log \e^{-1}|^{-1}.
\end{align*}
Putting all the estimates above into \eqref{t1-2-2}, we obtain the desired assertion.

(3)  We consider the case of $\alpha=2$ and $m<-1$ first.
With the expression  of $j(z)$, we have, for every $\beta\in (0,2)$,
$$
\gamma_2(\e)\le c_1(m),\,\, \gamma_{3,\beta}(\e)\le
c_1(m)
\e|\log \e|^{m},\,\,\gamma_4(\e)\le
c_1(m)
,\,\, \gamma_9(\e)\le c_2(m)|\log \e|^{m+1}.
$$
Here the constant $c_2(m)$ depends on $m$ and satisfies that
$\lim_{m\uparrow -1}c_2(m)=\infty$.
Combining these estimates with \eqref{t1-3-2} yields the desired assertion.

For the case that $\alpha>2$ and $m\in \R$, we can get by  direct calculations that
$$\xi_{2,\beta}(\e)\le c_1(m)\e,\quad \xi_{4,\beta}(\e)\le c_1(m)\max\{\e, \e^{\alpha-2}|\log \e|^m\}.$$ Then the desired assertion follows from \eqref{t1-3-2}.
\end{proof}

\begin{proof}[Proof of Example $\ref{ex1-3}$]
(1) In this case, \eqref{a1-1-1} holds with $\varphi(\e)=\e^{-\alpha}$, and
\begin{align*}
\Pi_\e(z)&\le \left(\e^{\alpha_0-\alpha}\frac{1}{|z|^{d+\alpha_0}}+\frac{1}{|z|^{d+\alpha}}\right)\I_{\{|z|\le 2\e\}}
+\left|\frac{1}{|z|^{d+\alpha}+\e^{\alpha-\alpha_1}|z|^{d+\alpha_1}}-\frac{1}{|z|^{d+\alpha}}\right|\I_{\{|z|>2\e\}}\\
&\le \left(\e^{\alpha_0-\alpha}\frac{1}{|z|^{d+\alpha_0}}+\frac{1}{|z|^{d+\alpha}}\right)\I_{\{|z|\le 2\e\}}
+c_1\e^{\alpha-\alpha_1}\frac{|z|^{\alpha_1}}{|z|^{ d+2\alpha}}\I_{\{|z|>2\e\}}.
\end{align*}
This implies that $\gamma_{1,\alpha}(\e)\le c_2\max\{\e^{2-\alpha},\e^{\alpha-\alpha_1}\}$.
On the other hand, one can get that
\begin{equation}\label{ex1-1-5}
\gamma_{2}(\e)\le c_3\xi_0(\e),\,\, \gamma_{3,\alpha}(\e)\le c_3\e,\,\, \gamma_4(\e)\le c_3\left(1+|\log \e|\I_{\{\alpha=1\}}+\e^{\alpha-1}\I_{\{\alpha\in (0,1)\}}\right).
\end{equation}
Putting all the estimates above together into \eqref{p2-1-1}, we
get the   desired assertion.

(2) In this case,  \eqref{a1-1-1} still holds with $\varphi(\e)=\e^{-\alpha}$. On the one hand, we have
\begin{align*}
\Pi_\e(z)&\le \left(\e^{\alpha_0-\alpha}\frac{1}{|z|^{d+\alpha_0}}+\frac{1}{|z|^{d+\alpha}}\right)\I_{\{|z|\le 2\e\}}
+\frac{\e^{\alpha_1-\alpha}}{|z|^{d+\alpha_1}}\I_{\{|z|>2\e\}}
\end{align*}
and so
$
\gamma_{1,\alpha}(\e)\le c_2\max\{\e^{2-\alpha},\e^{\alpha_1-\alpha}\}$.
One the other hand, we can see that \eqref{ex1-1-5} also holds true. Therefore, according to all the estimates above and \eqref{p2-1-1},
we get the  desired assertion.
\end{proof}

\section{Appendix} In this part, we will  prove the following statement.

\begin{proposition} Suppose that
there are constants $c_0>0$ and $\alpha_0\in (1,2)$ such that $j(z)=c_0 |z|^{-(d+\alpha_0)}$
for all $z\in \R^d$ with $|z|\le 1$. Then, Assumption $\ref{a1-2}$ holds. \end{proposition}

\begin{proof} We first note that the  operator $\sL$ given by \eqref{e1-1} is symmetric in the sense that
$$
\int_{\R^d}  \sL f (x) g (x) \, dx  = \int_{\R^d} f (x)\sL g (x)\, dx
\quad \hbox{ for every } f, g \in C_c^2(\R^d).
$$
In fact, for every $f, g \in C_c^2(\R^d)$,
$$
\sE (f, g):=
-  \int_{\R^d}  g (x)  \sL f (x)\,  dx
=\frac12 \int_{\R^d \times \R^d} (f(x)-f(y))(g(x)-g(y)) K(x, y) j(x-y)\, dx\, dy.
$$
For $\kappa >0$, define $\sE_\kappa  (f, g):= \sE(f, g) + \kappa\int_{\R^d} f(x)g(x) \,dx$ and
$\| f\|_{\sE_\kappa }:= \sE_\kappa (f, f)^{1/2}$ for any $f,g\in C_c^2(\R^d)$.
When $K(x, y)\equiv  K_0 $ is a positive constant, $\sL$ is the generator of a symmetric  L\'evy process $(X_t)_{t\ge0}$ on $\R^d$
having
L\'evy measure $\nu (dz) := K_0  j(z)\, dz$.
 In this case, it follows  by using the
Fourier transform that
$$
\sF := \overline{C_c^2(\R^d)}^{ \| \cdot \|_{\sE_1}}
= \Big\{ f\in L^2(\R^d; dx):  \int_{\R^d \times \R^d} (f(x)-f(y))^2
j(x-y)\, dx \,dy <\infty \Big\}.
$$
The above continues to hold for a general symmetric function $K(x, y)$ satisfying $0<K_1\le K(x,y)\le K_2<\infty$
because $\sE_1 (f, f)$ is comparable to that of  the constant case. This shows that  $(\sE, \sF)$ is a regular Dirichlet
form on $L^2(\R^d; dx)$.
It is well known that there is a symmetric Hunt process $X:=(X_t)_{t\ge0}$ on $\R^d \setminus \sN$ associated with the regular Dirichlet form
$(\sE, \sF)$ on $L^2(\R^d; dx)$, where $\sN\subset \R^d$ is a Borel properly exceptional set for the process $X$ which in particular is
$\sE$-polar
and hence has zero Lebesgue measure. We refer the reader
to \cite{CF, FOT}  for the terminology and these facts about regular Dirichlet forms.

Next, we assume that $j(z)=c_0 |z|^{-(d+\alpha_0)}$
for all $z\in \R^d$ with $|z|\le 1$, where $c_0>0$ and $\alpha_0\in (1,2)$. In this case, since $\int_{\{|z|>1\}}j(z)\,dz<\infty$ and $0<K_1\le K(x,y)=K(y,x)\le K_2<\infty$, $(\sE, \sF)$ can be seen as a bounded perturbation of
the symmetric
$\alpha_0$-stable-like Dirichlet form. Then,
according to \cite[Theorem 1.1]{CK03} and \cite[Theoem 4.12]{BSW}, we can see that the associated symmetric Hunt process $X:=(X_t)_{t\ge0}$
can be refined to be
 a Feller process that can start from every $x\in \R^d$ and enjoys the strong Feller property. Indeed, we can write the operator $\sL$ defined by \eqref{e1-1} as follows:
\begin{equation}\label{e:lev}\sL f(x)=\int_{\R^d}\left(f(x+z)-f(x)-\langle\nabla f(x), z\rangle\I_{\{|z|\le 1\}}\right)K(x,x+z)
j(z)
\,dz + \langle\nabla f(x), b(x)\rangle,\quad f\in C_b^2(\R^d)\end{equation}where
$$
b(x) =\frac12 \int_{\{|z|\le 1\}} z (K(x,x+z)-K(x,x-z))j(z)\,dz,
$$
 which is bounded thanks to $K\in C_b^1(\R^d\times \R^d)$ and $\int_{\{|z|\le 1\}}|z|^2j(z)\,dz<\infty$.
So the operator $\sL$ can be regarded as a L\'evy-type operator in \cite{BSW}. Furthermore, by \eqref{e:lev}, for any $f\in C_b^2(\R^d)$,
$$\sL f(x)=\sL_1f(x)+\sL_2f(x),$$ where
\begin{equation}\label{e:sss}\sL_1f(x)= \int_{\R^d}\left(f(x+z)-f(x)-\langle\nabla f(x), z\rangle\I_{\{|z|\le 1\}}\right)K(x,x+z)
c_0|z|^{-d-\alpha_0}
\,dz + \langle\nabla f(x), b(x)\rangle\end{equation} and
$$\sL_2f(x)=\int_{\{|z|>1\}}\left(f(x+z)-f(x)\right)K(x,x+z) (j(z)-
c_0|z|^{-d-\alpha_0})
\,dz.$$ It is easy to see that there is a constant $C_0>0$ such that for all $f\in B_b(\R^d)$,
$$\|\sL_2f\|_\infty\le C_0\|f\|_\infty.$$

Since $K(\cdot, \cdot)$ is multivariate $1$-periodic, we can regard $X$ as an $\T^d$-valued process, which will be written as $X^{\T^d}:=(X^{\T^d}_t)_{t\ge0}$.
We claim the process $X^{\T^d}$ has  the property that for any nonempty open set $U\subset \T^d$, $\Pp_x(X^{\T^d}_t\in U)>0$ for every $t>0$, $x\in \T^d$, and also
has the strong Feller property (i.e., for any $f\in B_b(\T^d)$ and $t>0$, the function $x\mapsto \Ee_x f(X^{\T^d}_t)$ is bounded and continuous). Indeed, according to \cite[Theorem 1.5]{CZ}, the operator $\sL_1$ defined by \eqref{e:sss} has a transition density kernel $p_1(t,x,y)$ with respect to the Lebesgue measure such that for any $t>0$, $(x,y)\mapsto p_1(t,x,y)$ is continuous and strictly positive, and for any $t>0$ and $y\in \R^d$, the function $x\mapsto p_1(t,x,y)$ is differentiable; moreover, $C_b^2(\R^d)\subset \D(\sL_1)$. Hence, by Duhamel's formula for bounded perturbation of the Markov semigroup and the condition that $\alpha_0\in (1,2)$, one can see that the operator $\sL$ has a transition density kernel $p(t,x,y)$ which enjoys all the properties as mentioned above for $p_1(t,x,y)$.
In particular,
by \cite[Theorem 1.5(vi)]{CZ}, there is a constant $C_1>0$ such that for all $t>0$ and $f\in B_b(\R^d)$,
\begin{equation}\label{e:der}
\sup_{x\in \R^d} |\nabla_x \Ee_x f(X_t)|\le C_1 (t\wedge1)^{-1/\alpha_0}\|f\|_\infty.
\end{equation}
Let $\pi$ denote the canonical projection from $\R^d$ to $\T^d$. Then, for all $t\ge0$,
$$X_t^{\T^d}=\pi(X_t).$$ Then the process $(X_t^{\T^d})_{t\ge0}$ has the transition density given by
\begin{equation} \label{e:5.4}
p_{\T^d}(t,x,y)=\sum_{y'\in \pi^{-1}(y)} p(t,x,y'),\quad t>0, x,y\in \T^d.
\end{equation}
 So the process $X^{\T^d}$ is irreducible, and has the strong Feller property.
 Moreover, it follows from  \eqref{e:der} and \eqref{e:5.4} that
 for all $t>0$ and $f\in B_b(\T^d)$,
\begin{equation}\label{e:der2}
\sup_{x\in \T^d} |\nabla_x \Ee_x f(X^{\T^d}_t)|\le C_2 (t\wedge1)^{-1/\alpha_0}\|f\|_\infty.
\end{equation}
Therefore, according to \cite[Theorem 6.1]{MT} and \cite[Theorem 5.1]{T} (or \cite[Proposition 1.1]{SVW}), there are constants $\lambda_0, C_3>0$ such that for all $t>0$,
\begin{equation}\label{e:der3}\sup_{x\in \T^d} \|\Pp_x^{\T^d}(X^{\T^d}_t\in dz)-{\rm Leb}(dz)\|_{{\rm var}}\le C_3e^{-\lambda_0t},\end{equation} where $\Pp_x^{\T^d}(X^{\T^d}_t\in dz)$
denotes the distribution of $X_t^{\T^d}$ when the process $X^{\T^d}$ starts from $x\in \T^d$, ${\rm Leb}(dz)$ is the Lebesgue measure on $\T^d$ (which is a probability measure), and $\|\cdot\|_{{\rm var}}$
stands for
the total variation norm on the space of signed measures on $\T^d$. Here we used the fact that $(\sL, C_c^2(\R^d))$ is symmetric on $L^2(\R^d;dx)$, and so ${\rm Leb}(dz)$ is an invariant probability measure for the process $X^{\T^d}$.

By \eqref{e:der3}, for any $f\in C(\T^d)$ with $\int_{\T^d} f(y)\,dy=0$,
the function $x\mapsto \phi_f(x):=-\int_0^\infty \Ee_x f(X^{\T^d}_t)\,dt$ is pointwise well defined in the sense that
\begin{equation}\label{e:co1}\sup_{x\in \T^d} |\phi_f(x)|\le \sup_{x\in \T^d} \int_0^\infty| \Ee_x f(X^{\T^d}_t)|\,dt\le C_3\|f\|_\infty \int_0^\infty e^{-\lambda_0 t}\,dt=
C_3
\|f\|_\infty/\lambda_0.\end{equation} Moreover, $\phi_{f}\in C^1(\T^d)$, and, according to
\eqref{e:der2},
\begin{equation}\label{e:co2}
\sup_{x\in \T^d} |\nabla\phi_f(x)|\le \sup_{x\in \T^d} \int_0^\infty| \nabla \Ee_x f(X^{\T^d}_t)|\,dt\le C_2\|f\|_\infty \int_0^\infty \frac{1}{(t\wedge1)^{1/\alpha_0}}e^{-\lambda_0 t}\,dt.\end{equation}
These estimates along with (the proof of) \cite[Theorem 1.4]{SVW} imply that the equation
$\sL^{\T^d} \phi_f=f$ has a unique pointwise solution $\phi_{f}\in C^1(\T^d)\cap \D(\sL^{\T^d})$, where $\sL^{\T^d}$ is the infinitesimal generator of the process $X^{\T^d}$. Furthermore, by \eqref{e:co1} and \eqref{e:co2}, there is a constant $C_4>0$ such that for all $f\in C(\T^d)$, $\|\phi_{f}\|_\infty+\|\nabla \phi_{f}\|_\infty\le C_4\|f\|_\infty.$ Then, the desired assertion follows by transforming the conclusions above into the process $X$.
\end{proof}

\bigskip

\noindent {\bf Acknowledgements.}\,\,  The research of Xin Chen is supported by the National Natural Science Foundation of China
(No.\ 12122111). The research of Zhen-Qing Chen is partially supported by a Simons Foundation fund.
The research of Takashi Kumagai is supported by JSPS
KAKENHI Grant Numbers
22H00099, 23KK0050 and JST Moonshot R\&D Grant Number JPMJMS25A5.
The research of Jian Wang is supported by the National Key R\&D Program of China (2022YFA1006003) and the National Natural Science Foundation of China (Nos. 12071076 and 12225104).

\end{document}